\documentclass[11pt,reqno]{amsart}
\usepackage[utf8]{inputenc}
\UseRawInputEncoding
\usepackage{appendix}
\usepackage[percent]{overpic}

\newcommand\WriteupNote[1]{}

\usepackage{comment}
\usepackage{caption}
\usepackage{subcaption}
\usepackage{graphicx}
\usepackage{enumerate,enumitem}
\usepackage{color}
\usepackage{mathtools}
\usepackage{marginnote}
\mathtoolsset{showonlyrefs}

\usepackage{fancyhdr}
\usepackage{epsfig,amsfonts,latexsym,amsmath,amssymb}
\usepackage{amsthm}

\theoremstyle{plain}
 \newtheorem{theorem}{Theorem}[section]
 \newtheorem{lemma}[theorem]{Lemma}
 \newtheorem{corollary}[theorem]{Corollary}
 \newtheorem{proposition}[theorem]{Proposition}

 \theoremstyle{definition}
 \newtheorem{definition}[theorem]{Definition}
 \theoremstyle{definition}
  \newtheorem{remark}[theorem]{Remark}

\numberwithin{equation}{section}

\newcommand{\N}{\mathbb N}

\newcommand{\QQ}{\mathbb Q}

\newcommand{\R}{\mathbb R}
\newcommand{\der}{\mathrm{d}}
\newcommand{\eps}{\varepsilon}
\renewcommand{\phi}{\varphi}
\newcommand{\abs}[1]{\left\lvert #1 \right\rvert}

\newcommand{\Order}{\mathcal O}
\newcommand{\Der}[1]{\frac{\der}{\der #1}}
\DeclareMathOperator{\sgn}{sgn}

\newcommand{\lsp}{\operatorname{lsp}}
\newcommand{\blsp}{\operatorname{blsp}}
\newcommand{\spec}{\operatorname{spec}}
\newcommand{\Tr}{\operatorname{Tr}}

\newcommand{\dd}{\operatorname{d} \!}
\newcommand{\ii}{\operatorname{i}}
\newcommand{\RR}{\mathbb{R}}

\renewcommand{\Re}{\operatorname{Re}}

\newcommand{\mstrut}[1]{\mbox{\rule{0mm}{#1}}}
\newcommand{\ba}{\begin{array}}
\newcommand{\ea}{\end{array}}

\newcommand{\col}[1]{\begin{bmatrix}#1\end{bmatrix}}

\newcommand{\bra}[1]{\left[#1\right]} % encloses the argument using stretchable square brackets
\newcommand{\mat}[1]{\begin{matrix}#1\end{matrix}} % no delimiters
\newcommand{\bmat}[1]{\bra{\mat{#1}}} % square brackets as delimiters

\def\beq{\begin{equation} } \def\eeq{\end{equation}}

\def\RR{\mathbb R}

\def\eps{\varepsilon}  

\def\ben{\begin{enumerate} }

\def\een{\end{enumerate} }

\def \R{ {\mathbb R}}
  
\def \p { \partial}

\def \surf{ 1}  \def \rturn{ R^\star}
\def \CMB{ R} \def \wkb{\beta}
\def \refl{b}
\def \refll{\mathbf{f}}

\def \tConst{ q_0}

\def \simp{{basic}}
\def \branch{{leg}}
\newcommand{\joonas}[1]{}
\newcommand{\vitaly}[1]{}
\newcommand{\maarten}[1]{}

\def \pg{p_\gamma}
\usepackage{amsaddr}

\title[Spectral rigidity with discontinuities]{Spherically symmetric terrestrial planets with discontinuities are spectrally rigid}

\author{Maarten V. de Hoop}
\address{Simons Chair in Computational and Applied Mathematics and Earth Science\\
Rice University\\
6100 Main MS-134, Houston, TX 77005-1892, USA\\
\texttt{mvdehoop@rice.edu}}
\author{Joonas Ilmavirta}
\address{Department of Mathematics and Statistics\\
University of Jyv\"askyl\"a\\
P.O. Box 35 (MaD), FI-40014 University of Jyv\"askyl\"a, Finland\\
\texttt{joonas.ilmavirta@jyu.fi}}
\author{Vitaly Katsnelson}
\address{College of Arts and Sciences\\
New York Institute of Technology\\
New York NY, USA\\
\texttt{vkatsnel@nyit.edu}}

\date{\today}

\begin{document}

\maketitle

\begin{abstract}
We establish spectral rigidity for spherically symmetric manifolds with boundary and interior interfaces determined by discontinuities in the metric under certain conditions. Rather than a single metric, we allow two distinct metrics in between the interfaces enabling the consideration of two wave types, like  \textit{P}- and \textit{S}-polarized waves in isotropic elastic solids. Terrestrial planets in our solar system are approximately spherically symmetric and support toroidal and spheroidal modes. Discontinuities typically correspond with phase transitions in their interiors. Our rigidity result applies to such planets as we ensure that our conditions are satisfied in generally accepted models in the presence of a fluid outer core. The proof is based on a novel trace formula. We also prove that the length spectrum of the Euclidean ball is simple.
\end{abstract}

%\keywords{Inverse problems, spectral rigidity, planets, seismology}

%\subjclass[2010]{...}

\subsection*{Acknowledgements
} MVdH was supported by the Simons Foundation under the MATH + X program, the National Science Foundation under grant DMS-1815143, and the corporate members of the Geo-Mathematical Imaging Group at Rice University.
JI was supported by the Academy of Finland (projects 332890 and 336254).
We thank Chunquan Yu for help with composing figure~\ref{fig:basic}. We also thank Gabriel Paternain and Anthony V\'{a}rilly-Alvarado for some helpful discussions. The authors greatly appreciate the detailed suggestions made by an anonymous referee which
improved this paper.
\section{Introduction}

We establish spectral rigidity for spherically symmetric manifolds with boundary and interfaces determined by discontinuities in the metric. We study the recovery of a (radially symmetric Riemannian) metric or wave speed containing jump discontinuities along finitely many $C^\infty$ hypersurfaces. To our knowledge, it is the first such result pertaining to a manifold with boundary and a piecewise continuous metric.

Terrestrial planets in our solar system are approximately spherically symmetric. On the one hand, the deviation from such a symmetry becomes apparent only at high eigenfrequencies. On the other hand, our results provide a stable approximation upon truncating the spectrum of eigenfrequencies. Discontinuities arise largely due to phase transitions. Hence, their radial depths play an important role in determining the thermal structure and chemical composition of planets as well as the dynamics of their interiors \cite{SchubertPhaseTransition}. The question of spectral rigidity is behind the validity of PREM \cite{DZIEWONSKIPREM} which is still widely used as a reference in linearized tomography. More interestingly, in space exploration such as the current NASA's InSight mission to Mars \cite{lognonne:hal-02526740}, with a single data point, spectral data could provide the leading information about its interior; other missions are being proposed.

The results presented, here, are an extension of our previous result \cite{HIKRigidity} where we proved a spectral rigidity for a smooth metric on a radial manifold. Allowing for certain discontinuities in the metric adds a new level of challenge for several reasons. First, the geodesics in such a manifold get reflected and transmitted when they hit an interface, creating a complex geometry for the analysis. In addition, we allow such geodesics to hit an interface at certain critical angles where a scattered ray can intersect an interface tangentially or ``glide'' along an interface. We also recover the location of the interfaces and do not assume that they are known.

We require the so-called Herglotz condition while allowing an unsigned curvature; that is,
curvature can be everywhere positive or it can change sign, and we
allow for conjugate points. Spherically symmetric manifolds with
boundary are models for planets, the preliminary reference Earth model
(PREM) being the prime example. Specifically, restricting to toroidal
modes, our spectral rigidity result determines the shear wave speed of
Earth's mantle in the rigidity sense.

The method of proof relies on a trace formula, relating the spectrum
of the manifold with boundary to its length spectrum, and the
injectivity of the periodic broken ray transform. Specifically, our
manifold is the Euclidean ball $M = \bar B(0,1)\subset \RR^3$, with the metric $g(x) =
c^{-2}(\abs{x}) e(x)$, where $e$ is the standard Euclidean metric and
$c \colon (0,1]_r \to (0,\infty)$ is a function satisfying suitable
conditions, where $r = |x|$ is the radial coordinate.
We work in dimension three but our result on length spectral rigidity (Theorem \ref{thm:blspr-multilayer}) carries over to higher dimensions, and our methods to prove spectral rigidity (Theorem \ref{t: spectral rigiditiy}) may be generalized to higher dimensions.

We assume $c(r)$ has a jump discontinuity at a finite set of values $r = r_1, \dots, r_K$; that is $\lim_{r \to r_i^-}c(r) \neq \lim_{r \to r_i^+}c(r)$ for each $i$.
Our assumption is the \emph{smooth Herglotz
condition}: $\Der{r}(r/c(r))>0$ is satisfied everywhere away from the discontinuities of $c$, but we note that $c$ is allowed to either increase or decrease across an interface.
We note that the natural extension of the Herglotz condition when $c$ is smooth to our case when $c$ has discontinuities is to view $c$ as a distribution and require $\Der{r}(r/c(r))>0$ in the distributional sense. If $c$ has a jump discontinuity at $r = r_i$, this distributional condition implies $\lim_{r \to r_i^-}c(r) > \lim_{r \to r_i^+}c(r)$. This would be too restrictive since radial models of Earth (PREM) and Mars (T13) (see \cite{KhanMars}) satisfy the smooth Herglotz condition but not this stronger distributional Herglotz condition, since the jump across the core-mantle boundary differs in sign to the jumps at other interfaces. Hence, our smooth Herglotz condition is weaker to allow the jump across interfaces to have any sign.
We  also allow trapped rays that never interact with the boundary. Such rays just  correspond  to  small  but nonzero boundary amplitudes of modes. The assumption $\Der{r}(r/c(r))>0$ when $c$ is smooth is the \emph{Herglotz condition} first discovered by Herglotz~\cite{H:kinematic} and
used by Wiechert and Zoeppritz~\cite{WZ:kinematic}.

By a maximal geodesic we mean a unit speed geodesic on the Riemannian
manifold $(M,g)$ with each endpoint at the boundary $\partial M$ or an interface. A broken ray or a billiard trajectory is a
concatenation of maximal geodesics satisfying the reflection condition
of geometrical optics at both inner and outer boundaries of $M$, and Snell's law for geometric optics at the interfaces. If the initial and final points
of a broken ray coincide at the boundary or an interface, we call it a periodic broken ray -- in
general, we would have to require the reflection condition at the
endpoints as well, but in the assumed spherical symmetry it is
automatic. We will describe later (Definition \ref{d: ccc}) what will
be called the \emph{countable conjugacy condition} which ensures
that up to rotation only countably many maximal geodesics have conjugate endpoints.
\vitaly{The above paragraph needs checking}
The length spectrum of a manifold $M$ with boundary is the set of
lengths of all periodic broken rays on $M$. If the
radial sound speed is $c$, we denote the length spectrum by $\lsp(c)$. 
We will introduce in Definition \ref{simple closed ray} the notion of closed \emph{\simp{} rays}, which are certain periodic rays that stay completely within a single layer. The set of lengths of such rays form the \simp{} length spectrum $\blsp(c)$.
We note that every broken ray is contained in a
unique two-dimensional plane in $\RR^n$ due to symmetry
considerations. Therefore, it will suffice to consider the case $n=2$;
the results regarding geodesics and the length spectrum carry over to
higher dimensions. We denote the Neumann spectrum of the Laplace--Beltrami operator in
three dimensions, $\Delta_c = c^3 \nabla \cdot c^{-1} \nabla$, on $M$
by $\spec(c)$, where we impose Neumann-type boundary
conditions on both the inner and outer boundary.
The spectrum $\spec(c)$ includes multiplicity, not just the set spectrum.

Some earlier results in tensor tomography, the methods of which are
related to ours, may be found in
\cite{Anasov14,Beurling15,Sharaf97,UhlSharaf}. Let us now enumerate the various geometric assumption we make in this manuscript for easy reference.

\subsection{Herglotz and other conditions}

 There are several geometric assumptions we make that we shall enumerate here:
\begin{enumerate}
\item[(A1)] \label{A1}
``Periodic conjugacy condition.'' This is an analog of the clean intersection hypothesis used in \cite{Mel79,DG75,HIKRigidity}; 
(see Definition \ref{d: pcc}). 

\item [(A2)]
``Principal amplitude injectivity condition.'' This is an analog to assuming \emph{simplicity} of the length spectrum.  (see section \ref{s: geometric spreading injectivity condition}).
 %This is a condition to ensure that for each period, if there are two distinct geodesics with that period that are neither time reversals or rotations of one another, their principal contributions to the trace do not coincide.
\item [(A3)]
``Countable conjugacy condition'' (Definition \ref{d: ccc}).
\item [(A4)]
Smooth Herglotz condition: $\frac{d}{dr} \frac{r}{c(r)} > 0$ away from the discontinuities.
\end{enumerate}

These assumptions allow us to prove that the singular support of the wave trace includes the \simp{} length spectrum. Assumption (A1) is a standard assumption (normally referred to as the clean intersection hypothesis when $c$ is smooth) when calculating the trace singularity by a stationary phase method to ensure that the critical manifolds are non-degenerate and the phase function is Bott-Morse nondegenerate (see \cite{DG75, Mel79}). A ubiquitous issue in computing a trace formula is the possibility of cancellations between the contributions of two components of the same length that are not time reversals of each other to the wave trace. One usually assumes ``simplicity'' of the length spectrum so that any two rays with a given period are either rotations of each other or time reversals of each other, but since our trace formula computation is more explicit, we have a slightly weaker assumption (A2) to take care of this issue. Assumptions (A1), (A2), and (A4) are needed for the trace formula (Proposition \ref{prop:Trace Formula}), and all four assumptions are needed for spectral rigidity (Theorem \ref{t: spectral rigiditiy}), while only assumptions (A3) and (A4) are used to prove length spectral rigidity (Theorem \ref{thm:blspr-multilayer}).  Below, we provide a chart for easy reference regarding which assumptions are needed for each theorem:

\begin{table}[h]
\begin{center}
\begin{tabular}{lcccc}                 & (A1)                   & (A2)                   & (A3)                   & (A4)                   \\ \cline{2-5}
\multicolumn{1}{l|}{Trace formula}            & \multicolumn{1}{c|}{X} & \multicolumn{1}{c|}{X} & \multicolumn{1}{c|}{}  & \multicolumn{1}{c|}{X} \\ \cline{2-5}
\multicolumn{1}{l|}{Length spectral rigidity} & \multicolumn{1}{c|}{}  & \multicolumn{1}{c|}{}  & \multicolumn{1}{c|}{X} & \multicolumn{1}{c|}{X} \\ \cline{2-5}
\multicolumn{1}{l|}{Spectral rigidity}        & \multicolumn{1}{c|}{X} & \multicolumn{1}{c|}{X} & \multicolumn{1}{c|}{X} & \multicolumn{1}{c|}{X} \\ \cline{2-5}
\end{tabular}
\end{center}
\end{table}

{\it Genericity of the assumptions: }
A key novelty of the present paper is that the final proof is essentially a length spectral rigidity result.
On a general Riemannian manifold the very existence of closed orbits is a complicated matter \cite{Bangert1980, benci1992, gromoll1969}, and having a boundary makes the problem even harder.
In some geometries, such as negatively curved closed manifolds, periodic orbits are known to exist and indeed spectral rigidity has been proven on such manifolds~\cite{Croke98, Croke1990}.
In the case of spherical symmetry we only need a very mild condition to ensure that there are enough periodic orbits and they behave cleanly enough; this is called the countably conjugacy condition. Thus, we only require assumptions (A3) and (A4) to show length spectral rigidity.
We expect these condition to be generic on the class of manifolds we study, but verifying it is not simple check, so we leave it as a conjecture that assumptions (A1), (A2), and (A3) are generic among metrics that satisfy the Herglotz condition.

In Appendix \ref{app: example satisfying A1-A4}, we prove that the class of metrics satisfying all of our assumptions is nonempty and a genericity result. Specifically, we construct an uncountable family of metrics that satisfy assumptions (A.1), (A.3), and (A.4). For assumption (A.2), which is an analog of having a simple length spectrum in our setting and is more subtle to verify, we prove that within this class, assumption (A.3) is generically satisfied.

In Appendix \ref{a: ball case}, we prove that the Euclidean disk does have a simple length spectrum, thereby satisfying all of the assumptions and demonstrating that the set of metrics satisfying assumptions (A.1)-(A.4) is nonempty. The proof relies on advanced concepts from number theory, even in this geometrically simplest case. This suggests that finding examples in non-Euclidean geometry will be difficult.
\medskip\medskip \medskip \medskip

\noindent

\subsection{Main results}
\label{s: main results}
Here we present our main theorems, which follow a discussion of the notation we use for the geometry. Let $A(r', r'') = \bar B(0, r'') \setminus B(0, r') \subset \RR^3$ be the closed annulus in a Euclidean space where $r''>r'$. Fix $K \in \mathbb N$ and let $r_k \in (0, \surf)$ so that $\surf=:r_0 > r_1 > \cdots > r_K$.
Assume $c(r)$ has jump discontinuities at each $r_k \in (0, \surf)$. Let $\Gamma = \bigcup_k \{ r = r_k\}$ be the collection of interfaces together with $\p M$, and denote $\Gamma_k := \{ r = r_k\}$. We sometimes refer to the smooth annular regions $A(r_k, r_{k-1})$ as \emph{layers}. We view $M$ as a Riemannian manifold with (rough) metric $g = c^{-2} dx^2$.

\begin{definition}

Fix any $\eps>0$ and $K\in\N$.
We say that a collection of functions $c_\tau\colon[0,1]\to(0,\infty)$ indexed by $\tau\in(-\eps,\eps)$ is an admissible family of profiles if the following hold:
\begin{itemize}
\item
There are radii $r_k\in(0,1)$ that depend $C^1$-smoothly on $\tau\in(-\eps,\eps)$ so that $1\eqqcolon r_0(\tau)>r_1(\tau)>\dots>r_K(\tau)>0$ for all $\tau\in(-\eps,\eps)$.
\item
For every $\tau\in(-\eps,\eps)$ the function $c_\tau$ is piecewise $C^{1,1}$ and satisfies the smooth Herglotz condition.
\item
The only singular points of each function $c_\tau$ are the radii $r_k(\tau)$ where it has a jump discontinuity.
\item
Within each annulus $A(r_k(\tau),r_{k-1}(\tau))$ the profile $c_\tau$ satisfies the countably conjugacy condition for all $\tau\in(-\eps,\eps)$.
\item
We assume that $(r,\tau)\mapsto c_\tau(r)$ is $C^1$ at all points where $r\notin\{r_1(\tau),\dots,r_K(\tau)\}$.
\end{itemize}
\end{definition}
Recall from the introduction that the length spectrum of a manifold $M$ with boundary is the set of
lengths of all periodic broken rays on $M$ and we denote the length spectrum by $\lsp(c)$. We will introduce in Definition \ref{simple closed ray} the notion of closed \emph{\simp{} rays}, which are certain periodic rays that stay completely within a single layer. The set of lengths of such rays form the \simp{} length spectrum $\blsp(c)$.

Our main theorem provides the rigidity of the basic length spectrum in the presence of ``countable noise''.
Choosing the ``noise'' suitably gives corollaries for the full length spectrum.
Missing or spurious points in the length spectrum or some amount of degeneracy do not matter.
The ``noise'' can be of the same size as the data, and this will play a role in the case of multiple wave speeds.

\begin{theorem}
\label{thm:blspr-multilayer}
Fix any $\eps>0$ and $K\in\N$, and let $c_\tau(r)$ be an admissible family of profiles with discontinuities at $r_k(\tau)$ for all $k=1,\dots,K$.

Let $\blsp(\tau)$ denote the basic length spectrum of 
%the annulus $A(1,r_K(\tau))$
the ball $\bar B(0,1)$
with the velocity profile $c_\tau$.
Suppose $\blsp(\tau)$ is countable for all $\tau$.
Let $S(\tau)$ be any collection of countable subsets of $\RR$ indexed by $\tau$.
If $\blsp(\tau)\cup S(\tau)=\blsp(0)\cup S(0)$ for all $\tau\in(-\eps,\eps)$, then $c_\tau=c_0$ and $r_k(\tau)=r_k(0)$ for all $\tau\in(-\eps,\eps)$ and $k=1,\dots,K$.
\end{theorem}

%{\color{red} [Add Remark on the interpretation / use of $S(\tau)$.]}
The theorem has two immediate corollaries.
The first one concerns the whole length spectrum, and the second one the length spectrum of two velocity profiles.

\begin{corollary}[Length spectral rigidity of a layered planet with moving interfaces]
\label{cor:lsp-rig}
Fix any $\eps>0$ and $K\in\N$, and let $c_\tau(r)$ be an admissible family of profiles with discontinuities at $r_k(\tau)$ for all $k=1,\dots,K$.
Suppose that the length spectrum for each $c_\tau$ is countable\joonas{Perhaps this countability can be proven from other assumptions. Basically we would prove that there are a countable number of topological types of an orbit and each type has only countably many instances modulo rotations. Cf. \cite[Lemma 4.14]{HIKRigidity}. I feel that a proof is either trivial or very hard.} in the 
%annulus $A(1,r_K(\tau))$.
ball $\bar B(0,1)$.

Let $\lsp(\tau)$ and $\blsp(\tau)$ denote the length spectrum and the basic length spectrum of 
%the annulus $A(1,r_K(\tau))$
the ball $\bar B(0,1)$
with the velocity profile $c_\tau$.
Suppose either one of the following holds:
\begin{itemize}
\item
$\lsp(\tau)=\lsp(0)$ for all $\tau\in(-\eps,\eps)$.
\item
$\blsp(\tau)=\blsp(0)$ for all $\tau\in(-\eps,\eps)$.
\end{itemize}
Then $c_\tau=c_0$ and $r_k(\tau)=r_k(0)$ for all $\tau\in(-\eps,\eps)$ and $k=1,\dots,K$.
\end{corollary}

\begin{corollary}[Length spectral rigidity with two polarizations]
\label{cor:2-speeds}
Fix any $\eps>0$ and $K\in\N$, and let $c^i_\tau(r)$ with both $i=1,2$ be an admissible family of profiles with discontinuities at $r_k(\tau)$ for all $k=1,\dots,K$.

Consider all periodic rays which are geodesics within each layer and satisfy the usual reflection or transmission conditions at interfaces, but which can change between the velocity profiles $c^1_\tau$ and $c^2_\tau$ at any reflection and transmission.
Suppose that the length spectrum of this whole family of geodesics, denoted by $\lsp(\tau)$, is countable in the 
%annulus $A(1,r_K(\tau))$.
ball $\bar B(0,1)$.

If $\lsp(\tau)=\lsp(0)$ for all $\tau\in(-\eps,\eps)$, then $c^i_\tau=c^i_0$ for both $i=1,2$ and $r_k(\tau)=r_k(0)$ for all $\tau\in(-\eps,\eps)$ and $k=1,\dots,K$.
\end{corollary}

The ``noise'' set $S(\tau)$ of Theorem~\ref{thm:blspr-multilayer} plays an important role.
One metric is recovered at a time, and all rays that have one leg following the other metric or different legs in different layers are treated as noise.

The proofs of the corollaries are immediate:
\begin{itemize}
\item For Corollary~\ref{cor:lsp-rig}, simply let $S(\tau)=\lsp(\tau)$.
\item For Corollary~\ref{cor:2-speeds}, study the basic length spectra of the profiles $c^1(\tau)$ and $c^2(\tau)$ independently of each other and let again $S(\tau)=\lsp(\tau)$.
\end{itemize}

\begin{remark}
\label{rmk:variations}
Some variations of Theorem~\ref{thm:blspr-multilayer} and its corollaries hold true.
One can introduce an impermeable core and work with a finite number of layers that do not exhaust the ball.
One can choose to include or exclude rays with reflections from the lower boundary $r_K(\tau)$ and the results remain true for this smaller length spectrum, at least when $r_K$ is independent of $\tau$.
%Similarly, one can include the whole ball $\bar B(0,1)$, so that $r_K(\tau)$ marks the lowest interface around the center rather than the lower boundary of the domain.
The proofs are immediate adaptations of the one we give.
\end{remark}

Recall the Neumann spectrum of the Laplace–Beltrami operator is denoted $\spec(c)$, where we impose Neumann-type
boundary conditions (we can allow for other boundary conditions cf. section \ref{sec: connect to LB}).

\begin{theorem}[Spectral rigidity with moving interfaces]
\label{t: spectral rigiditiy}
Fix any $\eps>0$ and $K\in\N$, and let $c_\tau(r)$ be an admissible family of profiles with discontinuities at $r_k(\tau)$ for all $k=1,\dots,K$.
Suppose that the length spectrum for each $c_\tau$ is countable in the 
%annulus $A(1,r_K(\tau))$.
ball $\bar B(0,1) \subset \RR^3$.
Assume also that the length spectrum satisfies the principal amplitude injectivity condition and the periodic conjugacy condition.

Suppose
$\spec(\tau)=\spec(0)$ for all $\tau\in(-\eps,\eps)$.
Then $c_\tau=c_0$ and $r_k(\tau)=r_k(0)$ for all $\tau\in(-\eps,\eps)$ and $k=1,\dots,K$.
\end{theorem}

\begin{proof}
The spectrum determines the trace of the Green's function by Proposition~\ref{prop:Trace Formula}.
As $\spec(\tau)=\spec(0)$ for all $\tau$, the trace is independent of $\tau$ and so are its singularities.
The singularities are contained in the set $\lsp(\tau)$ by Proposition~\ref{prop:Trace Formula}.
We apply Theorem~\ref{thm:blspr-multilayer} to pass from length spectral information to geometric information.

We set $S(\tau)$ to be the singular support of the trace.
Every length of a \simp{} periodic broken ray only appears once in the whole length spectrum by assumption, whence there is a singularity for every \simp{} length.
Therefore $\blsp(\tau)\subset S(\tau)$.
Now Theorem~\ref{thm:blspr-multilayer} implies the claim.
\end{proof}

Planets are full balls, but Theorem~\ref{t: spectral rigiditiy} holds for an annulus as well.
Cf. Remark~\ref{rmk:variations}.

\begin{remark}[Implications for planets] 
The theorem is stated for a scalar operator (the Laplace-Beltrami operator), but the proof extends to the radial elastic case and thus, round planets by considering the toroidal modes associated with the shear wave speed and their corresponding eigenfrequencies.
The proof of the theorem is using a trace formula to recover the \simp{} length spectrum from the spectrum and employ the length spectral rigidity results.
See sections \ref{s: toroidal modes} and \ref{sec: connect to LB}, where we initially start the proof of the trace formula using toroidal modes and show why the argument is identical for the scalar Laplace-Beltrami operator. In that case, we work inside an annulus with an inner boundary representing the core-mantle boundary for more generality. By considering toroidal modes, the argument for proving a trace formula for spheroidal modes that involve two wave speeds becomes more transparent and is discussed in section \ref{s: spheroidal modes}. Hence, by considering the spectrum of the radial isotropic elastic operator with natural boundary conditions, our arguments may be generalized to recover both elastic wave speeds using Corollary \ref{cor:2-speeds}.
\end{remark}

\begin{remark}\label{rem: dimension}
We note that the dimension is irrelevant for the length spectral rigidity results; if the sound speed is fixed, the length spectrum is independent of dimension. For spectral rigidity, we assume dimension three to ease the computation of the trace formula since it allows us to compute the leading order asymptotics of the eigenfunctions explicitly.
\end{remark}

This paper will be essentially divided into parts. The first part is proving length spectral rigidity. In the second part, we prove the trace formula in our setting, and as a corollary, we prove the spectral rigidity theorem.

\subsection{Reasonableness of radial models}
\label{sec:reasonable-radial}

Spherically symmetric Earth models are widely used in geophysics and there are a number of results showing how well such models fit seismic data. The $P$ and $S$ wave speeds are denoted $c_P$ and $c_S$.
%The Preliminary Reference Earth Model (PREM) was generated by matching eigenfrequencies \cite{DZIEWONSKIPREM} - here, we prove rigidity of models of the type of PREM (SNREI). 
There are several important questions to address when using PREM to analyze seismic data.

\subsubsection*{
Question 1. What is the uncertainty in the best-fitting spherical average profile?}

The classic reference for this question is Lee and Johnson in \cite{Lee1984}. They suggest that the extremal bounds in the upper mantle are around 0.6 km/s (around 6 \%) for $c_P$ and 0.4 km/s for $c_S$ (around 7 \%). In the lower mantle, it is around 0.18 km/s (around 2 \%) for $c_P$, and 0.14 km/s (around 2 \%) for $c_S$. Note that the bounds increase in the lowermost mantle and especially in the crust.

\subsubsection*{
Question 2. What is the standard deviation of the residuals to the spherical average model, as a function of depth?}

In theory, residuals can be calculated as a function of depth for any global tomographic model. However, this information is not always presented. A good, thorough, recent example is the SP12RTS model \cite{koelemeijer2016a}. Their figure 9a shows that variations are smallest in the mid-mantle (standard deviations of around 0.1 \% for $c_P$, 0.2 \% for $c_S$) and increase towards the surface (to around 1.0 \% for both $c_P$ and $c_S$) and towards the CMB (to around 0.3 \% for $c_P$, and 0.5 \% for $c_S$).

\subsubsection*{
Question 3. What is the measurement uncertainty in the wave speed at a given point in a typical tomographic model?
}
 Very few groups have given robust estimates of point-wise measurement uncertainties, and the best study to date could be the Bayesian study by Burdick and Leki\'{c} in \cite{BurdickLekic17}. They find the standard deviation in estimates of 0.25 \% $dc_P/c_P$ (so, for example the anomaly in California at 10 km depth might be 1.00 \% +/- 0.25 \%). We are not aware of any similar estimates for $c_S$, but they would most likely be more uncertain.

\subsubsection*{Question 4. In a given region, what is the typical variation in the absolute wavespeed?
}
Near Earth's surface, there are huge lateral variations in wavespeed, for example between continental and oceanic regions (for example, at a depth of 50 km, mountain belt may have a $c_P$ of 6.1 km/s, while an ocean basin may have a $c_P$ of 8.1 km/s at the same radial coordinate, a variation of 25 \%. However, within a given region type (e.g. 'island arc' or 'mountain belt'), typical variations around 0.3 km/s for $c_P$ (an authoritative reference is \cite{MooneyLaske98}; see their fig. 3b), which is about 5 \%. Variations in $c_S$ can be larger because $c_S$ is more strongly affected by fluids and temperature (partial melting and anelasticity). The reference given does not address $c_S$.

\section{Unraveling assumptions}

Let us give the relevant definition and assumptions on the geometry of the problem.
Recalling from the previous section, fix $K \in \mathbb N$ and let $r_k \in (0, \surf)$ so that $\surf=:r_0 > r_1 > \cdots > r_K$.
Assume $c(r)$ has jump discontinuities at each $r_k \in (0, \surf)$. Let $\Gamma = \bigcup_k \{ r = r_k\}$ be the collection of interfaces together with $\p M$, and denote $\Gamma_k := \{ r = r_k\}$. We view $M$ as a Riemannian manifold with (rough) metric $g = c^{-2} dx^2$.
We showed in \cite{HIKRigidity} that any rotation symmetric Riemannian manifold with the Herglotz condition is of this form. The same is true in the presence of jumps with essentially the same proof we used in the smooth setting.

\subsection{Geodesics in a spherically symmetric model with interfaces}

%On the $n$-dimensional manifold $M$ the phase space of the unit speed geodesic flow has dimension $2n-1$.
On the three-dimensional manifold $M$ the phase space of the unit speed geodesic flow has dimension $5$.
Due to rotation symmetry most of these dimensions are superfluous, and the dimension of the reduced phase space needed to represent all geodesics up to isometries of the manifold is only $2$.
The dimension of the ``reduced phase space'' is $2$ for any ambient dimension $2$ or higher.

Two natural coordinates in this space are the radius $r$ (Euclidean distance to the origin) and the angular momentum denoted as $p$.
Any geodesic is either radial or is contained in a unique plane through the origin, so it suffices to study geodesics in $2$-dimensional disks.
In dimension two, points on the disk can be described with polar coordinates $(r,\theta)$, and a geodesic $\gamma$ can be parameterized as $t \mapsto (r(t), \theta(t))$.
We then have the explicit formula $p = p_{\gamma} = c(r(t))^{-2}r(t)^2\theta'(t)$.
The angular momentum (often called the \emph{ray parameter} associated to $\gamma$) $p$ is conserved, even across discontinuities in the metric.
Therefore trajectories of the geodesic flow in the $(r,p)$-plane are horizontal lines.

Much of the geometry is conveniently encoded in the function $\rho(r)=r/c(r)$.
At a turning point (where $\dot r=0$) we have $\abs{p}=\rho(r)$, and elsewhere $\abs{p}<\rho(r)$.
Therefore the reduced phase space is the subgraph of the function $\rho\colon(0,1]\to(0,\infty)$.
The classical Herglotz condition states that $\rho'(r)>0$ for all $r$.
Three examples are given in figure~\ref{fig:profiles}.

\begin{figure}
    \centering
    \begin{overpic}[width=0.8\textwidth]{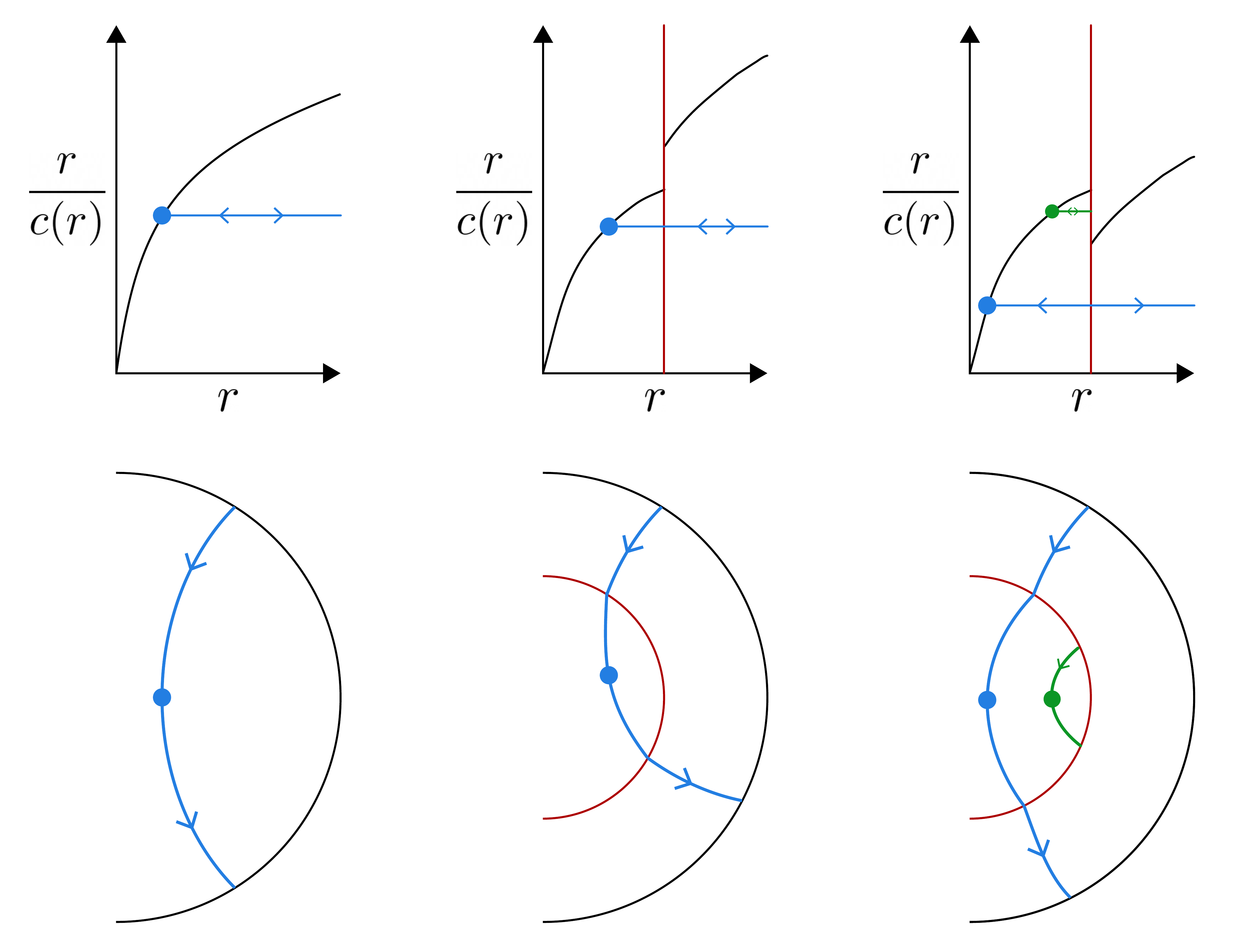}
    \put (0,73) {(a)}
    \put (35,73) {(b)}
    \put (70,73) {(c)}
    \end{overpic}
    \caption{Three different velocity profiles described in terms of the function $\rho(r)=r/c(r)$. Dashed vertical lines connect the plot with the manifold. The reduced phase space of the geodesic flow is the subgraph of the function $\rho$ and the trajectories are horizontal lines. The Herglotz condition implies that $\rho$ is increasing and thus all horizontal lines starting at the graph can be extended all the way to $r=1$ while staying under the graph. Therefore rays starting at any depth meet the surface. The classical Herglotz condition is satisfied in case (a) above. In case (b) an extended Herglotz condition is satisfied, where $\rho'>0$ in the sense of distributions. The jump at the interface ({\color{red}red}) has to be positive for this to hold. In case (c) the smooth segments satisfy the Herglotz condition but the jump is in the wrong direction. Therefore rays diving just below the corresponding interface ({\color{green}green}) are trapped by total internal reflection. Even in the presence of discontinuities the condition $\rho'>0$ implies that there is no trapping, and jumps in the wrong direction necessarily imply trapping. The Herglotz condition is a convexity condition on the phase space.}
    \label{fig:profiles}
\end{figure}

\begin{definition}
A (unit-speed) \emph{broken geodesic} or \emph{ray} in $(M, g)$ is a continuous, piecewise smooth path $\gamma: \RR \supset I \to M$ such that each smooth piece is a unit-speed geodesic with respect to $g_c$ on $M\setminus \Gamma$, intersecting the interfaces $\Gamma$ at a discrete set of times $t_i \in I$. Furthermore, at each $t_i$, if the intersection is
transversal, then Snell's law for reflections and refraction of waves is satisfied.  More precisely, a
broken geodesic (parameterized by a time variable) can be written as $\gamma : (t_0, t_1) \cup (t_1, t_2) \cup \cdots \cup
(t_{k-1}, t_k) \to M \setminus \Gamma$, which is a sequence of geodesics concatenated by reflections and refractions obeying Snell's law: for $i = 1, \dots, k-1,$ 
\[
\gamma(t_i) \in \Gamma, 
\qquad \qquad (d\iota_\Gamma)^* (\gamma(t_i),\dot \gamma(t_i^-))  = (d\iota_\Gamma)^* (\gamma(t_i), \dot \gamma(t_i^+)),
\]
where $\iota_\Gamma: \Gamma \to M$ is the inclusion map and $\dot \gamma(t_i^\mp) = \lim_{t \to t_i^\mp} \gamma(t)$. Each restriction 
$\gamma\restriction_{(t_i,t_{i+1})}$ is a maximal smooth geodesic that we call a \emph{\branch{}} of $\gamma$. For each $i$, note that $\gamma(t_i) \in \Gamma_{k_i}$ for some $k_i$.  One can view $\gamma$ as a concatenation of all of its \branch{}s. A \branch{} $\gamma\restriction_{(t_i, t_{i+1})}$ is \emph{reflected} if the inner product of $\dot \gamma(t_i^+)$ and $\dot \gamma(t_i^-)$ with a normal vector to $\Gamma_{k_i}$ have opposite signs. If they have the same sign, it is a \emph{transmitted \branch{}}. If $\dot \gamma(t_i^+)$ and $\dot \gamma(t_i^-)$ are equal, then $\gamma\restriction_{(t_{i-1},t_{i+1})}$ is a \emph{grazing \branch{}} or ray; in this case,  $\dot \gamma(t_i^\pm)$ is tangent to $\Gamma$. The only other situation is when  $\dot \gamma(t_i^+)$ is tangent to $\gamma$ while  $\dot \gamma(t_i^-)$ is not (or vice versa); in this case $\gamma\restriction_{(t_i, t_{i+1})}$ is called a \emph{gliding ray} or \branch{} because it travels along $\Gamma_{i_k}$. A ray with no gliding \branch{}s will be called a non-gliding ray.
Our results will also extend to the elastic setting, which has two wave speeds $c_P$ and $c_S$ corresponding to pressure waves and shear waves. In this case, the definition of broken rays is identical except that each \branch{} can either be a geodesic with the metric $g_{c_P}$ or $g_{c_S}$.
\end{definition}

We follow the discussion and notation in \cite[Section 2.1]{HIKRigidity}. Assume for the moment $n=2$ since due to spherical symmetry, rays are confined to a disk, and equip the annulus $M=A(1,r)$ with spherical coordinates $\theta,r$. Fix a broken geodesic $\gamma$ whose endpoints are both located at a particular interface $r_i$ for some $i \in \{0,\dots, K\}$.
We denote $\alpha= \alpha(p)$ to be the epicentral distance between both endpoints of $\gamma$, where $p = \pg$ is the ray parameter associated to $\gamma$. It is the angular distance that $\gamma$ travels.
It may happen that $\alpha(p) > 2\pi$ if the geodesic winds around the origin several times.

Each \branch{} can be parameterized as 
\begin{equation}
t \mapsto (r(t),\theta(t))
\end{equation} over some maximal interval $I$ associated to the \branch{}. Using both of the conserved quantities
$  c(r(t))^{-2}[r'(t)^2 + r(t)^2\theta'(t)^2 = 1$ and $p =c(r(t))^{-2}r(t)^2\theta'(t)$ (angular momentum) we can compute $\alpha_\gamma$ explicitly following \cite[Equation (2.2)]{HIKRigidity}.

Let $R^*$ be the smallest radius that $\gamma$ passes through, and there is a unique $k$ such that $r_{k} \leq R^* < r_{k-1}$. We refer to $R^*$ as the \emph{radius} of $\gamma$ and it may coincide with an interface or a boundary. Next, $\gamma$ will have a certain number of \branch{}s in each of the annular regions $A(r_{k-1},r_k)), A(r_{k-2}, r_{k-1}),\dots A(r_0,r_1)$. Since $\gamma$ might stay just within a single (or more) annular region, there could be zero \branch{}s in one or more of the annuli. By definition of $R^*$, $\gamma$ has no \branch{}s in $A(r_{k},r_K)$. We denote $n_j$ to be half of the number of \branch{}s of $\gamma$ in $A(r_{j-1},r_j)$. 
Next we introduce a certain quantity
$\wkb^2:= c(r)^{-2}-r^{-2}p^2$.

Analogous to \cite{HIKRigidity}, the length of a broken geodesic with only transmitted \branch{}s, starting in $r = r_0$ and ending at $r = \surf$ is an integer multiple of the quantity
\begin{equation}\label{e: transmitted length L}
L(r_0,p) := 
\int_{r_0}^{\surf} \frac{1}{c(r')^2\wkb(r';p)} \dd r'
\end{equation}
If $r_0 = R^*$ is the radius of $\gamma$, then $R^*$ is a function of $p$ and we will  write $L(p)$.
With this notation and using the computation for epicentral distance in \cite{HIKRigidity}, one can also find an explicit formula for $\alpha_\gamma(r):$ 
\begin{align}
\alpha(p) %&= \sum_{j=1}^{k-1} 2n_j\int^{r_{j-1}}_{r_{j}} \frac{r c(r')}{c(r) (r')^2}\left( 1 - \left(\frac{rc(r')}{r'c(r)}  \right) \right)^{-1/2} \dd r'
%\\&\quad
%+ 2 n_k\int_{R^*}^{r_{k-1}}\frac{r c(r')}{c(r) (r')^2}\left( 1 - \left(\frac{rc(r')}{r'c(r)}  \right) \right)^{-1/2} \dd r'
%\\
&= \sum_{j=1}^{k-1} 2n_j\int^{r_{j-1}}_{r_{j}} \frac{p}{(r')^2\wkb(r',p)} \dd r'
+2n_k \int^{r_{k-1}}_{R^*} \frac{p}{(r')^2\wkb(r',p)} \dd r'.
\end{align}

\begin{definition}
Following Hron in \cite{HronCriteria}, those waves which travel from the source to the receiver along different paths but
with identical travel-times are kinematically equivalent and are called kinematic analogs. We will refer to two different rays connecting source and receiver with the same ray parameter and travel time as \emph{kinematic analogs.} The groups of kinematic analogs may be further divided into subgroups of waves whose amplitude curves are identical. The members of this subgroup of phases may be called dynamic analogs. A sufficient condition for kinematic equivalence of two different broken rays $\gamma_1$ and $\gamma_2$ is they must have an equal number of \branch{}s in each layer along their paths. Since $\alpha(\pg)$ just measures the epicentral distance between the endpoints, $\alpha(\pg)$ will be the same for $\gamma$ and all of its kinematic analogs. We will say two non-gliding rays connecting source and receiver are \emph{dynamic analogs} if they have the same ray parameter, travel time, and inside each $A(r_k, r_{k-1})$, they have the same number of \branch{}s that are reflections starting at $\Gamma_k$, transmission starting at $\Gamma_k$, reflections starting at $\Gamma_{k-1}$ and transmissions starting at $\Gamma_{k-1}$. This is a sufficient condition to ensure that the principal amplitudes of the corresponding waves are identical. See \cite{HronCriteria} for examples and figures of kinematic and dynamic analogs.
\end{definition}

For length spectral rigidity, we only require what we term \emph{\simp} closed rays.

\begin{definition}[Basic rays]\label{simple closed ray}
A broken ray is called \emph{\simp{}} if either it stays within a single layer and all of its \branch{}s are reflections from a single interface (type 1), or it is a \emph{radial} ray contained in a single layer (type 2).
A \emph{radial} ray is defined to be a ray with zero epicentral distance.
It will necessarily reflect from two interface and cannot be type 1. 
The first type of \simp{} rays are analogs to the \emph{turning} rays in \cite{HIKRigidity} that formed $\lsp(c)$ in the notation there. A closed \simp{} ray of the first type will be periodic, stay within a singular layer, and only consists of reflected \branch{}s from a single interface. 
We have illustrated \simp{} and other periodic rays in Figure~\ref{fig:basic}.

The lengths of periodic \simp{} will suffice to prove length spectral rigidity so we define $\blsp(c)$ as the set of lengths of all periodic \simp{} rays.
\end{definition}

%\begin{comment}
\begin{figure}
    \centering
    \includegraphics[width=0.28\textwidth]{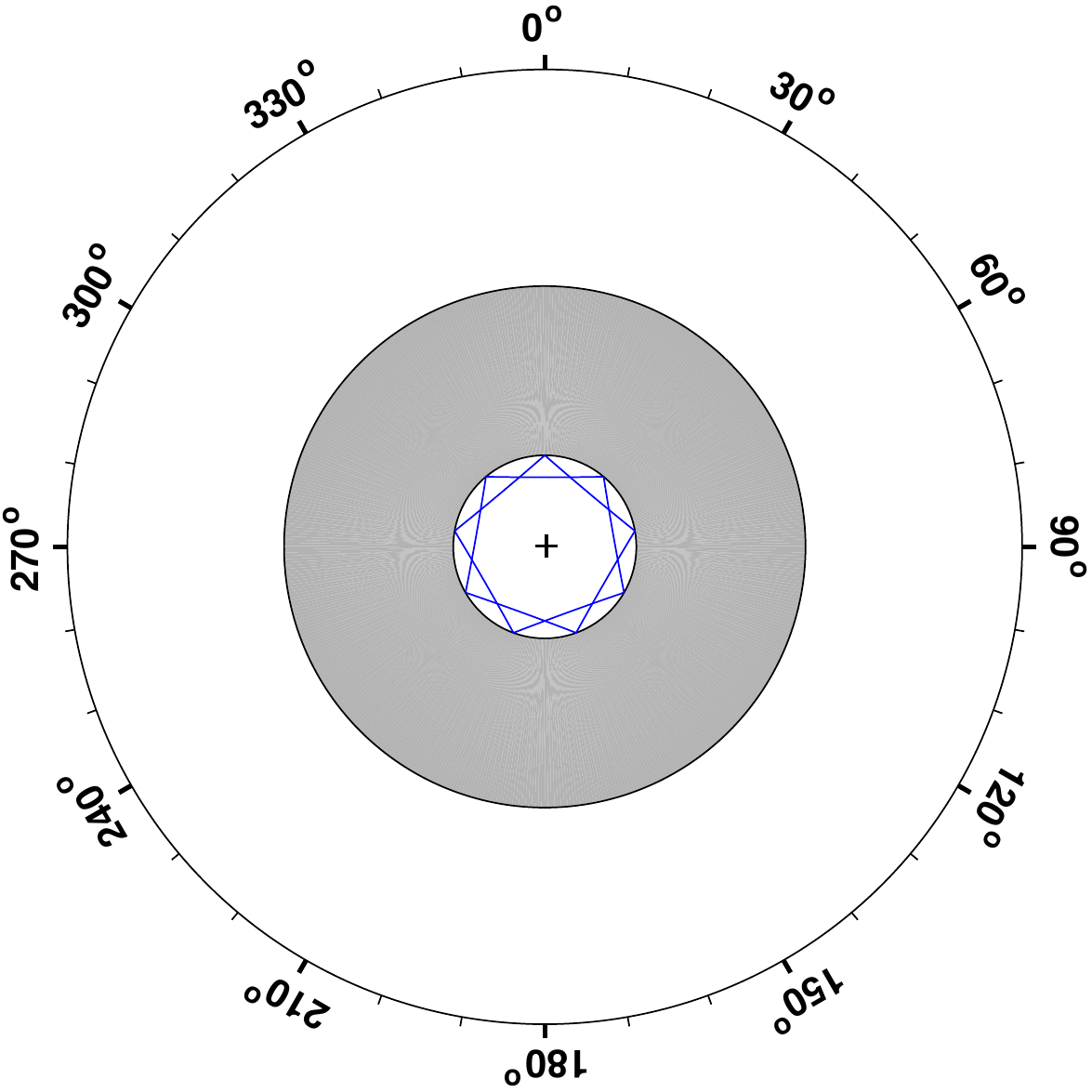}
    \includegraphics[width=0.28\textwidth]{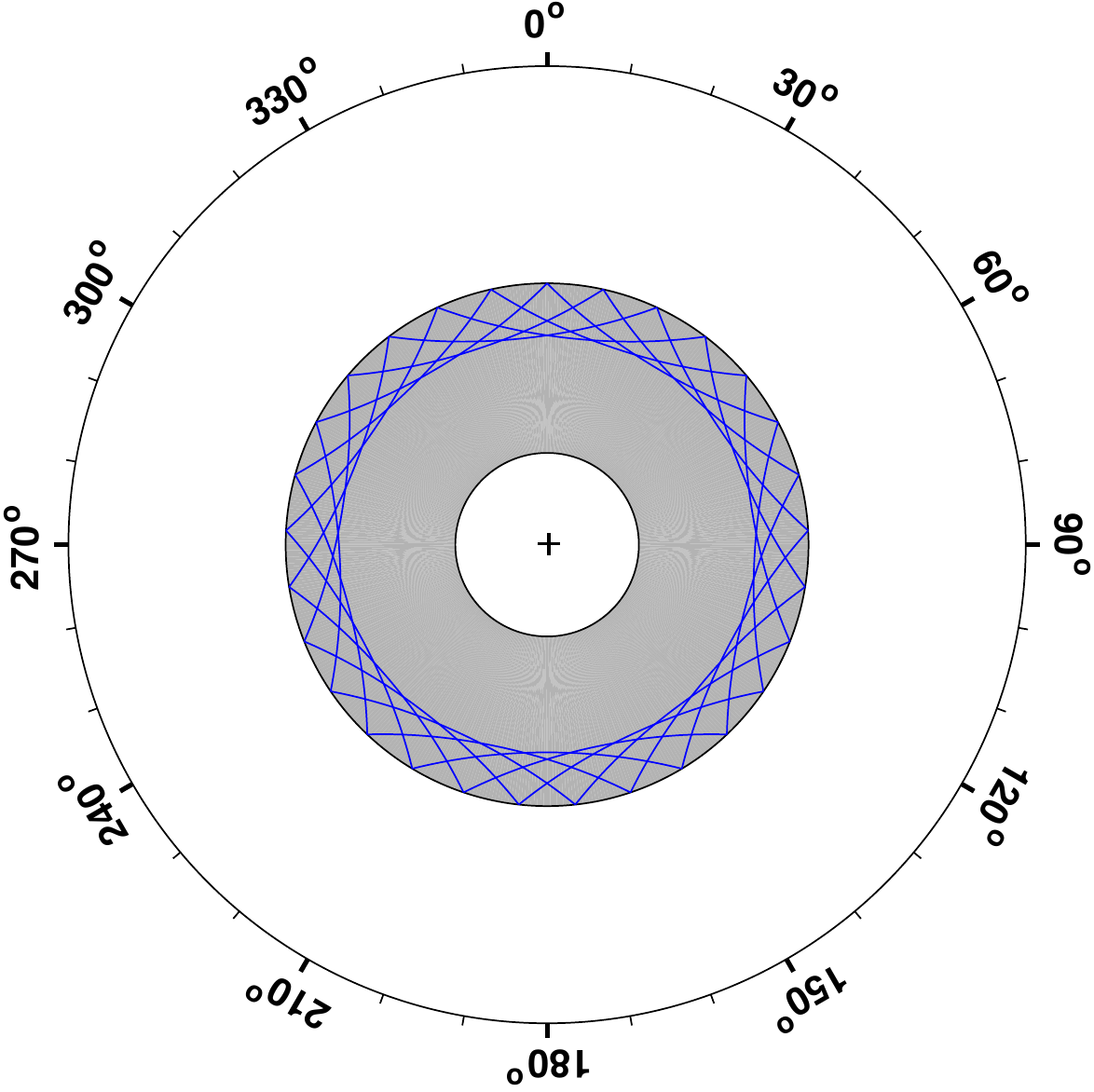}
    \includegraphics[width=0.28\textwidth]{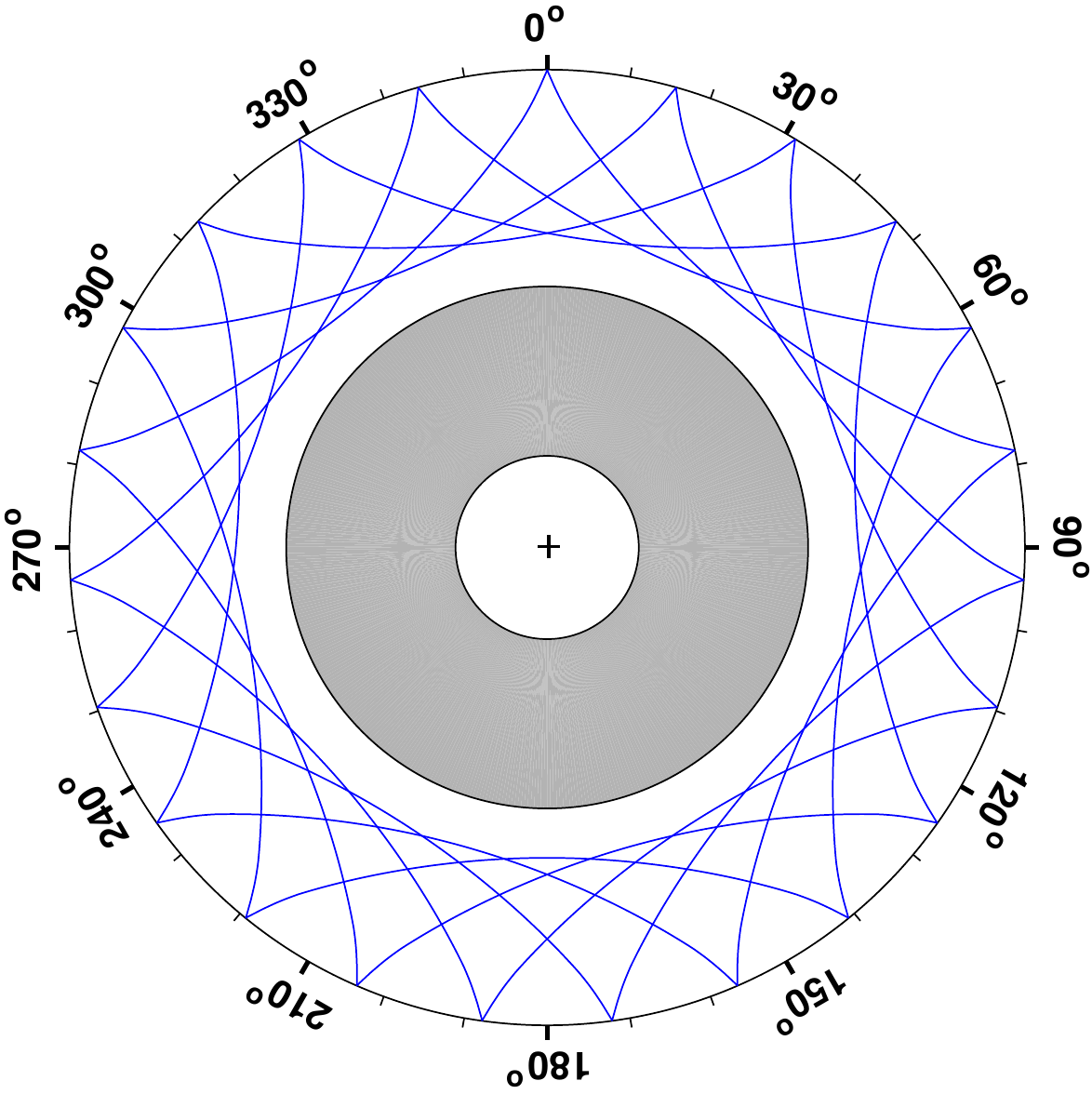} \\[0.25cm]
    \includegraphics[width=0.28\textwidth]{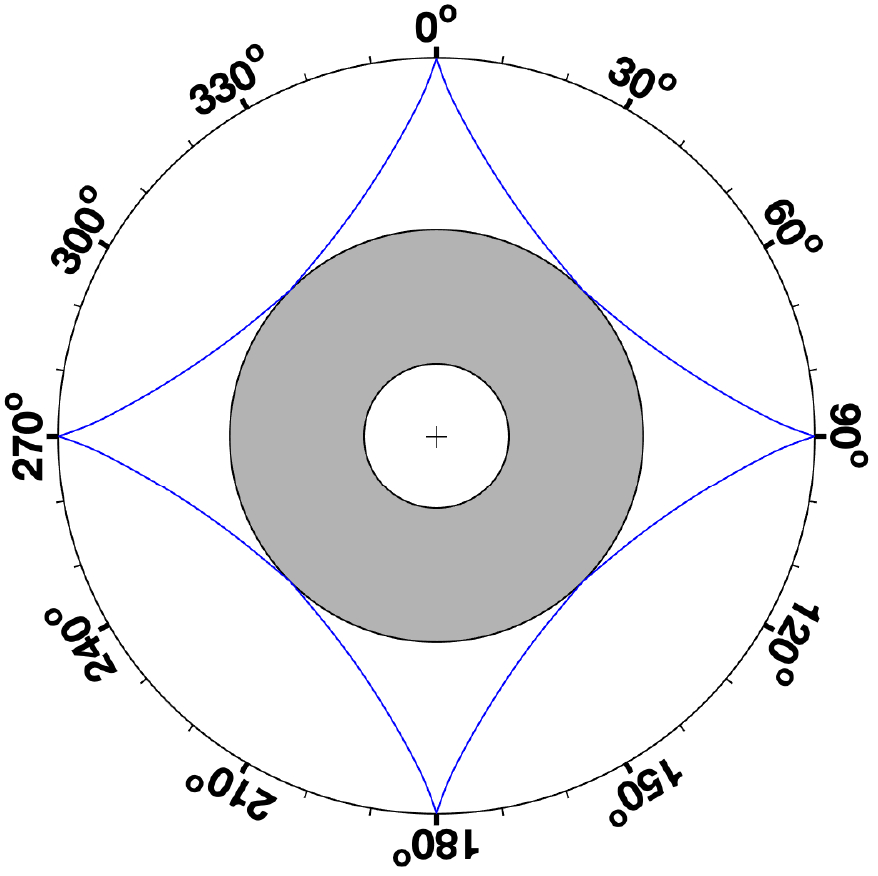}
    \includegraphics[width=0.28\textwidth]{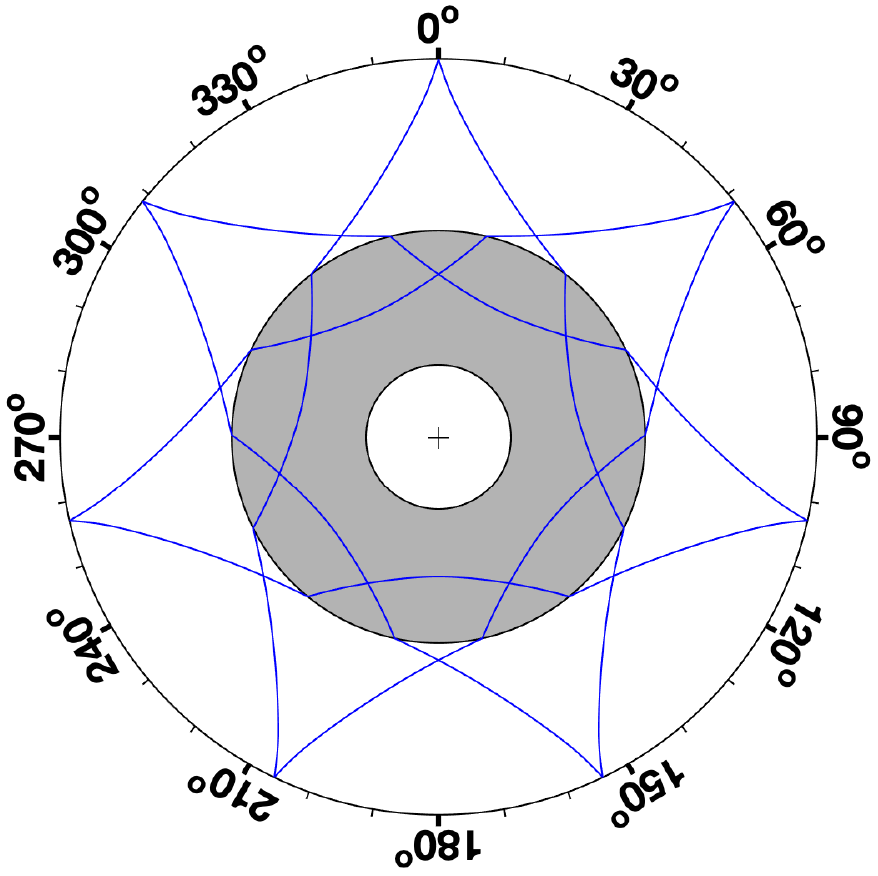}
    \includegraphics[width=0.28\textwidth]{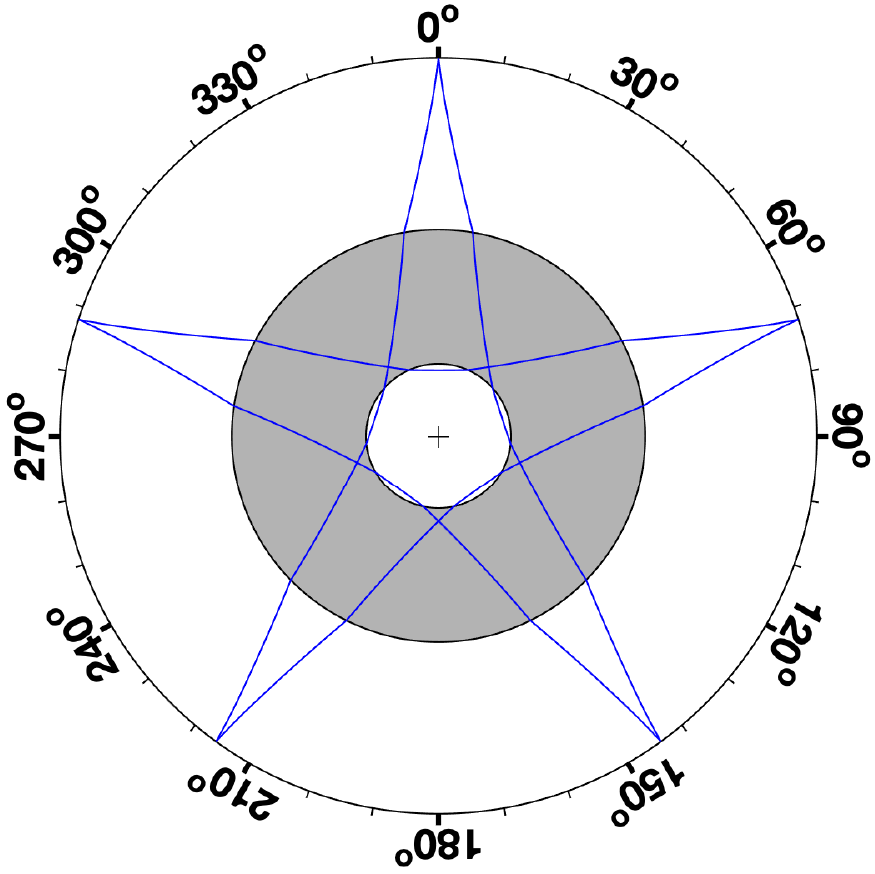} \\[0.25cm]
    \includegraphics[width=0.28\textwidth]{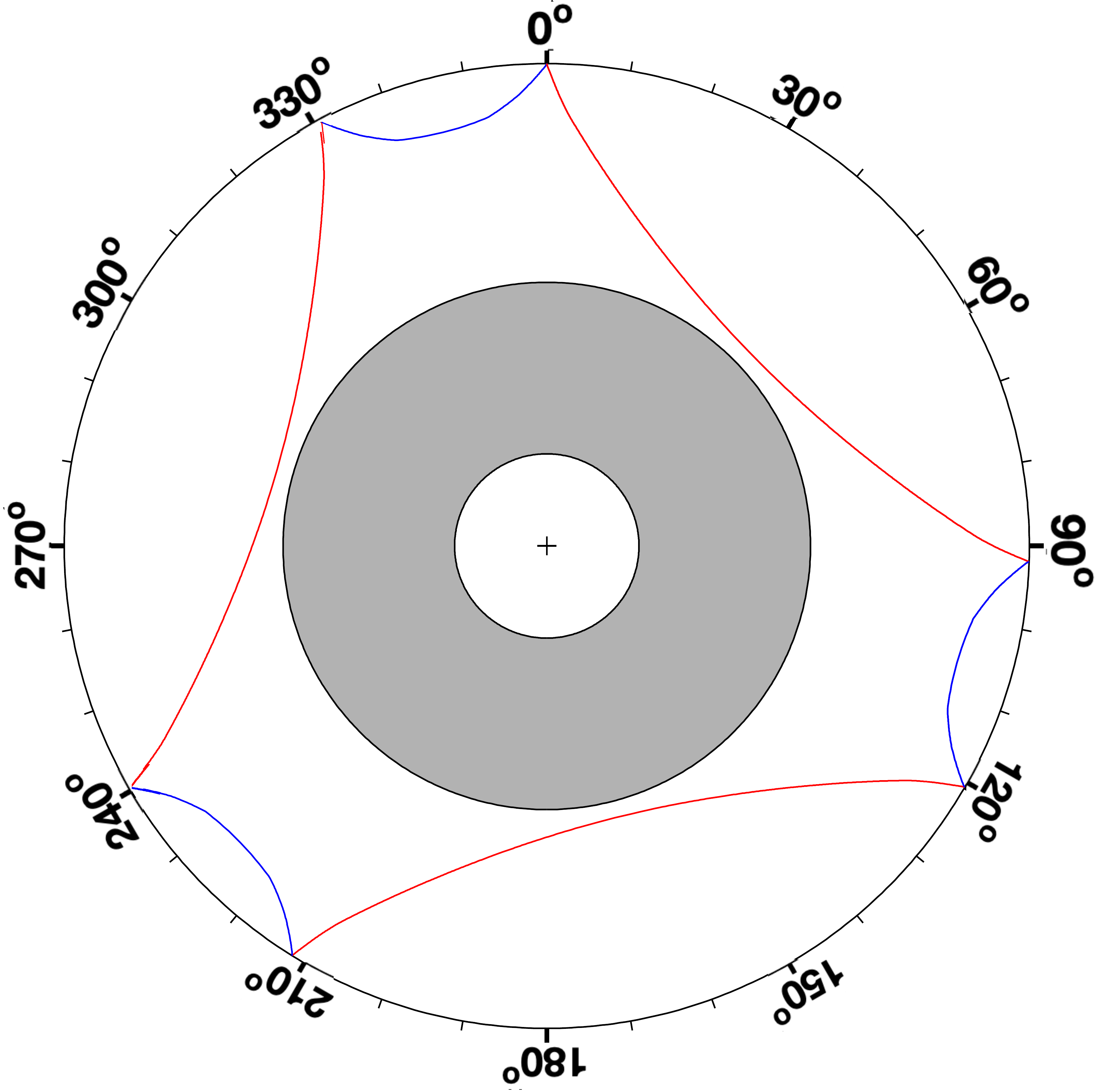}
    \includegraphics[width=0.28\textwidth]{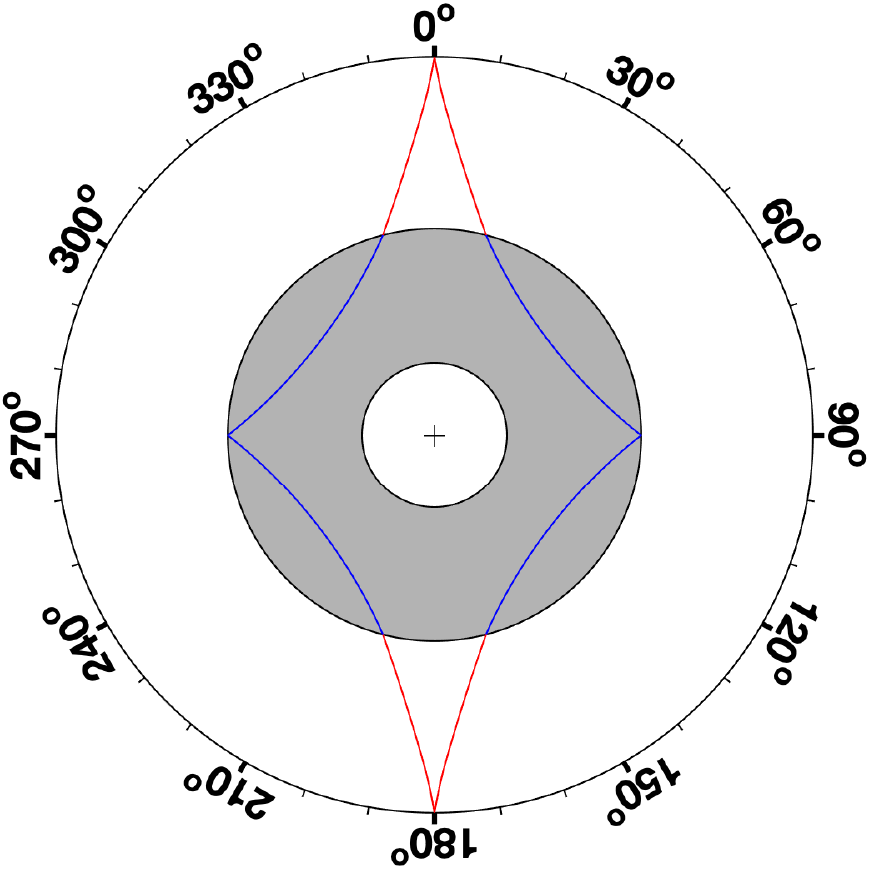}
    \includegraphics[width=0.28\textwidth]{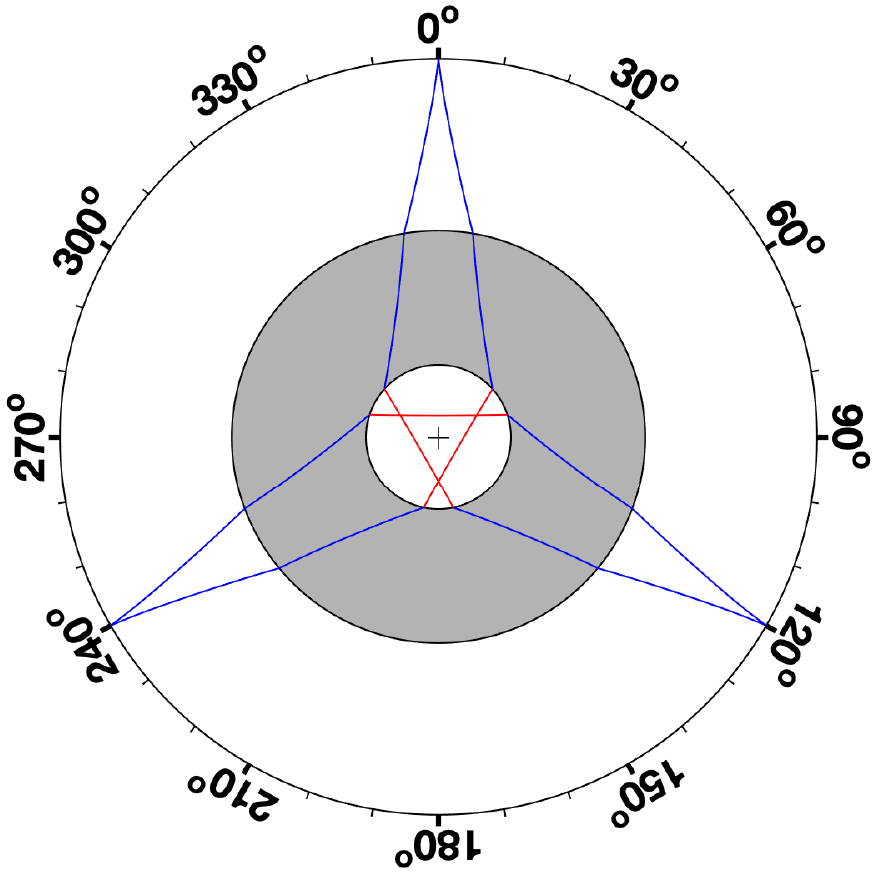}
    \caption{Some periodic rays in a radial planet with two interfaces (PREM). The top row illustrates examples of \textit{basic} rays (with different winding numbers), the middle row illustrates rays (left-to-right: PcP, PKPab, PKIKP) that are not basic and only probe the \textit{P} wave speed, and the bottom row also illustrates examples of non-basic rays (left-to-right: SP, SKKS, PKJKP) that probe both \textit{P} (in blue) and \textit{S} (in red) wave speeds. Acknowledgement: Chunquan Yu.}
    \label{fig:basic}
\end{figure}
%\end{comment}
Computing the length and epicentral distance of \simp{} rays is much simpler. Let $\gamma$ be a \simp{} ray with radius $R^*$, ray parameter $p$, and lies inside inside $A(r_{k-1},r_k)$. Then there is a unique $N(p) \in \mathbb{N}$ so that the length, denoted $T(p)$, of $\gamma$ is 
\[
T(p) = 2N(p)L(p) = 2N(p)\int_{R^*}^{r_{k-1}} \frac{1}{c(r')^2\wkb(r';p)} \dd r'
\]
and
\[
\alpha(p) =  2N(p)\int_{R^*}^{r_{k-1}} \frac{p}{(r')^2\wkb(r';p)} \dd r'. 
\]

\begin{definition}\label{d: ccc}
Consider geodesics in an annulus $A(a,b)$ equipped with a $C^{1,1}$ sound speed $c\colon(a,b]\to(0,\infty)$.
We say that $c$ satisfies the \emph{countable conjugacy condition} if there are only countably many radii $r\in (a,b)$ so that the endpoints of the
corresponding maximal geodesic $\gamma(r)$ are conjugate along that geodesic.
\end{definition}

We will only need the countable conjugacy condition with each layer, so we do not need a definition in the presence of discontinuities.
We point out that ``countable'' includes the possibility that the set be empty or finite.
Definition~\ref{d: ccc} is the same as the one given in~\cite{HIKRigidity}.

We need an analog to the clean intersection hypothesis used in \cite{HIKRigidity,Mel79} to prove a trace formula that also makes sure that the phase function is Bott-Morse nondegenerate when applying a stationary phase argument.

\begin{definition}\label{d: pcc}
We say that the radial wave speed $c$ satisfies the \emph{periodic conjugacy condition} if for each periodic, nongliding ray with a ray parameter $p$, $\p_p \alpha(p) \neq 0$. This condition ensures that the phase function in the stationary phase argument for computing the trace formula is Bott-Morse nondegenerate.
\end{definition}

\subsection{Gliding rays as limits} \label{s: gliding as limits}

\joonas{What follows is not fully rigorous but should do for the project review draft. It might benefit from a picture, but I wouldn't conjure one up this week unless my rough drawings will do. Do we want to have this as text or as a proper proposition or something?}

\maarten{Condense this further, just give the main ideas? (From today's chat -- J)}
\joonas{I would wait with this over the project review, unless it happens easily and there is time to spare.}

Consider a periodic broken ray $\gamma_0$ with a gliding \branch{} of positive length.
\joonas{It is possible to have zero-length gliding legs. In this case there is potential for gliding but none occurs. Those I can't approximate but those probably are not an issue either.}
We assume that gliding occurs at only one interface; this is ensured by the smooth Herglotz condition.
We may rearrange the \branch{}s of the periodic broken ray without changing its length or essential geometry so that there is only one gliding \branch{} per period.
\joonas{Is there a name for this rearrangement? Having several gliding legs is not a real issue, but inconvenient to write, and rearrangement invariance is built in anyway.}
We will argue that there is a sequence of periodic non-gliding broken rays $\gamma_i$ so that $\gamma_i\to\gamma_0$.
This is very simple for any finite segment of a gliding broken ray; the subtlety lies in ensuring periodicity of the approximating rays. We will prove the following lemma.

\begin{lemma}\label{l: gliding} Let $\gamma_0$ be a periodic broken ray with a gliding leg of positive length as described above. Then there is a sequence $\{\gamma_i\}_{i=1}^\infty$ of periodic, non-gliding broken rays such that \[
\lim_{i \to \infty} \gamma_i = \gamma
\]
\end{lemma}

\begin{proof}
Let $x$ and $y$ be the final and initial point, respectively, of the gliding \branch{} of $\gamma_0$, and let $\theta_0$ be the angle between $\gamma_0$ and the interface.
We wish to find angles $\theta_i>\theta_0$ with the correct approximating and periodicity properties.

For any angle $\theta>\theta_0$, let the angle between the interface and the \branch{} of the refracted ray in the lower layer be denoted by $\kappa$.
In the limiting case $\kappa_0=0$ as the ray $\gamma_0$ glides on the interface.
It follows from Snell's law and a calculation that
\begin{equation}
\label{eq:kappa-sim}
\kappa
%\sim
=
a(\theta-\theta_0)^{1/2}
+
\Order(\theta-\theta_0)
\end{equation}
for some constant $a>0$.
%Here and throughout this subsection $LHS\sim RHS$ stands for $\lim_{\theta\to\theta_0+}\frac{LHS}{RHS}=1$.
%\joonas{It is also possible to write this with error terms $\Order(...)$. The conclusion is the same, so it is much like a choice of words. --- Switched to "error form".}

When $\theta$ is slightly above $\theta_0$ --- or when $\kappa>0$ is small --- the opening angle of a single short diving \branch{} under the interface is denoted by $\phi(\theta)$.
A simple calculation shows that $\phi(\theta)$ is asymptotically comparable to $\kappa$, whence
\begin{equation}
\label{eq:phi-sim}
\phi(\theta)
%\sim
=
b(\theta-\theta_0)^{1/2} 
+
\Order(\theta-\theta_0)
\end{equation}
for some constant $b>0$.

Let the angle between the points $y$ and $x$ be $\alpha_0>0$.
Starting from the point $x$ and following the broken ray near $\gamma_0$ with the initial angle $\theta\approx\theta_0$ we get a map $\theta\mapsto y(\theta)$.
This map is $C^1$.
Denote the angle between $y(\theta)$ and $x$ by $\alpha(\theta)$.
This map is well defined in a neighborhood of $\theta_0$, as the relevant broken ray stays above the interface and total internal reflection is not an issue.

If $\alpha'(\theta_0)=0$, then the points $x$ and $y$ are conjugate along the non-gliding part of the broken ray $\gamma_0$.
But this turns out not to be an issue.
%In this case little can be said.
\joonas{I initially thought it would be an issue, but apparently not! The coefficient $c$ disappears into the next-to-leading order. If the separation angle $\alpha_0$ vanishes, then not being conjugate becomes a leading order issue, but the proof fails due to wrong scaling anyway. So nothing can be won by assuming no conjugate points, it seems.}
%We therefore suppose that the points are not conjugate, whence $\alpha'(\theta_0)=c\neq0$.
%This gives
Denoting $\alpha'(\theta_0)=c$, we have
\begin{equation}
\label{eq:alpha-sim}
\alpha(\theta)
-
\alpha_0
%\sim
=
c
(\theta-\theta_0)
+
\Order((\theta-\theta_0)^2)
\end{equation}
due to simple Taylor approximation.

We want to choose the angle $\theta>\theta_0$ so that an integer amount of these short diving \branch{}s connect $y(\theta)$ to $x$.
The condition is $\alpha(\theta)/\phi(\theta)\in\N$.
Combining with equations~\eqref{eq:kappa-sim}, \eqref{eq:phi-sim}, and~\eqref{eq:alpha-sim}, we end up with the condition that
\begin{equation}
\label{eq:gliding-periodic}
\alpha_0 b^{-1} (\theta-\theta_0)^{-1/2}
+
\Order((\theta-\theta_0)^{1/2})
\in\N.
\end{equation}
Here the error term depends continuously on $\theta$, so the left-hand side of equation~\eqref{eq:gliding-periodic} obtains integer values infinitely many times as $\theta\to\theta_0+$.
This gives us a choice of directions $\theta_i$ starting at $x$, and thus a sequence of periodic broken rays $\gamma_i$ which converge to $\gamma_0$.
\joonas{Do we need to discuss the mode of convergence? It is pointwise or locally uniform. It can only be globally uniform if the periods stay constant, but that need not be the case.}

This concludes the argument that every periodic broken ray with a gliding \branch{} can be approximated by periodic non-gliding rays.
\end{proof}
%, provided that the endpoints of the gliding \branch{} are not conjugate along the rest of the broken ray.
\joonas{This can be put into a lemma/proposition/other environment if it's better.}

\subsection{Principal amplitude injectivity condition}\label{s: geometric spreading injectivity condition}

We also need an assumption similar to ``simplicity'' of the length spectrum modulo the group action in order to recover the length spectrum when there are multiple components in the length spectrum. For a closed ray $\gamma$, denote $[\gamma]$ to be the equivalence class to include all rotations and dynamic analogs of $\gamma$ along with its time reversal. We will see that $[\gamma]$ has a particular contribution to the trace formula.

The principal contribution of $[\gamma]$ with ray parameter $p$ to the trace formula has the form (see \eqref{eq: trace coefficient})
\[
c(t-T(p)+i0)^{-k} \ii^{N(p)}  n(p)Q(p)L(p)\abs{p^{-2}\p_p \alpha}^{-1/2}
\]
where $c$ is independent of $\gamma$, $Q(p)$ is a product of reflection and transmission coefficients, and $T(p)$ is the length of $\gamma$. Theoretically, there may be another class $[\gamma']$ with an identical period whose principal contribution to the trace cancels with that of $[\gamma]$, thereby preventing recovery of $T$. 

We say that the length spectrum satisfies the \emph{principal amplitude injectivity condition} if given two closed rays $\gamma_1$ and $\gamma_2$ with the same period and disjoint equivalence classes (so they must have different ray parameters $p_1$ and $p_2$), then
\[
n(p_1)Q(p_1)\abs{p_1^{-2}\p_p \alpha(p_1)}^{-1/2}
\neq n(p_2)Q(p_2)\abs{p_2^{-2}\p_p \alpha(p_2)}^{-1/2}.
\]
We assume that $\lsp(c)$ satisfies the principal amplitude injectivity condition in order to prove Theorem \ref{t: spectral rigiditiy}.

\subsection{Spherical symmetry}
\label{sec:symmetry}

In section~\ref{sec:reasonable-radial} we saw that spherical symmetry is a good approximation for the Earth.
This symmetry is of tremendous technical convenience.
The geodesic flow is integrable with simple conserved quantities (an orbital plane and an angular momentum) and many of our calculations can be done explicitly.

The geometry of periodic broken rays is poorly understood outside symmetric situations.
It is not clear whether there are densely many such rays on a general manifold with boundary, nor whether the periodic rays are stable under deformations of the geometry.

On general manifolds, small smooth perturbations of a smooth metric only have a second order effect on the direction of the geodesics.
However, small smooth deformations of an interface have a first order effect, and this increased sensitivity substantially complicates matters.
Radial deformations of radial models are better behaved in that the conserved symmetry and deformed conserved quantities make the deformations tractable.

\section{Proofs: Length spectral rigidity}

\subsection{Auxiliary results}

We denote by $A(r_1,r_0)=\bar B(0,r_1)\setminus B(0,r_0)\subset\RR^n$ the closed annulus in a Euclidean space.

\begin{lemma}
\label{lma:lspr-annulus}
Fix any $\eps>0$ and $r_1\in(0,1)$, and any finite set $F\subset(0,1)$.
Let $r(\tau)\in(0,1)$ depend $C^1$-smoothly on $\tau$.
Let $c_\tau$ with $\tau\in(-\eps,\eps)$ be $C^{1,1}$ functions $[r_1,1]\to (0,\infty)$ satisfying the Herglotz condition and the countable conjugacy condition and depending $C^1$-smoothly on $\tau$.

If $\partial_\tau c_\tau(r)\restriction_{\tau=0}\neq0$ for some $r\in(r_1,1)$, then there is a periodic broken ray $\gamma_\tau$ with respect to $c_\tau$ so that
\begin{itemize}
\item
$\tau\mapsto\ell_\tau(\gamma_\tau)$ is $C^1$ on $(-\delta,\delta)$ for some $\delta\in(0,\eps)$,
\item
$\partial_\tau\ell(\gamma_\tau)\restriction_{\tau=0}\neq0$,
and
\item
the depth (minimum of Euclidean distance to the origin) of $\gamma_0$ is not in $F$.
\end{itemize}
Here $\ell_\tau$ is the length functional corresponding to the velocity profile $c_\tau$.
\end{lemma}

While in our application we have $F=\emptyset$, we include this freedom in the lemma so that finitely many problematic depths can be avoided if needed.
\joonas{I added this.}

We say that a broken ray is \emph{radial} if it is contained in a one-dimensional linear (not affine) subspace of $\RR^n$.
\joonas{I'm not sure how to phrase this most clearly.}

\begin{lemma}
\label{lma:lspr-moving-interface}
Fix any $\eps>0$.
Let $c_\tau\colon(0,1]\to(0,\infty)$ be a family if $C^{1,1}$ functions depending smoothly on $\tau\in(-\eps,\eps)$.
Let $r(\tau)\colon(-\eps,\eps)\to(0,1)$ be $C^1$.

Let $\ell_\tau$ be the length of the radial geodesic between $r=r_1$ and $r=1$.
If $\partial_\tau c_\tau(r)\restriction_{\tau=0}=0$ for all $\tau\in(-\eps,\eps)$, then 
\[
\ell'(0)
=
c_0(r_1(0))^{-1}
r_1'(0).
\]
\end{lemma}

\subsection{Proof of Theorem~\ref{thm:blspr-multilayer}}

The idea of the proof is as follows:
We first show that $c_\tau$ is independent of $\tau$ within the first layer.
Then we show that the first interface is also independent of $\tau$.
After these steps we can ``peel off'' the top layer and repeat the argument for the second one.
Countability of the basic length spectrum provides sufficient decoupling between the layers and between the ``data'' $\blsp(\tau)$ and the ``noise'' $S(\tau)$.

We give most arguments at $\tau=0$ first for definiteness, but the exact value of the parameter is unimportant.

\begin{proof}[Proof of Theorem~\ref{thm:blspr-multilayer}]
Let us denote $f_\tau(r)=\partial_\tau c_\tau(r)$ and $\hat S(\tau)=\blsp(\tau)\cup S(\tau)$.

Take any $r\in(r_1(0),1)$.
If $\partial f_0(r)\neq0$, then by Lemma~\ref{lma:lspr-annulus} there is a family of basic periodic broken rays $\gamma_\tau$ for which the length map $\tau\mapsto\ell(\gamma_\tau)$ is $C^1$ in a neighborhood of $\tau=0$ and the derivative at $\tau=0$ is non-zero.

As $\ell(\gamma_\tau)\in \hat S(\tau)$ and by assumption $\hat S(\tau)=\hat S(0)$ for all $\tau$, this implies that the set $\hat S(0)$ contains a neighborhood of $\ell(\gamma_0)$.
This is in contradiction with countability of $\hat S(\tau)$, and so $\partial f_0(r)\neq0$ is impossible.

We conclude that $\partial f_\tau(r)=0$ for all $r\in(r_1(0),1)$.
The same argument can be repeated at any value of the parameter $\tau$, leading us to conclude that $\partial f_\tau(r)=0$ whenever $r\in(r_1(\tau),1)$.

If $r_1'(0)\neq0$, then by Lemma~\ref{lma:lspr-moving-interface} the radial broken rays (which are basic and periodic with period twice their length) there is a family of broken rays whose lengths vary differentiably and with a non-zero derivative at $\tau=0$.
This contradicts countability as above.
The same argument is valid for any value of $\tau$, so we conclude that $r_1'(\tau)=0$ for all $\tau\in(-\eps,\eps)$.

We have thus found that $r_1(\tau)=r_1(0)$ and $c_\tau(r)=c_0(r)$ for all $\tau\in(-\eps,\eps)$ and $r\in(r_1(0),1)$.
We may now turn our attention to the annulus $A(r_2(\tau),r_1(\tau))$, whose top interface is now fixed at $r=r_1(0)=r_1(\tau)$ for all $\tau$.
Repeating the same argument in this annulus shows that both the velocity profile in this annulus and the location of the second interface are independent of $\tau$.
Carrying on inductively, we exhaust all layers of the ball and find that the claim does indeed hold true.
\end{proof}

\subsection{Proofs of the lemmas}

Lemma~\ref{lma:lspr-annulus} is a small variation of the reasoning in \cite{HIKRigidity}, rewritten in a way that is useful in the presence of interfaces.
The proof is concise; the reader is invited to refer to \cite{HIKRigidity} for details.

\begin{proof}[Proof of Lemma~\ref{lma:lspr-annulus}]
Consider the velocity profile for any fixed $\tau$.
A maximal broken ray without reflections from the inner boundary component is determined uniquely up to rotation by its deepest point.
Let us denote the essentially unique geodesic of depth $r\in(0,1)$ by $\gamma_r^\tau$.
For a subset $P^\tau\subset(r_1,1)$ the corresponding broken rays are periodic, and we denote the minimal period by $\ell(\tau,r)$.

A periodic broken ray with respect to $c_0$ is called stable if there is $\delta\in(0,\eps)$ so that there is a family of paths $\gamma^\tau\colon\R\to A(1,r_1)$ which is $C^1$ in $\tau$ (and only continuously at reflection points) and each $\gamma^\tau$ is a periodic broken ray with respect to $c_\tau$.
When such a family exists, let us denote the depth corresponding to the parameter $\tau\in(-\delta,\delta)$ by $r^\tau$.
Let us denote by $C^0\subset P^0\subset(r_1,1)$ the set of depths of stable periodic broken rays.
It was shown in \cite{HIKRigidity} that under the countable conjugacy condition and the Herglotz condition the set $C^0$ is dense in $[r_1,1]$.
Thus also $C^0\setminus F$ is dense.

Let us denote $f(r)=\partial_\tau c_\tau(r)\restriction_{\tau=0}$.
Suppose that $f(r)\neq0$ for some $r\in(r_1,1)$.
Due to the injectivity of generalized Abel transforms, the function
\[
h(r)
=
\int_r^1
f(s)
\left[
1-
\left(
\frac{rc(s)}{sc(r)}
\right)^2
\right]^{-1/2}
\frac{\der s}{c(s)}
\]
is also non-trivial.
As $h$ is continuous and $C_0$ is dense, there is $r'\in C_0\setminus F$ so that $h(r')\neq0$.

The length $\ell(\tau,r^\tau)$ of the family of periodic broken rays is differentiable in $\tau$ near $\tau=0$ because $r'\in C^0$ and
\[
\partial_\tau
\ell(\tau,r^\tau)\restriction_{\tau=0}
=
2nh(r'),
\]
where $n$ is the (constant) winding number of the minimal period of $\gamma^\tau$.

Therefore the claimed derivative is indeed non-zero.
\end{proof}

The proof of Lemma~\ref{lma:lspr-moving-interface} is a straightforward calculation and the statement is geometrically intuitive, so we skip the proof.
The essential statement concerns simply the derivative of the length of a geodesic with respect to its endpoint.

%\begin{proof}[Proof of Lemma~\ref{lma:lspr-moving-interface}]
%...
%\end{proof}

%\begin{remark}
%\joonas{Do we want to do something with this or just throw away? I'm inclined to throw away. VITALY: Me too since this seems more relevant for the trace formula and that's not our focus}
%\textcolor{red}{TO BE WRITTEN AND THOUGHT:}
%We need to discuss degeneracy, as we may allow some.
%It was pointed out in \cite[Remark 2.4]{SRHIK} that it is enough that the primitive length spectrum does not degenerate.
%That should be enough here too, but we need to check what happens with a discrete but countable number of accumulation points (or gliding ray lengths).
%Perhaps we should define a ``gliding length spectrum'' that contains everything with glides and then require this to be disjoint from the primitive length spectrum.
%It sounds plausible to actually use the gliding spectrum as data as well, but it seems to be unnecessary with the current approach.
%\end{remark}

\section{The Trace formula and its proof}
As in \cite{HIKRigidity}, we will prove a trace formula in order to recover part of the length spectrum, and then use the argument in the previous sections on length spectral rigidity in order to prove Theorem \ref{t: spectral rigiditiy}.  
Although the main theorems as stated in subsection \ref{s: main results} refer to the scalar operator $\Delta_c$, for greater generality, we initially consider the toroidal modes corresponding to the isotropic elastic operator (see \cite{DahlTromp,HIKRigidity} for definitions). As in \cite{HIKRigidity}, the proof is identical when considering the scalar Laplace-Beltrami operator. This allows us to naturally consider and extend our results to spheroidal modes in section \ref{s: spheroidal modes} where two waves speed are present. 
First, we give the general setup and state the trace formula as Proposition \ref{prop:Trace Formula}, followed by its proof.

\subsection{Toroidal modes, eigenfrequencies, and trace formula}
\label{s: toroidal modes}
We now use spherical coordinates $(r,\theta,\psi)$. Toroidal modes are
precisely the eigenfunctions of the isotropic elastic operator that
are sensitive to only the shear wave speed. We forgo writing down the
full elastic equation, and merely write down these special
eigenfunctions connected to the shear wave speed. Analytically, these eigenfunctions admit a separation in
radial functions and real-valued spherical harmonics, that is,
\begin{equation}
   u = {}_n\mathbf{D}_l Y^m_l ,
\end{equation}
where
\begin{equation}
   \mathbf{D} = U(r)\ (- k^{-1})
      [-\widehat{\theta} (\sin \theta)^{-1} \partial_{\psi}
                + \widehat{\psi} \partial_{\theta}] ,
\end{equation}
in which $k = \sqrt{l (l + 1)}$ and $U$ represents a radial function
(${}_n U_l$). In the further analysis, we ignore the curl (which
signifies a polarization); that is, we think of ${}_n\mathbf{D}_l$ as
the multiplication with ${}_n U_l(-k^{-1})$. In the above, $Y^m_l$ are
spherical harmonics, defined by
\[
   Y^m_l(\theta,\psi) = \left\{ \ba{rcl}
   \sqrt{2} X^{\abs{m}}_l(\theta) \cos(m \psi) & \text{if}
                             & -l \le m < 0 ,
\\
   X^0_l(\theta) & \text{if} & m = 0 ,
\\
   \sqrt{2} X^m_l(\theta) \sin(m \psi) & \text{if} & 0 <  m \le l ,
   \ea\right.
\]
where
\[
   X^m_l(\theta) = (-)^m \sqrt{\frac{2l + 1}{4\pi}}
               \sqrt{\frac{(l-m)!}{(l+m)!}} P^m_l(\cos\theta) ,
\]
in which
\[
   P^m_l(\cos(\theta)) = (-)^m \frac{1}{2^l l!} (\sin\theta)^m
     \left( \frac{1}{\sin\theta} \frac{\mathrm{d}}{\mathrm{d}\theta}
     \right)^{l+m} (\sin\theta)^{2l} .
\]
The function, $U$ (a component of displacement), satisfies the
equation
\begin{equation}\label{eq: equation for U_2}
   [-r^{-2} \partial_r\ r^2 \mu \partial_r
       + r^{-2} \partial_r\ \mu r - r^{-1} \mu \partial_r
         + r^{-2} (-1 + k^2) \mu] \, U - \omega^2 \rho U = 0 ,
\end{equation}
where $\mu = \mu(r)$ is a Lam\'{e} parameter and $\rho = \rho(r)$ is
the density, both of which are smooth, and $c = \sqrt{\mu/\rho}$. Also, $\omega = {}_n\omega_l$
denotes the associated eigenvalue. Here, $l$ is referred to as the
angular order and $m$ as the azimuthal order.

The traction is given by
\begin{equation}\label{eq: Neumann condition for U_2}
   T(U) = \mathcal{N} U ,\qquad
   \mathcal{N} = \mu \p_r - r^{-1}\mu
\end{equation}
which vanishes at the boundaries (Neumann condition). The transmission conditions are that $U$ and $T(U)$ remain continuous across the interfaces.
 If $r = \refl$ is an interface and $U_{\pm}$ represent two solutions on opposite sides of the interface, then in the high frequency limit as $\omega \to \infty$, the transmission conditions will amount to
 \begin{align}\label{eq: trans conditions for U}
 U_+\restriction_{r = \refl} &=  U_-\restriction_{r = \refl} \\
 \mu_+ \p_rU_+\restriction_{r = \refl} &=  \mu_-\p_rU_-\restriction_{r = \refl}
 \end{align}
 for the principal terms in the WKB expansion of the solution.
 
 The radial
equations do not depend on $m$ and, hence, every eigenfrequency is
degenerate with an associated $(2l + 1)$-dimensional eigenspace
spanned by
\[
   \{ Y^{-l}_l,\ldots,Y^l_l \} .
\]
Following \cite{ZhaoModeSum}, let $d$ indicate the overtone number $n$ and the angular degree $l$. The radial eigenfunction $U_d(r)$ is independent of the order $m$. We define the inner product of the eigenfunctions:
\beq \label{I_d inner product}
{}_nI_l = I_d := \int_\CMB^\surf \abs{U_d(r)}^2 \rho(r) \dd r
\eeq

We use spherical coordinates $(r_0,\theta_0,\psi_0)$ for the location,
$x_0$, of a source, and introduce the shorthand notation
$({}_n\mathbf{D}_l)_0$ for the operator expressed in coordinates
$(r_0,\theta_0,\psi_0)$. We now write the (toroidal contributions to
the) fundamental solution as a normal mode summation
\begin{equation}\label{normal-mode-summation}
   G(x,x_0,t) = \operatorname{Re}\
          \sum_{l=0}^{\infty} \sum_{n=0}^{\infty}\
          {}_n\mathbf{D}_l ({}_n\mathbf{D}_l)_0\
   \sum_{m=-l}^l Y^m_l(\theta,\psi) Y^m_l(\theta_0,\psi_0)\
        \frac{e^{\ii {}_n\omega_l t}}{\ii ({}_n\omega_l) ({}_n I_l)} .
\end{equation}
On the diagonal, $(r,\theta,\psi) = (r_0,\theta_0,\psi_0)$ and, hence,
$\Theta = 0$.
Here $\Theta$ is the angular epicentral distance, %cf.~\eqref{eq: eq for big Theta}.
We observe the following reductions in the evaluation of
the trace of~\eqref{normal-mode-summation}:
%\\[-0.2cm]
\begin{itemize}
\item
 We will not normalize $U(r)$.
Meanwhile, the spherical harmonic terms satisfy
\begin{equation} \label{eq:Ylmnorm}
   \sum_{m=-l}^l \iint
      Y^m_l(\theta,\psi)^2 \sin \theta \dd\theta \dd\psi
           = 2l + 1
\end{equation}
(counting the degeneracies of eigenfrequencies).
\item
If we were to include the curl in our analyis (generating vector
spherical harmonics), taking the trace of the matrix on the diagonal
yields
\begin{equation} \label{eq:GYlmnorm}
   \sum_{m=-l}^l \iint
   (-k^{-2})
      \abs{[-\widehat{\theta} (\sin \theta)^{-1} \partial_{\psi}
                + \widehat{\psi} \partial_{\theta}]
       Y^m_l(\theta,\psi)}^2 \sin \theta \dd\theta \dd\psi
           = 2l + 1 .
\end{equation}
\end{itemize}

From the reductions above, we obtain
\begin{equation}
   \int_M
      G(x,x,t) \, \rho(x) \dd x
      = \sum_{l=0}^{\infty} \sum_{n=0}^{\infty}\
        (2l + 1)
        \operatorname{Re} \left\{
        \frac{e^{\ii {}_n\omega_l t}}{
                               \ii ({}_n\omega_l) }\right\}
\end{equation}
or
\begin{equation} \label{eq:TrptG}
  \Tr(\p_t G)(t) = \int_M
      \partial_t G(x,x,t) \, \rho(x) \dd x
      = \sum_{l=0}^{\infty} \sum_{n=0}^{\infty}\
        (2l + 1)
        \operatorname{Re} \left\{
        e^{\ii {}_n\omega_l t} \right\} .
\end{equation}
Let us also denote $\Sigma = \text{singsupp}( \Tr(\p_t G))\subset \mathbb{R}_t.$

\begin{comment}
We now write
\begin{equation*}
   {}_n f_l(t) = \operatorname{Re} \left\{
       \frac{e^{\ii {}_n\omega_l t}}{
                               \ii {}_n\omega_l} \right\}
\end{equation*}
which is the inverse Fourier transform of
\begin{equation}
   {}_n \hat{f}_l(\omega) = \frac{1}{2\ii {}_n\omega_l}
      \left[ \mstrut{0.5cm} \right.
      \pi \delta(\omega - {}_n\omega_l)
           - \pi \delta(\omega + {}_n\omega_l)
                              \left. \mstrut{0.5cm} \right] .
\end{equation}
Moreover, taking the Laplace--Fourier transform yields
\begin{equation} \label{eq:nflom}
   \int_0^{\infty} {}_n f_l(t) e^{-\ii \omega t} \, \dd t
   = \frac{1}{2\ii {}_n\omega_l}
    \left[ \mstrut{0.5cm} \right.
    \frac{\ii}{-(\omega - {}_n\omega_l) + \ii 0}
    - \frac{\ii}{-(\omega + {}_n\omega_l) + \ii 0}
            \left. \mstrut{0.5cm} \right] .
\end{equation}
This confirms that the trace is equal to the inverse Fourier transform
of
\begin{equation}
   \sum_{l=0}^{\infty} (2l + 1) \sum_{n=0}^{\infty}
   \frac{1}{2\ii ({}_n\omega_l)({}_nI_l)}
      \left[ \mstrut{0.5cm} \right.
      \pi \delta(\omega - {}_n\omega_l)
           - \pi \delta(\omega + {}_n\omega_l)
                              \left. \mstrut{0.5cm} \right] .
\end{equation}
\end{comment}
 
 \subsection{Connection between toroidal eigenfrequencies, spectrum of the Laplace--Beltrami operator, and the Schr\"{o}dinger equation} \label{sec: connect to LB}

We repeat the discussion in \cite{HIKRigidity} to relate the spectrum of a scalar Laplacian, the eigenvalues
associated to the vector valued toroidal modes, and the trace
distribution $\sum_{l=0}^{\infty} \sum_{n=0}^{\infty}\ (2l+1)
\cos(t{}_n\omega_l)$.

We note that \eqref{eq: equation for U_2} and \eqref{eq: Neumann
  condition for U_2} for $U$ ensure that $v = U Y^m_l$ satisfies
\begin{equation} \label{eq: scalar P}
   P v \coloneqq \rho^{-1}
       (-\nabla \cdot \mu \nabla + P_0) v = \omega^2 v ,\qquad
   \mathcal{N} v = 0 \text{ on }\p M
\end{equation}
where $P_0 = r^{-1}(\p_r\mu)$ is a $0$th order operator, $\omega^2$
is a particular eigenvalue, and $\mathcal N$ is as in \eqref{eq: Neumann condition for U_2}. Hence $U Y^m_l$ are scalar eigenfunctions
for the self-adjoint (with respect to the measure $\rho\dd x$) scalar
operator $P$ with Neumann boundary conditions (on both boundaries)
expressed in terms of $\mathcal{N}$.

The above argument shows that we may view the toroidal spectrum
$\{{}_n\omega^2_l\}_{n,l}$ as also the collection of eigenvalues
$\lambda$ for the boundary problem on scalar functions \eqref{eq:
  scalar P}. Thus \eqref{eq:TrptG} can be written in the form
\begin{equation}\label{eq: toroidal trace equal laplace trace}
   \Tr \, (\p_t G)
          = \sum_{\lambda \in \spec(P)} \cos(t \sqrt{\lambda}) ,
\end{equation}
where the last sum is taken with multiplicities for the
eigenvalues. (While $G$ is a vector valued distribution, the
asymptotic trace formula we obtain is for $\Tr (\p_t G)$, which is equal
to $\sum_{\lambda \in \spec(P)} \cos(t \sqrt{\lambda})$ by the
normalizations we have chosen.) Up to principal symbols, $P$ coincides
with the $\Delta_c = c^3 \nabla \cdot c^{-1} \nabla$ upon identifying
$c^2$ with $\rho^{-1} \mu$. This means that the length spectra of $P$
and $\Delta_c$ will be the same even though they have differing subprincipal symbols and spectra. Thus, the trace formula which will
appear to have a unified form, connects two different spectra to a
common length spectrum and the proof is identical for both.

We will prove a trace formula using a WKB expansion of
eigenfunctions. To this end, it is convenient to establish a
connection with the Schr\"{o}dinger equation. Indeed, we present an
asymptotic transformation finding this connection. In boundary normal
coordinates $(r,\theta)$ (which are spherical coordinates in dimension
three by treating $\theta$ as coordinates on the $2$-sphere),
\begin{equation}
   P = \rho^{-1} (-r^{-2} \p_r r^2 \mu \p_r
                  - \mu r^{-2} \Delta_\theta + P_0) ,
\end{equation}
where $\Delta_\theta$ is the Laplacian on the $2$-sphere.

Let us now simplify the PDE \eqref{eq: scalar P} for $v$.
Let $Y(\theta)$ be an eigenfunction of $\Delta_\theta$ with eigenvalue
$-k^2$ as before and $V = V(r):= \mu^{1/2}r U$ a radial function with $U$ as in \eqref{eq: scalar P}. Then after a straightforward calculation, as a leading order term in a WKB expansion, $V(r)$ must satisfy
\begin{equation}\label{eq:VSchro}
   \p_r^2 V + \omega^2 \wkb^2 V = 0 ,\quad
   \p_r V = 0\ \text{ on }\ \p M ,
\end{equation}
 with transmission conditions for $V$ to leading order 
 \begin{align}\label{eq: trans conditions for V}
 \mu_+^{-1/2} V_+\restriction_{r = \refl} &=  \mu^{-1/2}_-V_-\restriction_{r = \refl} \\
 \mu_+^{1/2} \p_rV_+\restriction_{r = \refl} &=  \mu_-^{1/2}\p_rV_-\restriction_{r = \refl},
 \end{align}
where $\wkb^2 = \rho(r) \mu(r)^{-1} - \omega^{-2}r^{-2}k^2$ and $\{ r = b\}$ is an interface, generating
two linearly independent solutions. The WKB asymptotic solution to
this PDE with Neumann boundary conditions will precisely give us the
leading order asymptotics for the trace formula, and is all that is
needed.

For the boundary condition, we note that we would end up with the same
partial differential equation with different boundary conditions for
$V$ in the previous section if we had used the boundary condition
$\p_r u = 0 \text{ on }\p M$. Indeed, one would merely choose
$\mathcal{N}u = \mu \p_r u$ instead without the $0$th order
term. However, the boundary condition for $V$ would be of the form
\begin{equation}
   \p_r V = K(r) V\quad\ \text{ on }\ \p M
\end{equation}
with $K$ signifying a smooth radial function. Nevertheless, the
leading order (in $\omega$) asymptotic behavior for $V$ stays the same
despite the $K$ term as clearly seen in the calculation of Appendix \ref{a: Generalized Debye}. Thus, our analysis applies with no
change using the standard Neumann boundary conditions. This should
come as no surprise since in~\cite{Mel79}, the $0$'th order term in the
Neumann condition played no role in the leading asymptotic analysis of
their trace formula. Only if one desires the lower-order terms in the
trace formula would it play a role.

In addition, we could also consider a Dirichlet boundary condition, where for $V$, it is also $ V = 0$ on $\p M$. This would slightly modify the Debye expansion in Appendix \ref{a: Generalized Debye} by constant factors. Nevertheless, the same argument holds to obtain the trace formula and recover the length spectrum. More general boundary conditions such as Robin boundary conditions may be considered as well. However, since we only need to look at the principal term in the high frequency asymptotics, this would just reduce to the Neumann boundary case. Thus, our arguments work with all these boundary conditions, and we choose Neumann boundary conditions only because it has a natural interpretation from geophysics.

An interesting feature of the trace formula in this setup is that a broken ray $\gamma$ can have \branch{}s that glide along the interface. This happens when a reflected ray hits an interface at a critical angle leading to a transmitted \branch{} that glides along the interface. Technically, such a ray is \emph{not} a broken geodesic of the metric $g$, but it will be a limit of periodic broken geodesics as shown in section \ref{s: gliding as limits} and makes a contribution to the singular support of the trace as an accumulation point.

Since the length spectral rigidity theorems only require the \simp{} length spectrum,
the main goal is to determine the leading contribution of \simp{} rays without gliding \branch{}s to the trace. 
\begin{proposition}\label{prop:Trace Formula}(Non-gliding case)
Suppose the radial wave speed $c$ satisfies the extended Herglotz condition and the periodic conjugacy condition (definition \ref{I_d inner product}).

 Suppose $T = T(\pg) \in \lsp(c)$ corresponds to a periodic ray $\gamma$ with ray parameter $\pg$ such that no periodic ray with a gliding \branch{} has period $T$.
  Then there exists a neighborhood of $T$ such that, the leading order singularity of $(\Tr(\p_t
G))(t)$ near $T(\pg)$ is the real part of
\begin{equation}\label{eq: trace coefficient}
  \sum_{[\gamma]} (t-T(\pg)+ \ii 0)^{-5/2} \left(\frac{1}{2\pi \ii}\right)^{3/2}
   \ii^{N(\pg)}n(\pg)Q(\pg)\abs{p^{-2}_\gamma \p_p\alpha_\gamma(p_\gamma)}^{-1/2}
        L(p_\gamma)c \abs{SO(3)} ,
\end{equation}
%\abs{I-P_{[\gamma]}}^{-1/2}
where
\begin{enumerate}
\item[$\bullet$] the sum is taken over all equivalence classes $[\gamma]$ with period $T(\pg)$ and ray parameter $p_\gamma = p_{[\gamma]}$.
\item[$\bullet$] $N(\pg)$ is the Keller-Maslov-Arnold-H\"{o}rmander (KMAH) index associated to
  $\gamma$;
\item[$\bullet$] $c$ independent of $[\gamma]$;
\item[$\bullet$] $\abs{SO(3)}$ is the volume of the compact Lie group
  $SO(3)$ under the Haar measure.
  \item[$\bullet$]$Q(\pg)$ is a product of reflection and transmission coefficients of the corresponding broken ray. 
 \item[$\bullet$] $n(\pg) \in \mathbb N$ is a combinatorial constant counting the number of dynamic analogs of $\gamma$.
\end{enumerate}

Moreover, if the principal amplitude injectivity condition holds, the distribution $(\operatorname{Tr} \, (\p_t G))(t) =
\sum_{n,l}(2l+1)\cos(t{}_n\omega_l)$ is singular at the lengths of periodic \simp{} rays.
% False since multiple rays can have the same period; they just dont cancel each other due to PAI condition
\begin{comment}
Suppose $T \in \blsp(c)$ and $\gamma$ is a periodic \simp{} ray of period $T$ and has ray parameter $p_\gamma$. Then there exists a neighborhood of $T$ such that, to leading order in singularity, $(\Tr(\p_t
G))(t)$ is the real part of
\begin{equation}\label{eq: basic trace coefficient}
   (t-T+ \ii 0)^{-5/2} \left(\frac{1}{2\pi \ii}\right)^{3/2}
   \ii^{\sigma_{\gamma}}n_{[\gamma]}Q_{[\gamma]}|p^{-2}_\gamma \p_p\alpha_\gamma(p_\gamma)|^{-1/2}
        L(p_\gamma)c_d \abs{SO(3)} ,
\end{equation}
\end{comment}

\end{proposition}

\begin{remark} 
Our proof will show that one may obtain the leading order contribution of $\gamma^l$, which is $\gamma$ traversed $l$ times, from the above expression for $\gamma$. The contribution from $[\gamma^l]$ will be
\begin{multline}
   (t-lT(\pg)+ \ii 0)^{-5/2} \left(\frac{1}{2\pi \ii}\right)^{3/2}
   \ii^{l N(\pg)}n^l(\pg)Q^l(\pg)\abs{p^{-2}_\gamma l\p_p\alpha_\gamma(p_\gamma)}^{-1/2}
   \\
        \cdot L(p_\gamma) c_d \abs{SO(3)}
\end{multline}
\end{remark}

\begin{comment}
\joonas{Do the harmonics have a nice shape? If the primitive period is $T$, is the singularity with multiplicity $k$ of the form $f_k(t)=(t-T)^{-\alpha} \cdot a \dot b^k \cdot (1-p^k)^{-1/2}$ or something?}
\end{comment}

\begin{remark}
Note the above trace formula is almost identical to that of \cite{HIKRigidity} except for the $Q(\pg)$ term. This is natural since a wave corresponding to a periodic broken bicharacteristic in this nonsmooth case will have a principal symbol containing transmission and reflection coefficients while the rest of the principal symbol remains the same. The KMAH index also differs slightly than the nonsmooth case when a turning ray grazes an interface.
\end{remark}

\begin{remark} 
Similar to remark 2.5 in \cite{HIKRigidity}, our trace formula holds in an annulus where the boundary is not geodesically convex unlike the case in \cite{Mel79}. Hence, there could be periodic \emph{grazing rays} at the inner boundary of the annulus or rays that graze an interface. As described in \cite{TaylorGrazing}, grazing rays are bicharacteristics that intersect the boundary of a layer tangentially, have exactly second order contact with the boundary, and remain in $\bar M$. This is another reason our proof is via a careful study of the asymptotics of the eigenfunctions rather than the parametrix construction appearing in \cite{Mel79}, where the presence of a periodic grazing ray would make the analysis significantly more technical (cf. \cite{TaylorGrazing,MelroseGliding}). The spherical symmetry essentially allows us to construct a global parametrix (to leading order) to obtain the leading order contribution of a periodic grazing ray to the trace, which would be more challenging in a general setting (see Appendix \ref{a: Generalized Debye} and \ref{a: grazing} for the analysis and \cite{bennett1982poisson} for a similar computation). The leading order contribution of the grazing ray has the same form as in the above proposition, but the lower order contributions will not have this ``classical'' form since stationary phase cannot be applied to such terms, and will instead involve Airy functions as in \cite{bennett1982poisson} and \cite[Appendix B]{HIKRigidity}. 
Nevertheless, we note that for the main theorems, we do not need to recover the travel time of a periodic grazing ray if one exists. Travel times of sufficiently many non-grazing basic rays suffice. Our methods also produce a precise trace formula where periodic orbits are no longer simple as in \cite{Mel79}, but come in higher dimensional families (see \cite{GuillLieGroups,Creagh91,Creagh92,Gornet} for related formulas albeit in different settings). 
\end{remark}

We showed in section \ref{s: gliding as limits} that a ray with a gliding \branch{} is a limit of broken non-gliding rays, and we can also describe its contribution to the singular support to leading order. Let $\gamma$ be a periodic broken ray with travel time $T$ and contains a gliding \branch{} (see \cite[Figure 4.1]{CervenyHeadWaves} for a diagram of such a ray in the piecewise constant wavespeed setting). By lemma \ref{l: gliding}, there is a sequence of non-degenerate closed broken rays $\gamma_n$ with travel time $T_n$ such that $T_n \nearrow T$ and $\gamma_n$ converges to $\gamma$. We will state our trace formula near gliding rays in the same form as \cite[Theorem (42)]{Bennett1982}.
Denote $a_n = a_{n,[\gamma_n]}$ to be the coefficient in \eqref{eq: trace coefficient} in front of $(t-T_n+i0)^{-5/2}$ corresponding to ray $\gamma_n$. We assume that there are no periodic broken rays with travel time $T$ besides for $\gamma$ and its image under the group action. Let us introduce the notation for any real number $s$,
\[
H^{s-}_{loc} = \{ f: f \in H^t_{loc}(\RR) \text{ for } t<s\}.
\]
We will prove the following proposition.
\begin{proposition}\label{t: gliding ray trace}
Let $T$ be as above, and let $J$ be a small enough interval containing $T$   such that $\lsp(c) \cap J = \{T_n\}_{n=1}^\infty \cup \{T\}$.

Then 
\[
\operatorname{Tr} \, (\p_t G)(t))\restriction_J = \Re \sum_{n=1}^\infty a_n(t-T_n+ \ii 0)^{-5/2} + R(t),
\]
where $R(t)$ is a distribution that lies in the Sobolev space $H^{-2+\delta}$ for some $\delta > 0$.
\end{proposition}
Note that this is a genuine error estimate even though we do not have a sharp result on which Sobolev space contains $R(t)$ since the sum in the formula above lies in $H^{-2-}_{loc}$. Proposition \ref{t: gliding ray trace} is not needed for spectral rigidity  and will be proved in appendix \ref{a: proof of gliding}. 

Also, implicit in the above proposition is that away from the singularities, the infinite sum converges. It is not clear which Sobolev space $R(t)$ belongs to since we only compute the principal term in the trace (which appears as the sum in the above proposition) using stationary phase, and we show that the remainder is in a weaker Sobolev space even though we cannot use stationary phase for it.

 In fact, it is not even clear whether a term of the form $(t-T+i0)^{-\epsilon}$ appears in $R(t)$. Denote $Z(t) = \text{Tr}(\p_t G)(t)$. Then for small enough $\epsilon > 0$, 
 $(T-\epsilon, T) \cap \lsp(c) = \{T_n\}_{n=1}^\infty$
 while $(T, T+\epsilon) \cap \lsp(c) = \emptyset$. Thus $\text{Re} Z(t)$ is $C^\infty$ for $t \in (T,T+\epsilon)$, and it becomes an interesting question, what is the asymptotic behavior of $Z(t)$ as $t \to T$ from the right? This is subtle and
 Colin de Verdi\'{e}re (see \cite{ColinClusterpoints,ColinGliding}) showed how in certain, simpler examples than what we consider here, $Z(t)$ is actually $C^\infty$ on $[T, T+\epsilon)$ for some $\epsilon$.
 Thus, the trace is actually smooth from the right up to and including $T$ (it is obviously not smooth from the left). \v{C}erven\'{y} points out in \cite{CervenyHeadWaves} that the contribution of the singularity precisely at $T$ cannot be investigated with ray theory in this setting, and it remains an open question of the precise nature of this singularity. However, in our computations of the principal term in the WKB expansion, it is not present, which is how we know it can only be in a lower order term, if it is there at all.

The trace formula allows us to recover the \simp{} length spectrum from the spectrum, and then apply the theorems on length spectral rigidity to prove Theorem \ref{thm:blspr-multilayer}.

\subsection{Proof of the trace formula} \label{s: proof of trace formula}

We need several preliminary computations before proving proposition \ref{prop:Trace Formula}. The key to the trace formula is the Debye expansion that will give geometric meaning to the leading order amplitudes of the radial eigenfunctions.  A key step will be a derivation for an alternative way of expressing $I_d$ in \eqref{I_d inner product}.
\subsubsection{A key formula for the Green's function}
As pointed out in \cite{ZhaoModeSum}, the inner product $I_d$ can be expressed in terms of the derivatives of a quantity involving the radial eigenfunctions $U_d(r)$ as well as their radial derivatives with respect to frequency $\omega$. We repeat the argument here to show that it holds even when the PDE parameters have discontinuities.

The key is obtaining a special formula for $\langle U_n, U_n \rangle$ shown in \cite{Singh69}. We recall the ordinary differential equation \eqref{eq: equation for U_2} for the radial constituent of the eigenfunction:
\begin{equation}\label{e: ode for U}
\p^2_rU + \left(\frac{2}{r} + \mu^{-1}\p_r \mu\right)\p_r U
+ \left[\omega^2-\frac{1}{r\mu} - \frac{k^2}{r^2}\right]U = 0.
\end{equation}
Here $U= U_k = U_l$ denotes the above solution for general $\omega$ while $U_n$ is a solution for such $\omega_n = {}_n\omega_l$ such that $T(U_n) = \mu(\p_r -r^{-1})U_n = 0$ at $r = \surf $ and $r = \CMB$.
It will be convenient to write
\begin{equation}\label{e: ode for U_n}
\p^2_rU_n + \left(\frac{2}{r} + \mu^{-1}\p_r \mu\right)\p_r U_n
+ \left[\omega_n^2-\frac{1}{r\mu} - \frac{k^2}{r^2}\right]U_n = 0
\end{equation}
Multiply \eqref{e: ode for U} by $U_n$ and \eqref{e: ode for U_n} by $U$ and subtract the two equation to get
\[
U_n\p^2_r U - U \p^2_rU_n
+ \left(\frac{2}{r} + \mu^{-1}\p_r \mu\right)(U_n\p_r U - U\p_r U_n)
+ \rho/\mu(\omega^2 - \omega_n^2)U U_n = 0
\]
which may be simplified to 
%more steps
\begin{comment}(by first multiplying the above equation by $r^2\mu$)
\[
\frac{d}{dr} \left[ r^2 \mu (U_n \p_r U - U \p_r U_n)\right]
= \rho r^2(\omega_n^2 - \omega^2)U_n U
\]
We then note that $\mu (U_n \p_r U - U \p_r U_n) = U_n T - U T_n$
so that
\end{comment}
\[
\frac{d}{dr} \left[ r^2 (U_n T - U T_n)\right]
= \rho r^2(\omega_n^2 - \omega^2)U_n U.
\]
We integrate over $(\CMB, \surf)$ to obtain
\[
 \frac{\left[ r^2 (U_n T - U T_n)\right]_{r=\CMB}^\surf}{\omega_n^2-\omega^2}
 = \int_\CMB^\surf r'^2 \rho(r') U(r')U_n(r') \dd r'.
\]
Above, we use that $U,U_n,T,T_n$ are continuous across the interface to apply the fundamental theorem of calculus.
Let us suppose $\omega$ is not an eigenfrequency and then take the limit as $\omega \to \omega_n$. Let
\beq \label{e: introducing D}D := [ r^2 (U_n T - U T_n)]_{r=\CMB}^\surf=
[ r^2 U_n T ]_{r=\CMB}^\surf
 \eeq
 using the Neumann conditions. Note that the solutions to $D = 0$ are precisely the eigenvalues ${}_n \omega_l$ determined by the Neumann boundary conditions. A key fact is that even for such general solutions, we can enforce the inner boundary condition $T(U)\restriction_{r = \CMB} = 0$ to leading order while still keeping $\omega$ generic. This simplifies the computations so that
 \beq \label{e: D only at outer boundary}
 D =
[ r^2 U_n T ]_{r=\surf}.
 \eeq
 Then by L'Hospital's rule using the limit $\omega \to \omega_n$, we obtain
\[
\int_\CMB^\surf r'^2 \rho(r') U_n(r')U_n(r') \dd r'
= -\frac{(\p_\omega D)_{\omega_n}}{2\omega_n}.
\]

Next we recall
\[
G(x,x_0,t) = \frac{1}{2\pi}\
          \sum_{l=0}^{\infty} \sum_{n=0}^{\infty}\
     (l + \tfrac{1}{2})\ \frac{\sin({}_n\omega_l t)}{{}_n\omega_l I_l}\ \underbrace{
        {}_n\mathbf{D}_l ({}_n\mathbf{D}_l)_0}_{\eqqcolon {}_n H_l}\
                   P_l(\cos \Theta)
\]
Where $I_d = I_{n,l}$ is equal to $l(l+1) \int_{r=\CMB}^\surf \rho r^2 U_n^2 \dd r.$

What we have shown is that
\begin{equation}\label{e: I_l in terms of D}
I_l = -\frac{l(l+1)}{2{}_n\omega_l} \left( \frac{\p D}{\p \omega}\right)_{{}_n \omega_l}
\end{equation}
so the Green's function becomes
\[
G(x,x_0,t) = -\frac{1}{\pi}\
          \sum_{l=0}^{\infty} \sum_{n=0}^{\infty}\
     \frac{l + \frac{1}{2}}{l(l+1)}\ \frac{\sin({}_n\omega_l t)}{\left( \tfrac{\p D}{\p \omega}\right)_{{}_n \omega_l}}\
        {}_n\mathbf{D}_l ({}_n\mathbf{D}_l)_0\
                   P_l(\cos \Theta).
\]
Next, observe that ${}_n \omega_l$ are exactly the zeros of $D$ so we can replace the sum over $n$ by a complex line integral over $\omega$.
First use $\text{Re} \tfrac{e^{-i\omega t}}{i} = \sin (\omega t)$. Then for fixed $l$, we compute as in \cite{ZhaoModeSum}
\begin{equation}
   \sum_{n=0}^{\infty}\
     \frac{\sin({}_n\omega_l t)}{\left( \tfrac{\p D}{\p \omega}\right)_{{}_n \omega_l}}\
        {}_n\mathbf{D}_l ({}_n\mathbf{D}_l)_0
       = -\frac{1}{2\pi}\text{Re}\int_{-\infty}^\infty D^{-1} \mathbf{D}_l (\mathbf{D}_l)_0 e^{-i\omega t}\dd \omega
\end{equation}
where the residue at $\omega = {}_n \omega_l$ of the integrand is calculated via
\[
\lim_{\omega \to {}_n \omega_l} \frac{w-{}_n\omega_l}{D} \mathbf{D}_l (\mathbf{D}_l)_0 e^{-i\omega t}
\]
and one uses L'Hospital's rule to get the desired formula. As in \cite{ZhaoModeSum}, the lack of a prefix $n$ on $U_l(r)$ and $U_l(r')$ indicates that these are general solutions which \emph{do not necessarily} satisfy the free-surface boundary conditions although \emph{we are enforcing the inner boundary condition.}

\begin{remark}
We note that \cite{HIKRigidity} also used residue theory to compute the infinite sum over $n$. However, the argument would not readily apply here since ${}_n\omega_l$ is more complicated in our case, so we employ a trick to circumvent using the equations involving ${}_n \omega_l$, which cannot be solved explicitly.
\end{remark} 

Thus, we have managed to write $G$ as the Fourier transform in $\omega$ of $D^{-1} \mathbf{D}_l (\mathbf{D}_l)_0$. Taking the inverse of the transform, we obtain
\beq \label{e: hat G with general efunctions}
\hat G(x,x_0,\omega) = \frac{1}{2\pi}\sum_{l=0}^\infty
\frac{l + \tfrac{1}{2}}{l(l+1)} D^{-1} \mathbf{D}_l (\mathbf{D}_l)_0  P_l(\cos \Theta).
\eeq
This corresponds with the residue theory in \cite{HIKRigidity} to calculate the infinite series over $n$.

\subsubsection{Poisson's formula for the Green's function}
We abuse notation and denote
\[
H(k) = k^{-2}U_l(r)U_l(r')
\]
in the formula for $G$ to not treat the curl operations at first. This will not cause risk of confusion since we will specify the exact moment we apply the curl operators. Note that $U_l$ does not necessarily satisfy the Neumann boundary conditions.

\begin{proof}[Proof of proposition \ref{prop:Trace Formula}]

By the identical argument in \cite[Appendix A]{HIKRigidity}, we use \emph{Poisson's formula} to rewrite $\hat G(x,x_0,\omega)$ in a different form.

\begin{multline} \label{eq:PSS-poss}
   \hat{G}(x,x_0,\omega) =  
   \\ \frac{1}{2\pi}\
          \sum_{s=1}^{\infty}\ (-)^s
     \int_0^{\infty} \left[
      \ D^{-1}\ H(k) \right]
          P_{k - 1/2}(\cos \Theta)
        \{ e^{-2 \ii s k \pi} + e^{2 \ii s k \pi} \} \, k \dd k
\\
   + \frac{1}{2\pi}\ \int_0^{\infty} \left[
       D^{-1}\ H(k) \right]
          P_{k - 1/2}(\cos \Theta) \, k \dd k .
\end{multline}

Note that $H(k)$ has the general eigenfunctions that do not necessarily satisfy Neumann boundary conditions.
We substitute $k = \omega p$ so $k \dd k = p^{-1} \dd p$ and the above expression becomes (see \cite[Appendix A]{HIKRigidity} for details)

\begin{multline}\label{e: G with p integral}
   \hat{G}(x,x_0,\omega) = \frac{1}{2\pi}\
          \left[ \mstrut{0.6cm} \right.
   \sum_{s = 1,3,5,\ldots} (-)^{(s-1)/2}
\\
   \int_0^{\infty} \left[
    D^{-1}\ H(\omega p) \right]
          Q_{\omega p - 1/2}^{(1)}(\cos \Theta)
          \{ e^{-\ii (s-1) \omega p \pi} - e^{\ii (s+1) \omega p \pi} \}
                   \, p^{-1} \dd p
\\
   + \sum_{s = 2,4,\ldots} (-)^{s/2}
     \int_0^{\infty} \left[
     D^{-1}\ H(\omega p) \right]
          Q_{\omega p - 1/2}^{(2)}(\cos \Theta)
          \{ e^{-\ii s \omega p \pi} - e^{\ii (s-2) \omega p \pi} \}
                   \, p^{-1} \dd p  \left. \mstrut{0.6cm} \right],
\end{multline}
where $Q^{(j)}_k (\cos \Theta)$ are certain travelling wave Legendre functions described in \cite[Appendix A]{HIKRigidity}.

To obtain the leading order asymptotics of the above formula, we will eventually employ the method of steepest descent. Before doing so, we will obtain an expression for $U_k(r)$ that has a geometric meaning representing all the multiple scattering of a single ray interacting with not just the boundary, but the interfaces as well. In the Appendix \ref{a: Generalized Debye}, we obtain a Debye series expansion of the $D^{-1} H(\omega p)$ term in the above sum.

After a lengthy computation in Appendix \ref{a: Generalized Debye}, we write down the updated formula for a single term in the sum over $s$ in the Green's function from \eqref{eq: wave propagator form N interface in app} when $r$ and $r_0$ are in the same layer,

\begin{multline}\label{eq: wave propagator form N interface}
   \simeq \frac{1}{4\pi} (-)^{(s-1)/2}
       (r r_0 c^{(+)}(r) c^{(+)}(r_0))^{-1}
       (2\pi \rho^{(+)}(r) \rho^{(+)}(r_0) \sin \Theta)^{-1/2}
\hspace*{3.0cm}
\\
   \int (\wkb_+(r;p) \wkb_+(r_0;p))^{-1/2}
   \sum_{M \in \mathbb{Z}_{\geq 0}^{4(n-4)}}n_M(p) \cdot
   \\
   \sum_{i=1}^{4}
           \exp[-\ii \omega (\tau_{M,i}(r,r_0;p) + p \Theta + (s-1) p \pi)]Q_{M,i}(p)
\\
   \exp[\ii (\pi/4) (2 N_{M,i} - 1)] (\omega p)^{-3/2} \dd p,
\end{multline}
where the formula is nearly identical to that of \cite{HIKRigidity} with several key differences that encode (using the multiindex $M$) the amplitude and broken ray path consisting of reflecting/transmitting \branch{}s. Each component of $M$ indicates the number of reflected or transmitted \branch{}s of the ray in a particular layer.
First, $Q_{M,i}(p)$ is the leading amplitude of the wave, which is a product of reflection and transmission coefficients corresponding to the \branch{}s of a ray connecting two points at $r$ and $r_0$ with epicentral distance $\Theta$ (see \eqref{e: Q_M term} and \eqref{e: defining Q_M,i}), and $n_M$ is a combinatorial coefficient counting the dynamic analogs of this ray. 
Here, $\tau_{M,i}(r,r_0;p)$ is the radial component of the travel time of a broken ray with ray parameter $p$ that connect two points at $r$ and $r_0$. It is the sum of the radial travel times of each of the reflected and transmitted \branch{}s of the ray (see \eqref{e: Phi_m radial travel time} and \eqref{tau_M formula N interface}). Hence, $\tau_{M,i}$ and $Q_{M,i}$ encode the phase and amplitude (with all the reflections/transmission) of the wave associated to a particular ray. The index $i=1, \dots, 4$ corresponds to different ray paths with zero or one reflections connecting the source and receiver located at the radii $r$ and $r_0$ analogous to \cite{ZhaoModeSum}; once we take the trace, and apply the method of steepest descent, only the terms with $i =1,4$ make a contribution to the leading order asymptotics. Moreover, when taking the trace, the terms with $i=1$ and $i=4$ are identical so we will drop the subscript $i$.
Also, $N_{M,i}= N_{M,i}(p)$ is the KMAH index which is piecewise constant depending on the value of $p$ and is also affected by a ray grazing an interface.

\subsubsection*{Method of steepest descent} As in \cite[Section 3.2]{HIKRigidity}, we carry out the method of steepest descent in the integration over $p$. At this point, the argument is identical so we will be brief. Considering \eqref{eq: wave propagator form N interface}, we interchange the order of summation and integration, and invoke the
method of steepest descent in the variable $p$. Also notice that the path of integration is beneath the real axis, while taking $\omega > 0$. We carry out the analysis for a single
term, $s=1$. For $s=2,4,\ldots$ we have to add $s p \pi$ to $\tau_{M,i}$,
and for $s=3,5,\ldots$ we have to add $(s-1) p \pi$ to $\tau_{M,i}$,
in the analysis below.

Considering
\[
\phi_{M,i,s=1}=\phi_{M,i}(p) = \phi_{M,i}(r,r_0,\Theta,p):= \tau_{M,i}(r,r_0;p) + p \Theta
\]
as the phase function (for $s=1$) and $\omega$ as a large parameter,
we find (one or more) saddle points for each $i$, where
\[
   \partial_p \tau_{M,i}(r,r_0,p)\restriction_{p = p_k} = -\Theta .
\]
Later, we will consider the diagonal, setting $r_0 = r$ and $\Theta =
0$. We label the saddle points by $k$ for each $M,i$ (and $s$). We note
that $r, r_0$ and $\Theta$ determine the possible values for $p$
(given $M$,$i$ and $s$) which corresponds with the number of rays connecting the
receiver point with the source point (allowing conjugate points). Hence, there can be multiple saddle points for a fixed $M,i,s,r,r_0,\Theta$. For
$s=1$, the rays have not completed an orbit. With $s = 3$ we begin to
include multiple orbits.

We then apply the method of steepest descent to \eqref{eq: wave propagator form N interface} with a contour deformation as in \cite[Section 3.2]{HIKRigidity} and we obtain

\begin{comment}
\begin{equation} \label{eq:abws}
\begin{aligned}
   &\frac{1}{4\pi}  (-)^{(s-1)/2}
       (r r_0 c(r) c(r_0))^{-1}
       (2\pi \rho(r) \rho(r_0) \sin \Theta)^{-1/2}
\hspace*{3.0cm}
\\
   &\qquad \qquad \int (\wkb(r;p) \wkb(r_0;p))^{-1/2}
   \sum_{M \in \mathbb{Z}_{\geq 0}^{4(n-4)}}n_M \cdot 
   \\ & \qquad \qquad \qquad \sum_{i=1}^{4}
           \exp[-\ii \omega (\tau_{M,i}(r,r_0;p) + p \Theta + (s-1) p \pi)]
\\
   &\qquad \qquad \qquad \qquad
   \exp[\ii (\pi/4) (2 N_{M,i} - 1)]  (\omega p)^{-3/2} Q_{M,i}(p)\dd p
\\[0.25cm]
   &\hspace{1in}\simeq \frac{1}{4\pi} (-)^{(s-1)/2}
       (r r_0 c(r) c(r_0))^{-1}
       (2\pi \rho(r) \rho(r_0) \sin \Theta)^{-1/2}
\\[0.2cm]
   &\qquad \sum_{M \in \mathbb{Z}_{\geq 0}^{4(n-4)}}n_M\sum_{i=1}^{4} \sum_k
   \left[ \omega^{-2}p^{-3/2} (\wkb(r;.) \wkb(r_0;.))^{-1/2}
         \abs{\partial_p^2 \tau_{M,i}(r,r_0;.)}^{-1/2} \right.
         \\& \hspace{1 in}\qquad \qquad \qquad Q_{M,i}(p)\Big]_{p = p_k}
\hspace{0in}
                 \exp[-\ii \omega T_{Mik} + \ii \tilde N_{Mik} (\pi/2)] ,
\end{aligned}
\end{equation}
\end{comment}
\begin{multline}
\simeq -\frac{2\pi}{(2\pi i)^{3/2}} (-)^{(s-1)/2}
       (r r_0 c(r) c(r_0))^{-1}
       (\rho(r) \rho(r_0))^{-1/2}
\\[0.2cm]
  \sum_{M \in \mathbb{Z}_{\geq 0}^{4(N-4)}}  n_M\sum_{i=1}^4 \sum_k
   \left[ p (\wkb(r;.) \wkb(r_0;.))^{-1/2}
         \abs{\partial_p^2 \tau_{M,i}(r,r_0;.)}^{-1/2}Q_{M,i}(p) \right]_{p =
         p_k}
\\
               \frac{1}{2\pi}\int_0^{\infty}i\omega^{3/2}  \exp[-\ii \omega(
               T_{Mik}-t) + \ii \tilde N_{Mik} (\pi/2)] \dd\omega,
\end{multline}
as $\Theta \to 0$,
where
\begin{equation}
\begin{aligned}
   T_{Mik} &= T_{s;Mik}(r,r_0,\Theta)
       = \tau_{M,i}(r,r_0;p_k) + p_k \Delta_s ,\\
   \tilde N_{Mik} &= N_{M,i} - \tfrac{1}{2} (1 -
               \sgn \partial_p^2 \tau_{M,i} \restriction_{p = p_k}) ,
\end{aligned}
\end{equation}
in which
\begin{equation}
   \Delta_s =
\begin{cases}
   \Theta + (s-1) \pi & \text{if $s$ is odd} \\
   -\Theta + s \pi    & \text{if $s$ is even.}
                   \end{cases}
\end{equation}
The $\tilde N_{Mik}$ contribute to the KMAH indices, while the $T_{Mik}$ represent
geodesic lengths or travel times. The orientation of the contour
(after deformation) in the neighborhood of $p_k$ is determined by
$\sgn \partial_p^2 \tau_{M,i} \restriction_{p = p_k}$. Besides for the geometric spreading factor, the leading order amplitude is $Q_{M,i}(p)$, which is just a product of reflection and transmission coefficients corresponding to the \branch{}s of the associated ray; terms involving curvature of the interface do not appear in the lead order term and only make an appearance in the subsequent term that is not necessary for the theorem. We note that
\begin{itemize}
\item
$\tilde N_{Mik} = \tilde N_{s;Mik}(r,r_0,\Theta)$ for multi-orbit waves
  ($s=3,4,\ldots$) includes polar phase shifts produced by any angular
  passages through $\Theta = 0$ or $\Theta = \pi$ as well;
\item
if $r$ lies in a caustic, the asymptotic analysis needs to be adapted
in the usual way.
\end{itemize}

\begin{comment}
Next, as in \cite[Section 3.2]{HIKRigidity}, we consider the forward rays with their time reversal to deal with the $(\sin \Theta)^{-1/2}$ term as $\Theta \to 0$. We also take an inner product to deal with the curl terms in the definition of toroidal modes. After taking the inverse Fourier transform, we have \cite[Section 3.2]{HIKRigidity}
\end{comment}

Next, we take the trace of $\p_t G$ by restricting to $(r=r_0,\Theta =
0)$ and integrating. The phase function on the diagonal is $T_{Mik} =
\tau_{M,i}(r,r,p_k)+\pi(s-1)p_k$ and we apply stationary phase in the
variables $r,\theta,\psi$ with large parameter $\omega$. 
This is a standard computation exactly as done in \cite[Section 3.2]{HIKRigidity}.

Following the computation in \cite{HIKRigidity}, we obtain
the leading order term in the trace formula as
\begin{align} \label{eq: principal term spheric symm}
&\operatorname{Re}\sum_s\sum_{M \in \mathbb{Z}_{\geq 0}^{4(N-4)}}  \sum_k\
         \left(\frac{1}{2\pi \ii}\right)^{3/2}(t-T_{s;Mk}+\ii0)^{-5/2}
       \ii^{M_{k}+s-1}\\
&\qquad \qquad \qquad \qquad      \cdot c Q_{M}(p_k) L_{k}\abs{p_k^{-2}\p^2_p\tau_{M}(p_k)}^{-1/2}
\frac{1}{2\pi}\abs{SO(3)},
\end{align}
where $L_{k}$ is the travel time of a ray with only transmitted \branch{}s from $r= \surf$ to $r=R^*$ (see \eqref{e: transmitted length L}).
Note that the critical set becomes $\Theta_{M,k} = \p_p \tau_{M}(p_k)$ so
$\p^2_p \tau_{M}(p_k)= \p_p\alpha_{M,k}$ when restricting to the diagonal. Also, we use that here,
\[
   T_{s;Mk} = T_{s;Mk}(r,r;p_k) = \tau_{M}(r,r;p_k)
     + \left\{\begin{array}{rcl}
           p_k (s-1) \pi & \text{if $s$} & \text{odd} \\
           p_k s \pi     & \text{if $s$} & \text{even}
              \end{array}\right.
\]
is independent of $r$ on the critical set.
We note that $p_k$ exists only for $\abs{M}$, and $s$, sufficiently large,
which reflects the geometrical quantization.

\end{proof}

\subsection*{Harmonics of the principal ray}
From the argument above, if $\gamma$ is a periodic orbit with period $T_{s,Mik}$ for some indices $s,M,i,k$ described above, the principal symbol of the contribution of $\gamma$ to the trace is as above. We can immediately write down the leading order contribution of $\gamma^l$ which is $\gamma$ travelled $l$ times. The travel time will be $lT_{s,Mik}$ . Then
$Q_{M,i}(p_k)$ becomes
$Q_{M,i}(p_k)^l$, $M_{ik}$ becomes $lM_{ik}$, and $p_k^{-2}\p_p \alpha_{M,ik} $ becomes $lp_k^{-2} \p_p \alpha_{M,ik}$.

\subsection{Spheroidal modes}\label{s: spheroidal modes}
The above trace formula, theorems \ref{t: spectral rigiditiy} and \ref{thm:blspr-multilayer} are essentially dealing with a scalar wave equation with a single wavespeed. The analysis for toroidal modes reduced to a scalar wave equation with an associated Laplace-Beltrami operator. However, our methods can also treat the isotropic elastic setting to include spheroidal modes (with the PDE described in \cite[Chapter 8]{DahlTromp}) where two wavespeeds ($c_P$ and $c_S$) are present corresponding to the $P$-waves and the $S$-waves, and there is a spectrum associated to the elliptic, isotropic elastic operator. In the elastic setting, each \branch{} of a broken geodesic will be a geodesic for either the metric $c_P^{-2}dx^2$ or $c_S^{-2}dx^2$, so there is an associated length spectrum as well that includes \emph{mode converted} \branch{}s. Thus, theorem \ref{t: spectral rigiditiy} can be extended to the case of the elastic operator by using corollary \ref{cor:2-speeds} if the length spectrum (or a dense subset) can be recovered by a trace formula from the spectrum. 
The theorem would take the form
\begin{theorem}[Elastic spectral rigidity with moving interfaces]
\label{t: elastic spectral rigiditiy}
Fix any $\eps>0$ and $K\in\N$, and let $c_{P,\tau}(r)$ and $c_{S,\tau}(r)$ be an admissible family of profiles with discontinuities at $r_k(\tau)$ for all $k=1,\dots,K$.
Suppose that the length spectrum for each $c_{P/S,\tau}$ is countable in the 
ball $\bar B(0,1) \subset \RR^3$.
Assume also that the length spectrum satisfies the principal amplitude injectivity condition and the periodic conjugacy condition.

Suppose
$\spec(\tau)=\spec(0)$ for all $\tau\in(-\eps,\eps)$.
Then $c_{P,\tau}=c_{P,0}$, $c_{S,\tau}=c_{S,0}$ and $r_k(\tau)=r_k(0)$ for all $\tau\in(-\eps,\eps)$ and $k=1,\dots,K$.
\end{theorem}

%\textcolor{red}{Joonas: A key step is that we only use basic rays for the length spectral rigidity and a countable number may be "thrown away". Can you describe this a little more precisely here?}

Thus, all we need is to extend proposition \ref{prop:Trace Formula} to the elastic case and then apply corollary \ref{cor:2-speeds}. Since the calculation is similar, but a more tedious version of the case we consider here, We will just provide an outline of the proof.

\begin{enumerate}
    \item 
The Green's function associated to just the spheroidal modes can be computed analogously as in \cite[Equation (31)]{ZhaoModeSum}. 

\item One can then obtain (vector-valued) WKB solutions to approximate spheroidal modes, which are eigenfunctions of the static, elastic operator as in \cite[Appendix A]{ZhaoModeSum} and \cite[Chapter 8]{DahlTromp}.

\item
 We can use the methods presented here (with the method of steepest descent for the asymptotic analysis) to then determine the leading order asymptotics of the sum of eigenfunctions to obtain a corresponding trace formula. The scattering coefficients will be determined by the elastic transmission condition, with an associated Debye expansion as done in appendix \ref{a: Generalized Debye}. Afterward, the stationary phase analysis
will lead to the same form as \eqref{eq: trace coefficient} but the reflection and transmission coefficients appearing in $Q(\pg)$ will be different to account for mode conversions. Also, $\alpha(\pg)$ will be modified with the appropriate wave speed appearing in each constituent of the linear combination of epicentral distances that correspond to an associated $P$ or $S$ \branch{} of $\gamma$.
 
 \item
The computation in \cite{ZhaoModeSum} does not treat glancing nor grazing rays, but their formulas can be modified with the methods presented here to account for such rays as well. The $n(\pg)$ appearing in \eqref{eq: trace coefficient} will again count the number of ``dynamic analogs'' associated to $\gamma$ as described in \cite{HronCriteria} for the spheroidal case; that paper also has several figures of broken geodesics in the spheroidal case.

\end{enumerate}

Under an analog to the principal amplitude injectivity condition for spheroidal modes, one can recover the \simp{} length spectrum for each of the two wave speeds. One then uses  Corollary~\ref{cor:2-speeds} to recover both wave speeds.

\section{Declarations}
\subsection*{Funding}
MVdH was supported by the Simons Foundation under the MATH + X program, the National Science Foundation under grant DMS-1815143, and the corporate members of the Geo-Mathematical Imaging Group at Rice University.
JI was supported by the Academy of Finland (projects 332890 and 336254).

\subsection*{Conflict of interest/Competing interests}
{\bf Financial interests:} The authors declare they have no financial interests.
\\

\noindent {\bf Non-financial interests:} The authors declare they have no non-financial interests.

\subsection*{Availability of data and material} Not applicable

\subsection*{ Code availability} Not applicable

\appendix

\section{Class of metrics satisfying (A1)-(A4)} \label{app: example satisfying A1-A4}

Here, we present a class of metrics that meet all the assumptions mentioned in the theorems. Demonstrating the satisfaction of assumption (A.2) on a layered manifold (or any analogous concept of a simple length spectrum) is a challenging task, even for the most basic Herglotz metrics. Therefore, we instead prove that this assumption is generic within our class of metrics. We subsequently prove that the Euclidean ball possesses a simple length spectrum, thereby establishing the non-emptiness of the class of metrics satisfying our assumptions.

It is worth noting that the proof of Theorem \ref{t: spectral rigiditiy}  only necessitates the recovery of the basic length spectrum, and one may omit any finite number of such lengths. Thus, we can allow for cancellations in the trace formula as long as they do not occur infinitely many times for periods in the basic length spectrum.

We will begin by examining a single interface at $r = \refll$ and then proceed to generalize this approach to additional layers. 
% We consider the example given in \cite[Section 3.7]{Cerveny_2001}. 
Let us select a function $c(r)$ such that
\[
(c/r)^{-2} = a + b\ln(r), \text{ for } \refll < r < \surf
\]
and 

\[
(c/r)^{-2} = a' + b'\ln(r), \text{ for } \CMB < r < \refll
\]
where $a, a', b, b' > 0$ are parameters, ensuring that condition (A4) is satisfied. Since the proof of the main theorem relies solely on the basic length spectrum, we will verify our assumptions for this specific case.

For basic rays in the upper layer, we can calculate the epicentral distance of a single leg as follows:
\[
\alpha(p) = \int_{R^*}^1 \frac{p}{r'\sqrt{ (c/r')^{-2}-p^2}} \ dr'
= \int_{R^*}^1 \frac{p}{r'\sqrt{ a+b\ln(r)-p^2}}.
\]
Here, the value of $R^*$ depends on the parameter $p$. We can further break down the calculation:
\[
\alpha(p)
= \begin{cases}
\frac{2p}{b} [\sqrt{a - p^2} -  \sqrt{a+b\ln(\mathbf f) - p^2}] &\text{ if } 0 < p < \sqrt{a+b\ln(\mathbf f)}\\
\frac{2p}{b} \sqrt{a - p^2}  & \text{ if } \sqrt{a+b\ln(\mathbf f)} < p < \sqrt a
\end{cases},
\]
where in the case of reflecting rays, $R^* = \mathbf f$ in the first layer, and for diving rays (as discussed in Appendix \ref{a: Generalized Debye}), $R^*$ is determined by the condition $(c(R^*)/R^*)^{-2}-p^2 = 0$. In this scenario, we have:
\[
\alpha'(p) = \frac 2 b \sqrt{a-p^2} - 2p^2/(b\sqrt{a-p^2})
\]
This expression only equals zero when $p^2 = a/2$. We can similarly verify that $\alpha'(p)$ does not vanish in the reflecting region ($0<p<\sqrt{a+b\ln(\mathbf f)} $) except for $p^2 = a(a+b\ln(\mathbf f))/(2a+b\ln(\mathbf f))$. A similar argument holds for basic rays in the second layer, and this argument readily extends to any finite number of layers, each having a metric within the above class. If we have $N$ layers with similar such metrics, there is an associated $\alpha_i(p)$ for layer $i$. Then for fixed nonnegative integers $n_1, \dots n_N$, the epicentral distance is $\alpha(p) = \sum n_j \alpha_j$ so $\alpha'(p) =0$ for at most finitely many values of $p$. Thus, (A3) is satisfied

To get a periodic basic ray, we need $\alpha = 2\pi q$ for some $ q \in \mathbb Q$. For diving rays, this implies
\[
p^2(a-p^2) = \pi^2 q^2 b^2.
\]
Solving for $p$, we obtain
\[
p^2 = \frac{a \pm \sqrt{a^2 - 4\pi^2 q^2b^2}}{2}
\]
Hence, periodic diving rays are in one-to-one correspondence with elements of $\mathbb Q$ such that \[a^2 - 4\pi^2 q^2 b^2 \geq 0.\] If we choose $a, b$ such that $a/b \notin \pi \mathbb Q$, then $p^2 \neq a/2$ for a periodic ray, and the clean intersection hypothesis is satisfied on all diving periods. A similar calculation provides the result for reflecting rays and rays in the second layer. Periodic reflecting rays in a layer are in one-to-one correspondence with elements of $\mathbb Q$ such that $ad - \pi^2 q^2 b^2 \geq 0$ where $d = a+b\ln \mathbf f$. If we restrict to $a,b$ such that $\sqrt{ad}/b \notin \pi \mathbb Q$, then the clean intersection hypothesis is satisfied for all basic reflecting rays. The extension to any finite number of layers follows analogously since we only require the basic length spectrum.

It is important to note that this argument can also extend to non-basic rays since assumption (A.1) could fail on at most finitely many primitive periodic rays, which still suffices for length spectral rigidity.
%In conclusion, it's worth mentioning that the proof of the main theorems doesn't necessitate the recovery of all the lengths in $\lsp(c)$. Even if the clean intersection hypothesis fails for a finite number of rays and those lengths are not recovered with the trace formula, the theorems remain valid with the same proof.

\begin{remark}
It should be noted that the proof of the main theorems doesn't require the recovery of all lengths in $\lsp(c)$. Therefore, even if the clean intersection hypothesis fails for a finite number of rays and those lengths are not recovered with the trace formula, the theorems remain valid with the same proof.
\end{remark}

\subsection{Genericity of assumption (A.2)}
To meet the requirements of (A.2), we initially consider two layers with a wave speed given having the form $c = a+b \ln(r)$ in each layer, and prove that assumption (A.2) is generic (in a sense described below) among metrics in this class. The argument we provide will readily extend to the case of a finite number of layers.
\begin{comment}
The choice of $c_1$ aligns with the criteria established in \cite[appendix C]{HIKRigidity} to satisfy all four assumptions. The length spectrum for this layer, denoted as $\lsp_1$, is discrete, except for a potential accumulation point corresponding to the gliding ray at the boundary. If we introduce a small perturbation to $c_1$, denoted as $c_2$, resulting in a corresponding length spectrum $\lsp_2$, then $\lsp_1 \cap \lsp_2$ will be finite. As a result, terms in a trace formula corresponding to elements in $\lsp_1$ cannot cancel out with those from $\lsp_2$ except for possibly finitely many instances. We can repeat this process for $N$ layers, ensuring that only finitely many cancellations are feasible in a trace formula.
\end{comment}
Let $M$ represent the annulus where we connect both of these layers (layer $1$ and layer $2$), with wave speed $c_1$ in the upper layer and $c_2$ in the lower layer. Let $c$ denote the piecewise smooth wave speed with the length spectrum $\lsp$. %According to the aforementioned argument, the possibility of an infinite number of cancellations in a trace formula for $g$ must involve closed rays that traverse both layers.

Suppose we have a closed basic ray $\gamma$ with a period $T = T(p_\gamma)$ such that a cancellation occurs in the trace formula. Let's outline the conditions under which this is possible. First, denote the closed ray $\beta = \beta(p)$ that corresponds to such a cancellation in the trace formula where $[\gamma] \neq [\beta]$. Hence, the principal term in the wave trace corresponding to $[\gamma]$ cancels with that of $[\beta]$.
Let $L_i= L_i(p)$ denote the length of one leg in layer $i$ for ray $\beta$.  There exist $n_1$ and $n_2$ nonnegative integers that satisfies the equation:

\begin{equation}\label{e: T degenerate}
T = n_1 L_1(p) + n_2 L_2(p)
\end{equation}
\begin{comment}
%geometric injectivity condition (maybe not needed)
We denote the amplitude in the trace formula corresponding to $\gamma$ as $A=Q(p_\gamma)A_{p_\gamma}$, where $Q = 1$ or $R_{++}^n$ for some $n$. Then, $p, n_1, n_2$ must also satisfy the equation:

\begin{equation}\label{e: spread degenerate}
Q(p_\gamma)A_{p_\gamma}=Q(p)A_{p}
\end{equation}

Here, $Q(p) = T^{n_1}_{+-}(p)T^{n_1}_{-+}(p)R_{++}^{n_2}(p)$, and it also has only finitely many solutions for $p$.
\end{comment}
Furthermore, to ensure periodicity, $n_1, n_2, p$ must also satisfy:
\begin{equation}\label{e: angle degenerate}
n_1(\alpha_1+\alpha_2) + n_2 \alpha_2 = 2\pi m
\end{equation}
for some integer $m$.

\begin{definition}
We will define a period $T = T(p_\gamma)$ as \emph{not simple} if condition \eqref{e: T degenerate} is satisfied for some $n_1, n_2, q, m,$ and closed ray $\beta$ with a ray parameter $p$ such that $[\beta] \neq [\gamma]$, as described above. These conditions represent the simultaneous requirements necessary for a cancellation in the wave trace. In other words, if a period $T$ is simple, then it will not result in a cancellation in the wave trace.  We will say that $\blsp(c)$ is \emph{simple} if each non-gliding period in this set is simple. A period is gliding when its associated ray has a gliding leg, and such periods are not needed to prove length spectral rigidity.
\end{definition}

Suppose we have a class of metrics
\[
\mathcal C = \{ g_c:
(c/r)^{-2} = a + b\ln(r) \text{ for some $a,b > 0$ in each layer }
\}
\]
Our goal is to demonstrate that
\[E = \{ g_c \in \mathcal C: \blsp(c) \text{ is simple } \}\]
is \emph{generic} in the Baire category sense. In other words, within the class $\mathcal C$, metrics generically satisfy (A1)-(A4). Accumulation points in the length spectrum can pose difficulties, so it's helpful to define a set for $\delta > 0$ as follows:
\[
\mathcal G_\delta
= \{ t\in \RR: t \notin (2\pi l - \delta , 2\pi l + \delta) \cup (2\pi \mathbf f l - \delta, 2\pi \mathbf f l + \delta ) \text{ for every integer }l \}  
\]
Note that the gliding ray traveling once around the outer boundary or interface has a length of $2\pi$ or $2\pi \mathbf f$, respectively. Therefore, $\mathcal G_\delta$ represents lengths that are at least $\delta$ away from each gliding ray along the outer boundary or interface. Thus, throughout the arguments below, closed rays will have periods that stay a fixed distance away from any period associated to a gliding ray.

%Here, we let $\blsp(c)$ denote only basic rays without legs that are tangential to the interface, of which there could only be at most finitely many.

Let us define
\begin{multline}
E_n = \{ g_c \in \mathcal C: \text{for each } T \in \blsp(c), \text{ if } T\leq n \\
\text{ and } T \in \mathcal G_{1/n} \text{ and $T$ is Morse-Bott nondegenerate,}\text{ then $T$ is simple } \},
\end{multline}
so that $E = \cap_n E_n$. 
Then we have
\[
E \subset \dots \subset E_n \subset E_{n-1} \subset \dots
\subset E_2 \subset E_1 \subset E_0 := \mathcal C.
\]
We would like to prove that $E_n$ is open and that $E_n$ is dense in $E_{n-1}$ for each $n$ to apply the Baire category theorem. This will suffice since the proof of the main theorems go through even if we do not recover periodic lengths in the lengths spectrum that align with a gliding period. Furthermore, we demonstrated above that within this class of metrics, only a finite number of degenerate periods can exist. The spectral rigidity proof holds even if we do not recover any finite subset of the basic length spectrum. Therefore, all metrics in the set $E$ satisfy the assumptions necessary for theorem \ref{t: spectral rigiditiy} to hold.

We will follow similar arguments to \cite{Anasov1883}. Using the notation there, for a metric $g$, denote $\{ \Phi_t^g\}$ to be the bicharacteristic flow on $T^*M$, which restricts to a bicharacteristic flow on $S^*M.$ For a particular ray $\gamma$, we denote $N_{refl}= N_{refl}([\gamma]) = (N^{(1)}_{refl}, N^{(2)}_{refl}),$
where
\[
N^{(i)}([\gamma]) = \text{ number of reflections of $\gamma$ in the $i$'th layer}
\]
and 
\[
N_{WN} = N_{WN}([\gamma]) = \text{winding number of $\gamma$}.
\]

\subsubsection{Openness}
%\texorpdfstring{$E_k$}{} is open }
To prove that $E_k$ is open, let $g_c = g_{c_\mathbf a} \in E_k$ with $\mathbf a = (a,b, a', b')$. We want to demonstrate that there is a neighborhood $O\subset \mathbb R^4$ of $\mathbf a$ such that $g_{c_\mathbf b} \in E_k$ for each $\mathbf b \in O$.

Assume, for the sake of contradiction, that this is not the case. Then there exists a sequence $\mathbf a_n \to \mathbf a$ and a sequence of periods $T_n = T_n(p_n) \in \blsp(c_{\mathbf a_n})$, each of which is not simple and has a length $\leq k$ and belongs to $\mathcal G_{1/k}$. We will show that a subsequence of these periods converges to a period in $\blsp(c_{\mathbf a})$ that is also not simple, leading to a contradiction.

By definition, for each $n$, there exists closed rays $\gamma_n, \tilde \gamma_n$
with associated $p_n , \tilde p_n$ such that $[\gamma_n] \neq [\tilde \gamma_n]$, but the associated period $\tilde T_n = T(\tilde p_n) \in \lsp(c_n)$ is such that:
\[
T(p_n) =  T(\tilde p_n).
\]

%and 
%\[
%A_n(p_n) = \tilde A_n(\tilde p_n)
%\]
%where $A_n, \tilde A_n$
%are the leading coefficients in the %wave trace formula of %\ref{prop:Trace Formula}. 

Each $p_n$ is associated to a periodic ray equivalence class, so there is a $v_n \in S^*M$ such that
\[
\Phi^{g_n}_{T_n} v_n = v_n.
\]
Likewise, we have $\tilde v_n \in S^*M$ such that
\[
\Phi^{g_n}_{\tilde T_n} \tilde v_n = \tilde v_n.
\]
By compactness and passing to subsequences (while still denoting the index by $n$), we can find $T, \tilde T \leq k$, $v, \tilde v, p, \tilde p$ such that $T_n \to T$, $\tilde T_n \to \tilde T$, $p_n \to p$, $\tilde p_n \to \tilde p$, $v_n \to v$, and $\tilde v_n \to \tilde v$. Since we know $g_n \to g_c$, upon taking limits in the equations above, we obtain:
\[
\Phi_T^{g_c} v = v
\]
and
\[
\Phi_{\tilde T}^{g_c} \tilde v = \tilde v.
\]
and $T(p) = \tilde T(\tilde p)$. Let $\gamma, \tilde \gamma$ denote the rays $\Phi^{g_c}_t  v$ and $\Phi^{g_c}_t  \tilde v$ projected to $M$ respectively. They are periodic with periods $T$ and $\tilde T$, respectively.
The assumption regarding $\mathcal G_{1/k}$ implies there exists $\delta >0$ such that $L_1(\tilde p_n) > \delta$ and $L_2(\tilde p_n) >\delta$ for large enough $n$ and assuming the ray is in both layers. If the ray is basic, then a similar argument will hold.  By \eqref{e: T degenerate}, we have
\[
T_n = z_{1,n} L_1(\tilde p_n) + z_{2,n} L_2(\tilde p_n) > (z_{1,n}+z_{2,n})\delta
\]
for some nonnegative integers $z_{1,n}, z_{2,n}$. This shows that $N_{n,refl}, N_{n,WN}$ remain bounded for large enough $n$, and in particular, become constant. Thus, $N_{n,refl} = N_{refl}([\gamma])$ and 
$N_{n,WN} = N_{WN}([\gamma])$ for large enough $n$.
Let us show that $[\gamma] \neq [\tilde \gamma]$. For the sake of contradiction, suppose $[\gamma] = [\tilde \gamma]$. This would imply that $N_{refl} = \tilde N_{refl}$ and $N_{WN} = \tilde N_{WN}$. Hence, 
$N_{n, refl} = \tilde N_{n, refl}$ and $N_{n, WN} = \tilde N_{n, WN}$ for large enough $n$. Since $T_n = \tilde T_n$, then $\gamma$ and $\tilde \gamma$ are kinematic analogs so that $[\gamma_n] = [\tilde \gamma_n]$.
This is because since $p = \tilde p$, then $p_n$ is near $\tilde p_n$ for large $n$, and due to nondegeneracy, if $T(\sigma)$ denotes the travel time for $\sigma$ near some $p_n$, then $\p_\sigma \restriction_{\sigma = p_n} T \neq 0$. This implies $T(p_n) \neq T(\tilde p_n)$ when $p_n \neq \tilde p_n$ for large enough $n$. Thus, for large $n$, $[\gamma_n ] = [\tilde \gamma_n]$, which is a contradiction. Hence, $[\gamma] \neq [\tilde \gamma]$ so $T$ is a non-simple period for $g_c$, which is a contradiction. This establishes openness.
\\ \\

\subsubsection{$E_{k+1}$ is dense in $E_k$}
%\texorpdfstring{$E_{k+1}$}{} is dense in \texorpdfstring{$E_k$}{}}

Let $g_c \in E_k$. We will show there exists a $\tilde g \in E_{k+1}$ arbitrarily close to $g_c$. 

First, let us assume that $M$ is a single layer, and describe afterward how the proof extends to multiple layers. Let $[\gamma_1], \dots [\gamma_N]$ be all the distinct closed rays equivalence classes whose periods are $\leq k+1$, belong to $\mathcal G_{k+1}$, and are non-degenerate. In view of the non-degeneracy computation at the start of the section, and that we are a fixed distance $1/(k+1)$ away from the periods of gliding rays, the proof in \cite{Anasov1883} shows that there are finitely many such trajectories.

%Anasov argument
\begin{comment}
{\bf The closed manifold case reviewed:} Then we claim that there is a neighborhood $\mathcal U$ of $g_c$ and functions
\[
v_i : \mathcal U \to S^*M
\]
\[
\omega_i:\mathcal U \to \mathbb R_t
\]
such that the rays
\[
\Phi^g_{t} v_1(g), \dots
\Phi^g_{t} v_N(g)
\]
are periodic with periods $\omega_1(g), \dots, \omega_N(g)$ respectively. Also, there is a neighborhood $U$ of
\[
\cup_{i,t} \Phi_{g_c}
\]
and a neighborhood $\mathcal V \subset \mathcal U$, such that for any metric $g \in \mathcal V$, if the ray $\Phi_t^g v$ is periodic with period $\leq k+1$, then it coincides with one of the periodic rays
\[
\Phi^g_{t} v_1(g), \dots
\Phi^g_{t} v_N(g).
\]

This uses the nondegeneracy of the rays $\gamma_1 \dots \gamma_N$ and was done in the closed manifold case.
We may do something similar for our boundary case. 

{\bf Our manifold with boundary case:}
\end{comment}
 For each $i$, we have
an associated epicentral distance $\alpha(p_i)$ for the ray $\gamma_i$.
Consider any analytic perturbation of the metric $g_c$ denoted $g_{c_h}$ with parameter $h$.
By \cite{HIKRigidity}, we may find a $\epsilon > 0$ and a function
\[
\phi_i : (-\epsilon, \epsilon)_h \to (0, \surf/c(\surf))_p
\]
such that the ray
\[
\gamma_{i, h} = \gamma^{g_{c_h}}_i(\phi_i(h))
\]
has associated
\[
\alpha_{i, h} = \alpha_i(\phi_i(h))
= \alpha(p_i)
\]
and is thus periodic. Also $N_{i, refl, h}$ and $N_{i, WN, h}$ remain constant as well under varying $h$.

By the proof in \cite{Anasov1883}, for any period ray $\gamma_h$ associated to metric $g_{c_h}$ whose period is $\leq k+1$, is non-degenerate, and belongs to $\mathcal G_{k+1}$, then $[\gamma_h]$ must coincide with one of 
\[
[\gamma_{1, h}], [\gamma_{2, h}], \dots, [\gamma_{N, h}].
\]

Hence, if $g_c$ is not in $E_{k+1}$, then there exists $T, \tilde T \in \blsp(c)$, such that $T(p) = \tilde T(\tilde p) $ with length $\leq k+1$, 
they are non-degenerate, and belong to $\mathcal G_{k+1}$. 
We will show that under a slight perturbation of $c$, $T(p)$ becomes simple while keeping all the simple periods of $g_c$ simple under perturbation. Without loss of generality, we may assume these two periods correspond to the ray $\gamma_1$ and $\gamma_2$, and the other periods $T_3, T_4 \dots, T_N$ are simple (if other periods were non-simple, we can just repeat the following argument until all are simple). 

Under a perturbation of $c$, denoted $c_h$ with parameter $h$, the perturbed periods $T_{3,h}, \dots, T_{N, h}$ remain simple for $h$ small enough due to continuity of $h \mapsto T_{i, h}$ (in fact, differentiability by \cite{HIKRigidity}), the openness of $E_k$, and since these are all the periods $\leq k+1$.

Thus, we need to show that we may pick a particular perturbation so that $T_{1,h}$ becomes simple and the proof of density is complete. The perturbed travel time is
\[
T_h = T_1(p) + h \delta T(p) + O(h^2)
\]
and 
$\tilde T_h =  T_2(p) + h \delta T(\tilde p) + O(h^2)$. 

If we can arrange that $\delta T(p) \neq \delta T (\tilde p)$, then $T_h \neq \tilde T_h$ for small enough $h$. Note that for a fixed $\mathbf a' \in \RR^4$, consider the analytic perturbation $\mathbf a_h = \mathbf a + h \mathbf a '$. By openness, $c_h \in E_k$ for small enough $h$. Since we have an explicit formula for $T_h$ and $\tilde T_h$, we can compute explicitly $\frac{d}{dh}\restriction_{h= 0 } T_h
$
and so too for $\tilde T_h$, and we can pick $\mathbf a'$ to guarantee that $\delta T$ and $\delta \tilde T$ are different. Thus $T_h(p_h) \neq \tilde T_h(\tilde p_h) $. This establishes the required density.

The above argument readily applies to the case of multiple interfaces since it only relies on having a discrete length spectrum away from a fixed distance away from the gliding periods, and that the quantities $N_{n, \text{refl}}$ and $N_{n, \text{WN}}$ become constant for large $n$ when a family of periodic rays converges to another ray. Thus, the only issue is nonbasic periodic rays that have a gliding leg. However, if we restrict our class of metrics such that
\[
a' + b' \ln(\mathbf{f}) < a + b \ln(\mathbf{f}),
\]
then such nonbasic rays with a gliding leg do not exist. Hence, the only accumulation points in the length spectrum lie in $\mathcal{G}_\delta$, and we do not need to recover such periods for the proof.

On the other hand, if one wanted to consider such nonbasic rays with one or more gliding legs, then these would not create a cancellation in the wave trace, even if their period coincided with the period of a ray in a disjoint equivalence class. This is because the proof of the trace formula shows that the term in the wave trace corresponding to such a ray is smoother than the terms corresponding to non-gliding rays.

\subsection{Simplicity of length spectrum of a Euclidean ball}\label{a: ball case}

We give a concrete example of a simple basic length spectrum, the Euclidean disc.
The basic length spectrum is simple and the full length spectrum is almost simple; the only coincidence is that going twice around a hexagon takes as long as going three times back and forth along a diameter.

The proof relies heavily on number theory.
The nature of the proof in the geometrically simplest case suggests that finding any concrete examples in non-Euclidean geometry will be very difficult.
This is why genericity arguments like the one presented above are necessary.

\begin{theorem}
\label{thm:disc-lsp-example}
Let $\gamma_1$ and $\gamma_2$ be two periodic broken rays in the Euclidean unit disc.
Suppose that the rays are basic and not rotations of each other.
Then their lengths are different.

In fact, the ratio of their lengths is rational if and only if $\gamma_1$ is a hexagon (length $6$) and $\gamma_2$ is a diameter (length $4$, counted back and forth).
\end{theorem}

The proof is easy once we establish the following proposition.

\begin{proposition}
\label{prop:sine-ratio-irrat}
If $r_i$ for $i=1,2$ are two different rational numbers in $(0,\frac14]$, then
\begin{equation}
\frac{\sin(2r_1\pi)}{\sin(2r_2\pi)}
\notin
\QQ
\end{equation}
unless $r_1=\frac{1}{4}$ and $r_2=\frac{1}{12}$ or vice versa.
\end{proposition}

\begin{proof}[Proof of theorem~\ref{thm:disc-lsp-example}]
Each $\gamma_i$ is composed of line segments with an opening angle $\alpha_i$ as seen from the origin.
The condition for periodicity is that $\alpha_i\in\pi\QQ$, and in Euclidean geometry $\alpha\in(0,\pi]$.
We thus have $\alpha_i=2\pi\cdot\hat r_i$ with $\hat r_i=p_i/q_i\in(0,\frac12]\cap\QQ$.
The length of one line segment is $l_i=2\sin(\alpha_i/2)=2\sin(\hat r_i\pi)$ and the length of the whole ray is $L_i=q_il_i$.

The claim is that $q_1l_1\neq q_2l_2$.
In most cases this follows from proposition~\ref{prop:sine-ratio-irrat}.
The only case not covered by it is when $\hat r_1=\frac12$ and $\hat r_2=\frac16$ or vice versa.
These correspond to a diameter (with $q_1=2$ and $L_1=4$) and a hexagon (with $q_1=6$ and $L_1=6$).
In this one case the ratio is rational, but still $L_1\neq L_2$, completing the proof of the main claim and the additional claim.
\end{proof}

To prove the proposition we will use some notation and concepts from number theory without explanation in this section to keep exposition concise.
In what follows, $\zeta_k=e^{2\pi i/k}$ is a $k$th root of unity, $\QQ(\zeta_k)$ is the $k$th cyclotomic field, $\phi$ is Euler's totient function, $\gcd(a,b)$ is the greatest common divisor of $a$ and $b$, and the degree of a field extension of an element thereof is the degree of its minimal polynomial over the subfield.

\begin{lemma}
\label{lma:phi-product}
Let $a$ and $b$ be integers with $a>1$ and $b>1$.
Then $\phi(ab)=\phi(b)$ if and only if $b$ is odd and $a=2$.
\end{lemma}

\begin{proof}%[Proof of lemma~\ref{lma:phi-product}]
Let $p_1,\dots,p_m$ be the primes appearing in the prime factorizations of $a$ and $b$, so that
\begin{equation}
a=\prod_{i=1}^m p_i^{k_i}
\text{ and }
b=\prod_{i=1}^m p_i^{l_i}
,
\end{equation}
where some of the powers $k_i$ and $l_i$ may be zero.

If we define the functions $f_i\colon\N\to\N$ by
\begin{equation}
f_i(k)
=
\begin{cases}
1 & \text{when }k=0 \\
p_i^k-p_i^{k-1} & \text{when }k>0
,
\end{cases}
\end{equation}
then
\begin{equation}
\phi(b)=\prod_{i=1}^m f_i(l_i)
\text{ and }
\phi(ab)=\prod_{i=1}^m f_i(k_i+l_i)
.
\end{equation}
Since each $f_i$ is increasing, we have $f_i(l_i)\leq f_i(k_i+l_i)$ for all $i$ and so $\phi(b)\leq\phi(ab)$.

The equality $\phi(b)=\phi(ab)$ is only possible when $f_i(l_i)=f_i(k_i+l_i)$ for all $i$.
A quick inspection reveals that $f_i$ is strictly increasing with one exception:
$f_i(0)=f_i(1)$ if $p_i=2$.

We assumed $a>1$, so at least one $k_i$ is non-zero.
But $f_i(l_i)=f_i(k_i+l_i)$ for $k_i>0$ is only possible when $l_i=0$ and $k_i=1$ and $p_i=2$.
Therefore $\phi(b)=\phi(ab)$ implies that $a=2$ and $b$ is odd.
\end{proof}

\begin{lemma}
\label{lma:real-cyclotomic}
Let $a$ and $b$ be positive integers.
The following are equivalent:
\begin{enumerate}
\item $\RR\cap\QQ(\zeta_a)=\RR\cap\QQ(\zeta_b)$.
\item Any of the following holds:
\begin{enumerate}
    \item $a=b$.
    \item $a=2b$ and $b$ is odd.
    \item $b=2a$ and $a$ is odd.
    \item $a$ and $b$ are both in the set $\{1,2,3,4,6\}$.
\end{enumerate}
\end{enumerate}
\end{lemma}

\begin{proof}%[Proof of lemma~\ref{lma:real-cyclotomic}]
We only prove ``$\implies$'' --- the reverse implication ``$\impliedby$'' is simpler and can be easily constructed from the arguments presented here.
Suppose that $\RR\cap\QQ(\zeta_a)=\RR\cap\QQ(\zeta_b)$ and, without loss of generality, that $a<b$.

Intersections of cyclotomic fields satisfy
\begin{equation}
\QQ(\zeta_a)\cap\QQ(\zeta_b)
=
\QQ(\zeta_{\gcd(a,b)})
,
\end{equation}
so the assumption implies also that $\RR\cap\QQ(\zeta_a)=\RR\cap\QQ(\zeta_b)=\RR\cap\QQ(\zeta_g)$, where we have denoted $g=\gcd(a,b)$ for short.
The assumption $a<b$ implies $g<b$.

Since $\zeta_1=1$ and $\zeta_2=-1$, we have simply $\QQ(\zeta_1)=\QQ(\zeta_2)=\QQ$.
When $n\geq3$, the number $\zeta_n$ is not real and its degree is $\phi(n)$ and the minimal polynomial is the $n$th cyclotomic polynomial.
The degree of the real number $z_n\coloneqq\frac12(\zeta_n+\overline{\zeta_n})$ is $\phi(n)/2$ as shown in~\cite{WZ:deg-of-trig-numbers}.
Because $\RR\cap\QQ(\zeta_n)=\QQ(z_n)$, our assumption implies that the degrees of $z_a$ and $z_b$ and $z_g$ agree.

Let us split the analysis in different cases, using the fact that $\phi^{-1}(1)=\{1,2\}$ and $\phi^{-1}(2)=\{3,4,6\}$:
\begin{enumerate}
\item $b=5$ or $b\geq7$:
In this case $\phi(b)>2$ and so $\deg(z_b)>1$.
Thus also $\deg(z_g)>1$, which implies that $g=5$ or $g\geq7$.
Therefore the formula $\deg(z_n)=\phi(n)/2$ is valid for $n=b,g$ and we find that $\phi(g)=\phi(b)$.
Now that $g\vert b$ and $g<b$, lemma~\ref{lma:phi-product} gives that $b=2g$ and $b$ is odd.
The only possible value for $a$ is $g$.

\item $b\in\{2,3,4,6\}$:
In this case $\deg(z_b)=1$ and so $\RR\cap\QQ(\zeta_b)=\QQ$.
The set of values of $a$ for which $\deg(z_a)=1$ is $\{1,2,3,4,6\}$.
\end{enumerate}
This is exactly the claim.
\end{proof}

\begin{proof}[Proof of Proposition~\ref{prop:sine-ratio-irrat}]
We switch from the parameter $r\in(0,\frac14]$ to $s=\frac14-r\in[0,\frac14)$ so that $\sin(2r\pi)=\cos(2s\pi)\eqqcolon c_s$.

Suppose $c_{s_1}/c_{s_2}\in\QQ$.
This implies that $c_{s_1}\in\QQ(c_{s_2})$ and so $\QQ(c_{s_1})\subset\QQ(c_{s_2})$.
The same argument works in reverse, so $\QQ(c_{s_1})=\QQ(c_{s_2})$.

Consider first the case when $s_i\neq0$.
If $s_i=p_i/q_i$ with $\gcd(p_1,q_i)=1$, let us denote $\hat s_i\coloneqq 1/q_i$.
The assumption $s_i<\frac14$ implies that $q_i\geq5$.

By~\cite{WZ:deg-of-trig-numbers} the algebraic numbers $c_{s_i}$ and $c_{\hat s_i}$ have the same minimal polynomial.
Therefore
\begin{equation}
\RR\cap\QQ(\zeta_{q_1})
=
\QQ(c_{\hat s_1})
=
\QQ(c_{s_1})
=
\QQ(c_{s_2})
=
\QQ(c_{\hat s_2})
=
\RR\cap\QQ(\zeta_{q_2})
.
\end{equation}
By lemma~\ref{lma:real-cyclotomic} this implies that one of the following holds:
\begin{enumerate}
\item $q_1=q_2$.
\item $q_1=2q_2$ and $q_2$ is odd.
\item $q_2=2q_1$ and $q_1$ is odd.
\end{enumerate}
The other cases of the lemma are ruled out by the property $q_i\geq5$, and the third case listed above can be ignored by symmetry.

Let $P_q$ be the set of numbers $p$ with $1\leq p<q/2$ and $\gcd(p,q)=1$.
We have $\#P_q=\phi(q)/2$.
The cosines $c_{p/q}$ with $p\in P_q$ share the same minimal polynomial and they form a basis for $\QQ(c_{1/q})$ over $\QQ$.
They are therefore linearly independent over $\QQ$.

Consider first the case $q_1=q_2\eqqcolon q$.
The numerators $p_1$ and $p_2$ are distinct elements of $P_q$ by assumption.
The linear independence of the cosines contradicts $c_{s_1}/c_{s_2}\in\QQ$.

Consider then the case $q_1=2q_2\eqqcolon 2q$ with odd $q$.
Because $p_1$ and $q$ are odd, there is an integer $\tilde p_1$ so that $p_1=q_2-2\tilde p_1$.
It satisfies $\tilde p_1>q/4>p_2$.
The trigonometric identity $\cos(x)=-\cos(\pi-x)$ gives $c_{p_1/q_1}=c_{(q_2-2\tilde p_1)/2q_2}=-c_{\tilde p_1/q_2}$.
Because $\tilde p_1\neq p_2$, the argument of the previous case shows that the ratio of the cosines is irrational, a contradiction.

The only remaining case is $s_i=0$.
Without loss of generality, take $s_1=p/q>0$ (with $\gcd(p,q)=1$ and $q\geq5$) and $s_2=0$.
We have $c_0=1$, so we are led to study when $c_{p/q}\in\QQ$.
Applying lemma~\ref{lma:real-cyclotomic} in the case $a=1$ and $b=q$ gives that $\QQ(c_{p/q})=\QQ$ if and only if $q\in\{1,2,3,4,6\}$.
Due to this and the bound $q\geq5$ the assumption $c_{s_1}/c_{s_2}\in\QQ$ gives $q=6$.
The only possible value left for $p$ is $1$.

We have found that $s_1=\frac{1}{6}$ and $s_2=0$ or vice versa.
This means that $r_1=\frac{1}{4}-\frac{1}{6}=\frac{1}{12}$ and $r_2=\frac{1}{4}-0=\frac{1}{4}$ or vice versa, as claimed.
\end{proof}

\section{Generalized Debye Expansion}\label{a: Generalized Debye}

In this appendix, we will reduce equation \eqref{e: G with p integral} into the form \eqref{eq: wave propagator form N interface in app}, which resembles a wave propagator. We evaluate $D^{-1} U_l(r)U_l(r_0)$ appearing in \eqref{e: G with p integral} in such a way as to relate it to a certain wave propagator analogous to the computation in \cite{HIKRigidity}. However, the methodology will be different and more laborious to account for the multiple scattering created by the interfaces.

\subsection{Single interface case}

For simplicity, we first consider a $2$-layered sphere with an upper layer $\Omega_+$ and a lower layer $\Omega_-$. The general case will follow easily by a recursive argument. The wave speed and density in a layer $\pm$ (region $\pm$) is $c_\pm,\rho_\pm$. We have the upper surface layer $r = \surf$ the inner boundary $r=\CMB$ (where $R = 0$ if we consider the case of a ball) and the interface $r = \refl$.

Suppose $h^{(1)}(r), h^{(2)}(r)$ are two linearly independent solutions to the second order ODE \eqref{eq: equation for U_2} not necessarily satisfying any boundary condition. They implicitly depend on $k$ and $\omega$. Suppose $r>b$ is the $+$ region and $r < b$ is the $-$ region. Write the solutions in the $\pm$ region
\beq \label{e: u_+ and u_- intro}
u_+ = S(h^{(2)}_+(r) + Ah_+^{(1)}(r)), \quad 
u_- = (1+B)(h_-^{(2)}(r) + Ch_-^{(1)}(r)) .
\eeq
For $u_-$, we think of $C$ as being determined by either an inner boundary condition or another transmission condition if there were more layers, and not by the transmission conditions at $r =\refl$, which instead determine $B$. Hence, to emphasize this point, and to make the computation cleaner, denote
\[j(r) =h_-^{(2)}(r) + Ch_-^{(1)}(r).\]
It will also be useful to consider the solutions to the simpler ODE for $V = \mu^{1/2}r U$ with Neumann boundary condition $\p_r V\restriction_{r = \surf, \CMB} = 0$.

We will think of $h_+^{(2)}$ as an incoming wave into the interface $r= b$ and $h_+^{(1)}$ as a scattered wave even though the notation is merely symbolic at this point. Without writing it explicitly, the constants are not functions of $r$ but they do depend on $l$ and $\omega$. When we make the substitution, $k= \omega p$, then we will have $A = A(\omega, p)$, and for the remaining constants as well.
The general interface conditions to leading order asymptotics as $\omega \to \infty$ are (with $d_\pm$ some parameter on each side of the interface; in our particular case, $d_\pm = \mu_\pm(\refl)$)
\begin{align*}
 Sh^{(2)}_+(b) + SAh_+^{(1)}(b) &= B j(\refl) + j(\refl)\\
d_+ Sh^{(2)'}_+(b) + d_+ASh_+^{(1)'}(b) &= d_-Bj'(\refl) + d_-j'(\refl)
\end{align*}
To ease notation, omit the evaluation at $r = \refl$ so we have
\beq
\bmat{ S h_+^{(1)} & -j\\ d_+S h_+^{(1)'} & -d_-j'} \col{A \\ B}
= \col{ j - Sh_+^{(2)}\\ d_-j' - d_+S h^{(2)'}_+}
\eeq
Then
\beq
\col{A \\ B} = \frac{1}{Sjd_+h_+^{(1)'}-Sh_+^{(1)}d_- j'}
\bmat{  -d_-j' & j\\ -d_+Sh_+^{(1)'} & Sh_+^{(1)}}
\col{ j - Sh_+^{(2)}\\ d_-j' - Sd_+ h^{(2)'}_+}
\eeq

%old computation without j
\begin{comment}
To ease notation, omit the evaluation at $r = \refl$ so we have
\beq
\bmat{ S h_+^{(1)} & -h_-^{(2)}\\ c_+S h_+^{(1)'} & -c_-h_-^{(2)'}} \col{A \\ B}
= \col{ h_-^{(1)} - Sh_+^{(2)}\\ c_-h_-^{(1)'} - c_+S h^{(2)'}_+}
\eeq
Then
\beq
\col{A \\ B} = \frac{1}{Sh_-^{(2)}c_+h_+^{(1)'}-Sh_+^{(1)}c_- h^{(2)'}_-}
\bmat{  -c_-h_-^{(2)'} & h_-^{(2)}\\ -c_+Sh_+^{(1)'} & Sh_+^{(1)}}
\col{ h_-^{(1)} - Sh_+^{(2)}\\ c_-h_-^{(1)'} - Sc_+ h^{(2)'}_+}
\eeq
So we have
\[
A = \frac{ -c_-h_-^{(2)'} (h_-^{(1)} - Sh_+^{(2)}) + h_-^{(2)}(c_-h_-^{(1)'} - c_+ Sh^{(2)'}_+)
}
{
Sh_-^{(2)}c_+h_+^{(1)'}-Sh_+^{(1)}c_- h^{(2)'}_-
}
\]
\end{comment}

Thus,
\[
A = \frac{ -d_-j^{'} (j - Sh_+^{(2)}) + j(d_-j^{'} - d_+ Sh^{(2)'}_+)
}
{
jd_+Sh_+^{(1)'}-Sh_+^{(1)}d_- j^{'}
}
=
\frac{
d_-j'h^{(2)}_+ - d_+j h^{(2)'}_+
}{
jd_+h_+^{(1)'}-h_+^{(1)}d_- j^{'}
}
\]
Factor out $jd_-h^{(2)}_+$ in the numerator and
$jd_-h^{(1)}_+$ in the denominator to get
\[
\frac{h^{(2)}_+}{h^{(1)}_+}
\cdot
\frac{
\ln'(j) - (d_+/d_-)\ln' h^{(2)}_+
}
{(d_+/d_-)\ln' h^{(1)}_+ -\ln'(j)
}
\]

Let us use the notation
\begin{align*}
[2+] &=  (d_+/d_-)\ln' h^{(2)}_+\\
[1+] &=  (d_+/d_-)\ln' h^{(1)}_+ \\
[\alpha] &= \ln'(j)
\end{align*}
We then have
\beq \label{e: initially solving for A}
A = -\frac{h^{(2)}_+}{h^{(1)}_+}
\cdot \frac{  [2+]-[\alpha]}{[1+] - [\alpha]}
\eeq

Following \cite[Appendix]{Nussenzveig69}, we solve for reflection and transmission coefficients in terms of the above functions. We write
\[
u_+ = h^{(2)}_+(r)/h^{(2)}_+(b) + R_{++} h^{(1)}_+(r)/h^{(1)}_+(b)
\]
and
\[
u_- = T_{+-}h^{(2)}_-(r)/h^{(2)}_-(b)
\]
Here, $u_+$ and $u_-$ are solutions unrelated to the previous $u_{\pm}$. One should think of them as being defined only locally near an interface. The notation is that $R_{++}$ is reflection from above the interface and $T_{+-}$ is transmission from the upper layer $(+)$ to the lower layer $(-)$.
The transmission conditions give
\begin{align*}
1 + R_{++} &= T_{+-} \\
[2+] + R_{++}[1+] &= T_{+-}[2-]
\end{align*}
Thus
\[
\bmat{1 & -1 \\ [1+]& -[2-]}\col{ R_{++} \\ T_{+-}}
= \col{ -1 \\ -[2+]}
\]
So
\begin{align*}
\col{ R_{++} \\ T_{+-}}
&= \frac{1}{[1+]-[2-]}\bmat{-[2-] & 1 \\ -[1+]& 1}\col{ -1 \\ -[2+]}
\\
&= \frac{1}{[2-]-[1+]}\bmat{-[2-] & 1 \\ -[1+]& 1}\col{ 1 \\ [2+]}
\\
&=\col{\frac{[2+]-[2-]}{[2-]-[1+]} \\ \frac{[2+]-[1+]}{[2-]-[1+]}}
\end{align*}
So
\[
R_{++} = -\frac{[2+]-[2-]}{[1+]-[2-]} \qquad T_{+-} = \frac{[1+]-[2+]}{[1+]-[2-]}
\]

Likewise, we can show that
\beq \label{e: R_-- and T_-+}
R_{--} = -\frac{[1+]-[1-]}{[1+]-[2-]} \qquad T_{-+} = \frac{[1-]-[2-]}{[1+]-[2-]}
\eeq

\subsection*{Debye expansion for A in \eqref{e: initially solving for A}}
Using \eqref{e: initially solving for A} and the formulas for reflection and transmission coefficients, we follow Nussenszweig \cite{Nussenzveig69}(all functions are evaluated at $r =\refl$ without explicitly writing this for readability):
\begin{align*}
\frac{h^{(1)}_+}{h^{(2)}_+} A - &R_{++}\\
&= \frac{[\alpha]-[2+]}{[1+]-[\alpha]} + \frac{[2+]-[2-]}{[1+]-[2-]}
\\
&=\frac{
([1+]-[2-])([\alpha]-[2+]) + ([2+]-[2-])([1+]-[\alpha])
}
{
([1+]-[\alpha])([1+]-[2-])
}
\\
%&=\frac{
%[1+][\alpha]-[1+][2+]-[2-][\alpha]+[2-][2+] + [2+][1+]-[2+][\alpha]-[2-][1+]+[2-][\alpha]
%}
%{
%([1+]-[\alpha])([1+]-[2-])
%}
%\\
&=\frac{
[1+][\alpha]+[2-][2+] -[2+][\alpha]-[2-][1+]
}
{
([1+]-[\alpha])([1+]-[2-])
}
\\
&=
\frac{
-[2+]([\alpha]-[2-]) + [1+]([\alpha]-[2-])
}{
([1+]-[\alpha])([1+]-[2-])
}
\\
&=
\frac{
([1+]-[2+])([\alpha]-[2-])
}{
([1+]-[\alpha])([1+]-[2-])
}
= T_{+-}\frac{
[\alpha]-[2-]
}{
[1+]-[\alpha]
}
\end{align*}

Next, we use
\[
[\alpha] = \frac{ Ch^{(1)'}_- + h^{(2)'}_-}{ Ch^{(1)}_- + h^{(2)}_-}.
\]
After some algebra, we eventually get
\begin{align*}
&T_{+-}\frac{ Ch^{(1)}_-( [1-]-[2-])}
{Ch^{(1)}_-([1+] - [1-]) + h^{(2)}_-([1+]-[2-])}
\\
&=T_{+-}C \frac{h^{(1)}_-( [1-]-[2-])}{h^{(2)}_-([1+]-[2-])}
\frac{ 1}
{1+\frac{Ch^{(1)}_-([1+] - [1-])}{ h^{(2)}_-([1+]-[2-])}}
\\
&= T_{+-}CT_{-+}\frac{h^{(1)}_-}{h^{(2)}_-}\frac{1}{1-C\frac{h^{(1)}_-}{h^{(2)}_-}R_{--}}
\\
&=T_{+-}CT_{-+}\frac{h^{(1)}_-}{h^{(2)}_-}
\sum_{p=0}^\infty \left(C\frac{h^{(1)}_-}{h^{(2)}_-}R_{--} \right)^p.
\end{align*}
Thus, we obtain the formula we wanted
\beq \label{e: formula for A}
A = \frac{h_+^{(2)}}{h^{(1)}_+}R_{++}
+ \frac{h_+^{(2)}}{h^{(1)}_+}T_{+-}CT_{-+}\frac{h^{(1)}_-}{h^{(2)}_-}
\sum_{p=0}^\infty \left(C\frac{h^{(1)}_-}{h^{(2)}_-}R_{--} \right)^p.
\eeq

The formula above has a very intuitive geometric meaning. The first term represents the first reflection from the top layer. The term $T_{+-}T_{-+}$ represents transmission into the next layer and back out. For two interfaces, $C$ will be $1$, but in general, it will be a reflection coefficients defined recursively using the next adjoining layer. The $R_{--}$ represents reflection from below the interface and the $C$ corresponds to a reflection from the next subsequent interface from above which in this case is at $r = \CMB$. The exponent $p$ corresponds to how many such reflections occur before the wave transmits back to the upper layer. We will see later that terms such as $\frac{h_-^{(1)}}{h_-^{(2)}}$ correspond to travel times of each such interaction between two adjacent hypersurfaces (each being either a boundary or an interface).

\subsection*{Evaluating $1/D$}

To proceed, we will use the asymptotic solutions to the ODE where the prefix $n$ means it satisfies the Neumann boundary conditions. Asymptotically, we must distinguish the various regimes for the different types of rays that may occur: reflecting, turning, grazing, gliding, evanescent, as well as combinations of these.
\begin{itemize}
\item \textbf{Reflecting} ($0< p < \CMB/c(\CMB)$):
We use the linearly independent solutions to the ODE in the reflecting regime of the form  (see \cite[Appendix A]{HIKRigidity})
\begin{align*}\label{e: h_+,h_- for reflecting}
h^{(2)}_{+,n} &=
        \mu_+^{-1/2}r^{-1}\wkb_+^{-1/2}\exp\left( \ii \omega_n
       \int_{\refl}^r \wkb_+ \dd r' +\ii \delta_+/2\right)
\\
h^{(1)}_{+,n} &= \mu_+^{-1/2}r^{-1}\wkb_+^{-1/2}\exp\left( -\ii \omega_n
       \int_{\refl}^r \wkb_+ \dd r' -\ii \delta_+/2\right),
       \\
h^{(2)}_{-,n} &= 
        \mu_+^{-1/2}r^{-1}\wkb_+^{-1/2}\exp\left( \ii \omega_n
       \int_{R^*}^r \wkb_- \dd r' +\ii \delta_-/2\right)
\\
h^{(1)}_{-,n} &= \mu_-^{-1/2}r^{-1}\wkb_-^{-1/2}\exp\left( -\ii \omega_n
       \int_{R^*}^r \wkb_- \dd r' -\ii \delta_-/2\right),
\end{align*}
where $\delta_\pm$ is a function depending on $p$ that keeps track of phase changes for when a ray turns, and $R^*$ is the radius of the ray. For the reflecting regime where the ray never turns, then $\delta_\pm = 0$ and $R^* = \CMB$. For a general eigenfunction that does not necessarily satisfy the boundary conditions, we remove the subscript $n$ from the above definitions.

It is useful to note that in this reflecting regime, the transmission coefficients are independent of the frequency $\omega$ to leading order.
Indeed, notice
\[
[2+] = \ln ' h_+^{(2)} = \ii \omega \wkb_+
+ \ln'(\mu_+^{-1/2}r^{-1}\wkb_+^{-1/2})
\]
with similar formulas for the other terms $[2-], [1+],[1-]$. Then for $R_{++}$ (say), the $\omega$ in the first term above cancels from the numerator and denominator in the formula for $R_{++}$ so we have
\[
R_{++} = F(r) + O(1/\omega)
\]
where $F$ is independent of $\omega$. Analogous results hold for the remaining reflection and transmission coefficients so we conclude
\begin{lemma}\label{l: R/T independent of omega}
To leading order as $\omega \to \infty$, the reflection and transmission coefficients are independent of $\omega$.
\end{lemma}

At the outer boundary $r = \surf$, when identifying $U_n$ with its principal term in the WKB expansion, the Neumann condition is $\p_r U_n(\surf) = 0$, which gives when $\omega \to \infty$
\begin{equation}
\exp\left( \ii \omega_n
       \int_{\refl}^\surf \wkb_+ \dd r' +\ii \delta_+/2\right)
       -A \exp\left( -\ii \omega_n
       \int_{\refl}^\surf \wkb_+ \dd r' - \ii \delta_+/2\right)
       =0
\end{equation}

Then
\begin{align*}
(1/S)U_n(\surf) &=  \mu_+^{-1/2}r_s^{-1}\wkb_+^{-1/2} \exp\left( \ii {}_n\omega_l
       \int_{\refl}^\surf \wkb_+ \dd r' +\ii \delta_+/2 \right)
       \\&\qquad + A \mu_+^{-1/2}r_s^{-1}\wkb_+^{-1/2} \exp\left( -\ii {}_n\omega_l
       \int_{\refl}^\surf \wkb_+ \dd r' -\ii \delta_+/2 \right)
       \\
       &=2 \mu_+^{-1/2}r_s^{-1}\wkb_+^{-1/2}\exp(i{}_n\omega_l \tau(\surf)+\ii \delta_+/2).
\end{align*}

Recall that for the calculation of $D$ that we need in \eqref{e: hat G with general efunctions}, we replace the above $\omega_n$ by a general $\omega$ due to the contour integration and residue formula.
Then for $\omega \to \infty$
\begin{align*}
(1/S^2)&U_n(\surf) T(\surf) \\
&= (1/S^2)\mu_+(\surf) U_n \tfrac{d}{dr} U \\
&= 2\mu_+(\surf)\wkb_+^{-1}\mu^{-1}_+(\surf)\exp(i\omega \tau(\surf)+\ii \delta_+/2)  \\
&\qquad \qquad \cdot[\exp(i\omega \tau(\surf) + \ii \delta_+/2)
       - A \exp(-\ii\omega \tau(\surf)-\ii \delta_+/2)
       ]\\
&=
2\wkb_+^{-1}\exp(2i\omega \tau(\surf)+\ii \delta_+)
(1-A \exp(-2\ii\omega \tau(\surf)-\ii \delta_+))
\\
&=
2\wkb_+^{-1} \frac{h^{(2)}_+(\surf)}{h^{(1)}_+(\surf)}
\left(1 - A \frac{h^{(1)}_+(\surf)}{h^{(2)}_+(\surf)}\right)
\end{align*}

We finally obtain the expression for $r,r' > \refl$
\begin{align*}
D^{-1}U_l(r)U_l(r') = \frac{(h^{(2)}_+(r) + Ah_+^{(1)}(r))(h^{(2)}_+(r') + Ah_+^{(1)}(r'))}{2\wkb_+^{-1}(\surf)  \frac{h^{(2)}_+(\surf)}{h^{(1)}_+(\surf)}
\left(1 - A \frac{h^{(1)}_+(\surf)}{h^{(2)}_+(\surf)}\right)}
\end{align*}

It is convenient to set
\[
f(r,r'):=\wkb^{-1/2}(r)\wkb^{-1/2}(r')r^{-1}(r')^{-1}\mu_+^{-1/2}(r)\mu_+^{-1/2}(r')
\]
and let us label
\beq \label{e: Phi_+}
\Phi_+= \Phi_+(\omega,p) =  \int_{\refl}^\surf \wkb_+ \dd r' +\delta_+/(2\omega)
\eeq
and
\beq \label{e: Phi_-}
\Phi_-= \Phi_-(\omega,p) =  \int_{R^*}^\refl \wkb_- \dd r' +\delta_-/(2\omega),
\eeq
where $R^*$ is the turning radius of the ray as in \cite{HIKRigidity}, and for the reflecting regime, $R^* = \CMB$.
 Now $h^{(2)}_+(\surf)/h^{(1)}_+(\surf) = \exp(\ii 2\omega\Phi_+)$.
Observe that in the formula for $A$, $\frac{h^{(2)}_+}{h_+^{(1)}} = 1$.
We thus have $D^{-1}U_l(r)U_l(r')$ is equal to
\begin{align*}
&E \wkb_+(\surf)f(r,r')/2\exp\left[\ii \left(\omega\int_{\refl}^r \wkb_+ \dd r +
\omega\int_{\refl}^{r'} \wkb_+ \dd r + \delta_+ -2\omega\Phi_+\right)\right]
\\
&+E \wkb_+(\surf)f(r,r')/2\exp\left[\ii \left(\omega\int_{r'}^r \wkb_+ \dd r \right)\right]
\\
&+E \wkb_+(\surf)f(r,r')/2\exp\left[\ii \left(\omega\int_{r}^{r'} \wkb_+ \dd r \right)\right]
\\
&+
E A\wkb_+(\surf)f(r,r')/2\exp\left[\ii \left(-\omega\int_{\refl}^r \wkb_+ \dd r -
\omega\int_{\refl}^{r'} \wkb_+ \dd r - \delta_+ \right)\right],
\end{align*}
where
\begin{equation}\label{e: definition of E}
E = \left(1 - A \frac{h^{(1)}_+(\surf)}{h^{(2)}_+(\surf)}\right)^{-1}
= \sum_{l_0=0}^{\infty} A^{l_0} \exp(-2\ii l_0\Phi_+).
\end{equation}
The first term in the sum has no $A$ in the coefficient since it represents going from $r$ to $r'$ via $r=\surf$ with no interface interaction. The next two terms correspond to a direct ray from source and receiver located at radius $r$ and $r'$. The fourth term represent a path from $r$ to $r'$ via an interface interaction.

Next,
\[
A \frac{h^{(1)}_+(\surf)}{h^{(2)}_+(\surf)}= \exp(-2\ii \Phi_+) R_{++}
+ \frac{T_{+-}CT_{-+}\exp(-2\ii(\Phi_+ + \Phi_-))}
{1 - \exp(-2\ii\Phi_-) C R_{--}}
\]
To ease notation, note that terms like $\exp(-2\ii \Phi_\pm)$ correspond to 2-way radial travel times between two interfaces (one of which could be a boundary). So denote
\begin{align*}
\tilde R_{++} &= \exp(-2\ii \Phi_+) R_{++}
\qquad \qquad \  \tilde R_{--} = \exp(-2\ii \Phi_-) R_{--}\\
\qquad \tilde T_{+-} &=\exp(-\ii(\Phi_+ + \Phi_-))T_{+-}
\qquad \tilde T_{-+} =\exp(-\ii(\Phi_+ + \Phi_-))T_{-+}
\end{align*}

We then have
\begin{equation}\label{eq: A^l with tilde R and tilde T}
\left(A \frac{h^{(1)}_+(\surf)}{h^{(2)}_+(\surf)}\right)^{l_0}
= \sum_{l_1=0}^{l_0}\binom{l_0}{l_1}\tilde R^{l_0-l_1}_{++} \frac{(\tilde T_{+-}C\tilde T_{-+})^{l_1}}{(1-C\tilde R_{--})^{l_1}}
\end{equation}
Next observe that for a positive integer $q$
\[
\sum_{k=0}^\infty \binom{q+k-1}{k}z^k = \frac{1}{(1-z)^q}
\]

Hence, the above sum becomes
\beq \label{eq: sum with tilde R and tilde T}
\sum_{l_1=0}^l
\sum_{l_2=0}^\infty
\binom{l_0}{l_1}\binom{l_1+l_2-1}{l_2}
\tilde R^{l_0-l_1}_{++} (\tilde T_{+-}C\tilde T_{-+})^{l_1}(C\tilde R_{--})^{l_2}
\eeq
The boundary condition at $r = \CMB$ forces $C=1$ here.
We have shown that
\begin{align} \label{Debye expansion of E}
E &= \sum_{l_0=0}^\infty\sum_{l_1=0}^{l_0}
\sum_{l_2=0}^\infty
\binom{l_0}{l_1}\binom{l_1+l_2-1}{l_2}
\tilde R^{l_0-l_1}_{++} (\tilde T_{+-}C\tilde T_{-+})^{l_1}(C\tilde R_{--})^{l_2} \nonumber\\
&=\sum_{(m_0,m_1,m_2) \in \mathbb N^3}
n_{(m_0,m_1,m_2)}\exp(-2\ii \omega (m_0 \Phi_+ + m_1(\Phi_+ + \Phi_-)
+m_2 \Phi_-)) \cdot \\
&\qquad \qquad \qquad \qquad \qquad \qquad R_{++}^{m_0} (T_{+-}T_{-+})^{m_1}R_{--}^{m_2},
\end{align}
where $n_{(m_0,m_1,m_2)}$ is a combinatorial coefficient.
\begin{remark}\label{rem: Hron connection}
We note that the combinatorial coefficients in the above formula have a special physical meaning. In \cite{HronCriteria}, the author describes that in multilayered media, multiple waves travel different paths but arrive with the same travel times (kinematic analogs). Some of these can also have the same amplitude and phase characteristics (dynamic analogs). Hence, due to multiple scattering, there may be multiple waves with the same principal amplitude and travel time. If the corresponding ray is periodic, then all of these rays make a contribution to the trace and they are accounted for by the above combinatorial coefficient on the number of dynamic analogs for a particular ray. See \cite[Figure 2]{HronCriteria} for examples of these dynamic analogs. The coefficients in the above formula agree with the simple counting argument in \cite{HronCriteria} for counting the number of dynamic analogs.
\end{remark}

\subsection*{Radial travel times and amplitudes}Let us do the purely reflecting case first with no turning points since that is easier to index.
Based on the above calculations, we want a convenient indexing to represent radial travel times and amplitudes of each wave constituent in the sum, and can be unified when we study the other regimes.

For $M=(m_0,m_1,m_2) \in \mathbb{Z}_{\geq 0}^3$ let
\[
\Phi_M = 2m_0\Phi_+ + 2m_1(\Phi_+ + \Phi_-) + 2m_2 \Phi_-
\]
\begin{eqnarray*}
   \tau_{M,1}(r,r_0;p) &=&
         \int_{r_0}^r \wkb(r';p) \dd r'  + \Phi_M
\\
   \tau_{M,2}(r,r_0;p) &=&
        \int_{\refl}^{r_0} \wkb(r';p) \dd r'
      + \int_{\refl}^r \wkb(r';p) \dd r'  + \Phi_M,
\\
   \tau_{M,3}(r,r_0;p) &=&
        \int_{r_0}^{\surf}
                          \wkb(r';p) \dd r'
      + \int_r^{\surf}
                          \wkb(r';p) \dd r' + \Phi_M,
\\
   \tau_{M,4}(r,r_0;p) &=&
      - \int_{r_0}^r \wkb(r';p) \dd r'  +\Phi_M ,
\end{eqnarray*}

Now we have corresponding amplitudes:
\begin{eqnarray}
Q_{M,1} &=& R_{++}^{m_0}(T_{+-}CT_{-+})^{m_1}R^{m_2}_{--} \\ \\
\label{e: defining Q_M,i}
Q_{M,2} &=& AR_{++}^{m_0}(T_{+-}CT_{-+})^{m_1}R^{m_2}_{--}\\ \\
Q_{M,3} &=& R_{++}^{m_0}(T_{+-}CT_{-+})^{m_1}R^{m_2}_{--}\\ \\
Q_{M,4} &=& R_{++}^{m_0}(T_{+-}CT_{-+})^{m_1}R^{m_2}_{--}
\end{eqnarray}
Actually, we have to index more carefully since the amplitudes involving $A$ and $A^2$ above are not merely amplitudes but contain important phase information as well as seen in \eqref{eq: A^l with tilde R and tilde T}. Since this does not affect the main argument and only makes the index cumbersome, we opt not to do this.

Also, we note here that by forcing $C=1$, this enforces the inner Neumann boundary condition to leading order even on these generic solutions $U_k$. We choose to leave a generic $C$ in the formula since it will be needed for the case of multiple interfaces.

\subsection*{Substituting the Debye expansion for $\hat G$ in the 2
interfaces case}

We will now insert the Debye expansion into the formula for $\hat G(r,r_0,\omega)$ in \eqref{e: G with p integral}.

First, we insert the leading order expansion (valid for $\text{Re}\ p >0$),
\[
   Q_{\omega p - 1/2}^{(1)}(\cos \Theta) \simeq
     \left( \frac{1}{2\pi \omega p \sin \Theta} \right)^{1/2}
              e^{-\ii (\omega p \Theta - \pi/4)}
\]
to obtain, assuming $r,r_0 > \refl$ (analogous formulas hold for $r > \refl, r_0 < \refl$ or $r,r_0 < \refl$ or $r<\refl, r_0>\refl$)
\joonas{Repeated label. See the other copy of this footnote.}
\begin{multline}\label{eq: wave propagator form 2 interface}
    \frac{1}{4\pi}\ (-)^{(s-1)/2}
   (rr_0c^{(+)}(r)c^{(+)}_0(r))^{-1}(\rho^{(+)}(r)\rho^{(+)}(r_0))^{-1/2}
   \\
   \cdot \int_{-\infty}^{\infty}
   (\wkb_+(r;p)\wkb_+(r_0;p))^{-1/2}
\\
\cdot \left[ \mstrut{0.7cm} \right. \sum_{M=(m_0,m_1,m_2) \in \mathbb{N}^3}n_M\sum_{i=1}^{4}
          \exp\left[-\ii \omega \tau_{M,i}(r,r_0;p)
   + \ii N_{M,i} \frac{\pi}{2}\right] \left. \mstrut{0.7cm} Q_{M,i}(p)\right]
\\
   \quad\quad\quad\quad
          \cdot Q_{\omega p - 1/2}^{(1)}(\cos \Theta)
          e^{-\ii \omega(s-1) p \pi} \, p^{-1} \dd p
          \\
   \simeq \frac{1}{4\pi} (-)^{(s-1)/2}
       (r r_0 c^{(+)}(r) c^{(+)}(r_0))^{-1}
       (2\pi \rho^{(+)}(r) \rho^{(+)}(r_0) \sin \Theta)^{-1/2}
\hspace*{3.0cm}
\\
   \int (\wkb_+(r;p) \wkb_+(r_0;p))^{-1/2}
\\
\sum_{M=(m_0,m_1,m_2) \in \mathbb{N}^3}n_M\sum_{i=1}^{4}
           \exp[-\ii \omega (\tau_{M,i}(r,r_0;p) + p \Theta + (s-1) p \pi)]Q_{M,i}(p)
\\
   \exp[\ii (\pi/4) (2 N_{M,i} - 1)] (\omega p)^{-3/2} \dd p .
\end{multline}

The other regimes require only slight modifications to the computation above so we will be briefer on these since the notation can be unified to produce the above formula as well.
\item
\textbf{Total internal reflection} ($ \refl/c_-(\refl)< p < \refl/c_+(\refl)$):
In this case, we have a reflection from the interface with no transmission, which corresponds to an evanescent wave in $\Omega_-$.

Here $V_-$ will be evanescent with the form
\[
T \abs{\wkb_-}^{-1/2} \exp\left(-\omega \int_r^\refl \abs{\wkb_-} \dd r  \right)
\]
The reflection coefficients are computed identically except that in this case, $h_-^{(2)} = \exp\left(-\omega \int_r^\refl \abs{\wkb_-} \dd r  \right)$ so that the reflection and transmission coefficients are now complex valued.

The remaining formulas follow as before but are simpler since there is no propagation in the lower layer. Thus, we have the same formulas with $A$ replaced by $R_{++}$.
With $r,r' > \refl$, we have
\begin{align*}
D^{-1}&U_l(r) U_l(r')
=\\
&E \wkb_+(\surf)f(r,r')/2\exp\left[\ii \left(\omega\int_{\refl}^r \wkb_+ \dd r +
\omega\int_{\refl}^{r'} \wkb_+ \dd r + \delta_+ -2\Phi_+\right)\right]
\\
&+E \wkb_+(\surf)f(r,r')/2\exp\left[\ii \left(\omega\int_{r'}^r \wkb_+ \dd r \right)\right]
\\
&+E \wkb_+(\surf)f(r,r')/2\exp\left[\ii \left(\omega\int_{r}^{r'} \wkb_+ \dd r \right)\right]
\\
&+
E R_{++}\wkb_+(\surf)f(r,r')/2\exp\left[\ii \left(-\omega\int_{\refl}^r \wkb_+ \dd r -
\omega\int_{\refl}^{r'} \wkb_+ \dd r - \delta_+ \right)\right],
\end{align*}
where
\[
E = \left(1 - R_{++} \frac{h^{(1)}_+(\surf)}{h^{(2)}_+(\surf)}\right)^{-1}
= \sum_{l_0=0}^{\infty} R_{++}^{l_0} \exp(-2\ii l_0\Phi_+)
\]
and $\delta_+ = 0$ since there is no phase shift.

\item
\textbf{Diving} ($\refl/
c_+(\refl) < p < \surf/c_+(\surf)$ or $\CMB/
c_-(\CMB) < p < \refl /
c_-(\refl)$):
There are two possible cases, either the turning point is in the first layer (in this case, there will not be any reflections and transmissions so the analysis reduces to that of \cite{HIKRigidity}) or the turning point is in the second layer, which requires further analysis than that of \cite{HIKRigidity}.
In the latter case, the rays in the first layer transmit into the second layer but then turn rather than reflect  from $r = \CMB$.

 We summarize
the WKB solution of~\eqref{e: ode for U} in the vicinity of a general
turning point. A turning point, $r = \rturn$, is determined by
\[
   \wkb_-^2(\rturn) = 0 .
\]
Near a turning point, $r \approx \rturn$, and
\[
   \wkb_-^2(r) \simeq \tConst (r - \rturn),
\]
for an $\tConst$ determined by a Taylor expansion.
Away from a turning point,
\[
   \wkb^2_- > 0\text{ if $r \gg \rturn$} ,\quad
   \wkb^2_- < 0\text{ if $r \ll \rturn$} .
\]
Matching asymptotic solutions yields
\begin{equation}
   %B \left\{ \ba{lcl}
   %\displaystyle{
   %\abs{\wkb}^{-1/2} \exp\left(-\omega
   %\int_r^{\rturn} \abs{\wkb} \, \dd r\right)
   %},
         %&& r \ll \rturn
%\\[0.5cm]
   %2 \pi^{1/2} \alpha_0^{-1/6} \omega^{1/6}
         %\mathrm{Ai}(- \omega^{2/3} \alpha_0^{1/3} (r - \rturn)),
         %&& r \simeq \rturn
%\\[0.25cm]
   %\displaystyle{
   %2 \wkb^{-1/2} \cos \left(-\omega \int_{\rturn}^r \wkb \, \dd r
                              %- \pi/4 \right)
   %},
         %&& r \gg \rturn.
   %\ea\right.
%END OLD VERSION
B
\begin{cases}
\abs{\wkb_-}^{-1/2} \exp\left(-\omega
\int_r^{\rturn} \abs{\wkb_-} \, \dd r\right)
,
& r \ll \rturn
\\[.5em]
2 \pi^{1/2} \tConst^{-1/6} \omega^{1/6}
\mathrm{Ai}(- \omega^{2/3} \tConst^{1/3} (r - \rturn)),
& r \simeq \rturn
\\[.5em]
2 \wkb_-^{-1/2} \cos \left(-\omega \int_{\rturn}^r \wkb_- \, \dd r
   - \pi/4 \right)
,
& r \gg \rturn
.
\end{cases}
\end{equation}
From these one can obtain a uniform expansion, that is, the Langer
approximation
\begin{equation}\label{eq: Langer}
\begin{split}
%\hspace*{0.75cm}
   V_-(r,\omega;p) &= 2 \pi^{1/2} \chi^{1/6} (-\wkb_-^2)^{-1/4}
                   \mathrm{Ai}(\chi^{2/3}(r)) ,\\
   \chi(r) &= -(3/2) \omega \int_{\rturn}^r (-\wkb_-^2)^{1/2} \dd r ,
\end{split}
\end{equation}
valid for $r \in [\CMB,\surf]$. One obtains eigenfunctions
corresponding with turning rays.

Up to leading order, where
$r \gg \rturn$,
\begin{align}
   V_- &= 2 B \wkb_-^{-1/2}
       \cos \left(\omega \int_{\rturn}^r \wkb_- \, \dd r'
                              - \pi/4 \right),
\\
   \partial_r V_- &= -2 \omega B \wkb_-^{1/2}
       \sin \left(\omega \int_{\rturn}^r \wkb_- \, \dd r'
                              - \pi/4 \right).
\end{align}

\subsection{Gliding and grazing cases}
Recall that $u_+ = S(A_{+,2,r} + A A_{+, 1, r})$
If $u_+$ is an eigenfunction, then the surface boundary condition $u_+'(\surf)= 0$ determines the eigenvalues so
\[
u_{n,+}(r) = h^{(2)}_+(r)-\frac{h^{(2)'}_+(\surf)}{h^{(1)'}_+(\surf)} h^{(1)}_+(r)
\]
and 
\[
u_{n,+}(\surf) =
\frac{1}{h_+^{(1)'}(\surf)}W(h_+^{(1)}, h^{(2)}_+).
\]
where $W(\cdot, \cdot)$ is the Wronskian. We then obtain
\[
U_+T
= \frac{h^{(2)'}_+(\surf)}{h^{(1)'}_+(\surf)}W(h_+^{(1)}, h^{(2)}_+)
\left(1+ A \frac{h^{(1)'}_+(\surf)}{h^{(2)'}_+(\surf)} \right)
\]

Notice that $V_-$ is asymptotically $0$ near $r = \CMB$ and so the inner boundary condition is satisfied automatically.

With the above representation and following the ansantz in \eqref{e: u_+ and u_- intro}, one uses $C = 1$, and after expressing the cosine term in $V_-$ in terms of complex exponentials, we use for $h^{(1)}_-$ and $h^{(2)}_-$
\[
\exp\left(\ii \omega \int_{\rturn}^r \wkb_- \, \dd r'
                              - \ii\pi/4 \right), \qquad \exp\left(-\ii\omega \int_{\rturn}^r \wkb_- \, \dd r'
                              + \ii \pi/4\right)
.\]
Next, in this case there we have from \eqref{e: Phi_+} and \eqref{e: Phi_-}
that $\delta_+ = 0$ (since the first layer has no turning points) while $\delta_- = -\pi/2$ due to the turning point phase shift as in \cite{HIKRigidity}. The coefficent $A$ and $1/E$ are computed exactly as in the reflecting case except we have the extra $\pi/2$ phase shift coming from terms involving $V_-$. Hence, in the formula for $\hat G$, we have a contribution of $\pi/2$ to the KMAH index for each turning point in the ray path. More precisely, we will have $N_{M,i} = m_1+m_2+l_i$ where $l_i$ depends on the number of turning points. Hence, our earlier computation goes through where the KMAH index is the only difference.

One can see the phase shift analytically for each turning point along the ray in the formula for $A$. For example, a ray starting in the first layer that transmits into the second layer, if it transmits back to the first layer, the $\pi/2$ phase shift is accounted for in the terms $\tilde T_{-+} \tilde T_{+-}$ in \eqref{eq: sum with tilde R and tilde T}. If it reflects from below the interface, the $\tilde R_{--}$ will have the $\pi/2$ shift.\\

\item
\textbf{Reflection with gliding transmission} ($ p =\refl/c_-(\refl)$ and $\refl/c_-(\refl)< \refl/c_+(\refl)$ :

This is a case where a ray hits the interface at a critical angle so that there is a reflection but the transmitted ray begins tangent to the interface, and then propagates along the interface; inside $\Omega_-$, we are in the evanescent regime. 
In this case, $V_+$ is the same as before.
Here we can also use
\[
V_-(r, \omega; p) = 2T\pi^{1/2} \chi^{1/6}(-\wkb_-^2)^{-1/4}\text{Ai}(\chi^{2/3}(r)),
\]
while $V_+$ will be reflective
\begin{multline}
V_+ = S(h_+^{(2)}(r) + A h^{(1)}_+(r))
\\
= S\left( \wkb^{-1/2}_+(r)\exp\left( \ii \omega_n
       \int_{\refl}^r \wkb^{-1/2}_+ \dd r' \right)
       + A \wkb^{-1/2}_+(r)\exp\left( -\ii \omega_n
       \int_{\refl}^r \wkb_+ \dd r' \right)\right)
\end{multline}
from before.

Near $r = \refl$, we use the asymptotic formula $V_- \simeq 2T\pi^{1/2}q_0^{-1/6} \omega^{1/6} \mathrm{Ai}(-\omega^{2/3}(r - \refl))$
Let $S_1 = 2\pi^{1/2}q_0^{-1/6}\mathrm{Ai}(0)$ and $S_2 = 2\pi^{1/2}q_0^{1/6}\mathrm{Ai}'(0)$. The transmission conditions (for $V$) then become with $d_\pm = \mu_\pm^{1/2}(b)$
\begin{align*}
d_+^{-1}S\wkb_+^{-1/2}(1+A) &= d_-^{-1}\omega^{1/6}S_1T \\
d_+S\ii \omega \wkb_+^{1/2} (1-A) &= -\omega^{5/6} d_-S_2T
\end{align*}

So then
\[
\bmat{ d_+^{-1}S\wkb_+^{-1/2} & -d_-^{-1}\omega^{1/6}S_1 \\ -\omega d_+ \ii S \beta_+^{1/2} & \omega^{5/6}d_-S_2} \col{A \\ T} = \col{-d_+^{-1}S\wkb_+^{-1/2} \\ -d_+S\ii \omega \wkb_+^{1/2} \omega}
\]

We have
\[
\col{A \\ T} = \frac{1}{d} \bmat{ \omega^{5/6}d_-S_2 & d_-^{-1}\omega^{1/6}S_1 \\ \omega d_+ \ii S \beta_+^{1/2} & d_+^{-1}S\wkb_+^{-1/2} } \col{-d_+^{-1}S\wkb_+^{-1/2} \\ -d_+S\ii \omega \wkb_+^{1/2} \omega}
\]
where $d = \omega^{5/6}d_- d_+^{-1}\wkb_+^{-1/2}SS_2 - \omega^{7/6}d_+d_-^{-1}\ii \wkb_+^{1/2} S S_1$. Hence, we have
\[
A = \frac{-\omega^{5/6}d_-d_+^{-1}S_2S\wkb_+^{-1/2}-\omega^{7/6}d_-^{-1}d_+\wkb_+^{1/2}S_1S\ii }{\omega^{5/6}d_- d_+^{-1}\wkb_+^{-1/2}SS_2 - \omega^{7/6}d_+d_-^{-1}\ii \wkb_+^{1/2} S S_1    } \to 1 \text{ as } \omega \to \infty.
\]
Similarly, $T \sim O(\omega^{-1/6}) $ as $\omega \to \infty$.

In this case, to leading order we have
\[
E = \left(1 - A \frac{h^{(1)}_+(\surf)}{h^{(2)}_+(\surf)}\right)^{-1}
= \sum_{l_0=0}^{\infty}  \exp(-2\ii l_0\Phi_+).
\]

Note that to leading order, this is analogous to the internally reflected case with no gliding. Hence, to leading order, only the travel time of the reflected ray can be detected and not the gliding portion.

\begin{comment}
Thus, we note that the travel time of the gliding segment will not be in the phase function, which only contains the travel times of reflected portions of the ray. Thus, the phase is \emph{not} stationary at the gliding ray travel time and is smooth at this time, although not in any open neighborhood of $T_g$.
\end{comment}

\item
\textbf{Grazing} ($R^* = \refl \text{ or }\CMB$): The detailed asymptotic analysis can be found in appendix \ref{a: grazing}. For a grazing, turning ray, there are two new possibilities: The turning point is at the discontinuity $r = \refl$ or the turning point is at $r = \CMB$.

\subsubsection*{Turning point at the discontinuity} Now we suppose $R^* = \refl$. Observe that due to the extended Herglotz condition, the lower layer becomes an evanescent regime where no propagation can occur, so using the asymptotic expansion of the Airy function, $V_-$ is exponentially decreasing and so $V_- = O(\omega^{-\infty})$ in $\Omega_-$. The inner boundary condition thus becomes automatically satisfied. Nevertheless, when restricted to the interface, $V_-\restriction_{r = \refl}$ is not $0$ and is determined by the interface conditions.
Hence, the interface conditions have the form
\[
V_+ \restriction_{r= \refl} = f_1 ,\qquad \p_rV_+ \restriction_{r= \refl} = f_2
\]
for some $f_1,f_2$ depending on $p$ and $\omega$.
 Next, note that $\mathrm{Bi}(0) =\sqrt{3} \mathrm{Ai}(0)$.
 Using \eqref{e: V in glanging regime} to represent $V_+$, the interface conditions have the form
 \begin{align*}
 q_0^{-1/6}\omega^{1/6}\textrm{Ai}(0)[B_1 + \sqrt 3 B_2] &= f_1 ,\\
 q_0^{1/6}\omega^{5/6}\textrm{Ai}'(0)[B_1 - \sqrt 3 B_2] &= f_2 .
\end{align*}
Away from the turning point, as in the turning point regime, $V_+$ is a linear combination of $\exp\left(\omega \int_{\rturn}^r \wkb_+ \, \dd r'
                              - \pi/4 \right)$ and $ \exp\left(-\omega \int_{\rturn}^r \wkb_+ \, \dd r'
                              + \pi/4\right)$
 so upon replacing the factor $\sqrt{3}$ appearing in the equation for $B_1$ and $B_2$ with $e^{\ii \pi/6}+e^{-\ii \pi/6}$, this introduces an extra phase shift in the KMAH index, while the remaining calculations are analogous. Also, the formula for $A$ involving $h^{(2)}_+(\surf)/h^{(1)}_+(a)$ will also have this extra phase shift each time the ray turns due to the discontinuity.

 In fact, we can calculate ``reflection/transmission'' coefficients (that is, an analog of them since there is no reflection/transmission in this case)too see what is happening. We use the ansatz
 \[
 V_+(r)\backsimeq
\frac{\mathrm{Ai}(- \omega^{2/3} \tConst^{1/3} (r - \rturn))}{\mathrm{Ai}(0)}
+ R
\frac{\mathrm{Bi}(- \omega^{2/3} \tConst^{1/3} (r - \rturn))}{\mathrm{Bi}(0)}
\]
for some $R$ to be determined. Since $\Omega_-$ is an evanescent regime, we could use the ansatz $V_- (r) = T \exp \left( - \omega \int_r^\refl \wkb_-(r') \dd r' \right)$.
Then the same calculation for $R_{++}, T_{+-}$ earlier would give to leading order as $\omega \to \infty$, $R = 1$ and $T = 0$ as expected since to principal order, there is no transmission to the other side.

We do the explicit computations since it is interesting to see how the discontinuity affects the leading order behavior. There are two calculations to do: Computing $D$ and computing $A$.
We have
\[
V_+ = S2\pi^{1/2}\chi_+^{1/6}(-\wkb_+^2)^{-1/4}(\mathrm{Bi}(\chi_+^{2/3}(r)) + A \mathrm{Ai}(\chi_+^{2/3}(r)))
\]
To satisfy the Neumann condition, we evaluate at $r = \surf \gg \refl$ so we can use the leading order asymptotics of the Airy functions
\[
V_+\simeq S(-\wkb_+^2)^{-1/2}\left(-\sin\left(\omega \int_{\refl}^r \wkb_+ dr + \pi/4\right)+A\cos\left(\omega \int_{\refl}^r \wkb_+ dr + \pi/4\right) \right)
\]
Writing the $\sin$ and $\cos$ terms in terms of complex exponentials and using the functions defined earlier $\Phi_+$ with $\delta = \pi/4$, we have
\[
\simeq S(-\wkb_+^2)^{-1/2}\left(
(i+A)\exp(i\Phi_+) + (-i+A) \exp(-i\Phi)
\right)
\]
The Neumann condition $\p_r V_{n,+}(\surf) = 0$ gives
\[
(i+A)\exp(i\Phi_{n,+}) - (-i+A) \exp(-i\Phi_{n,+}) = 0
\]
where we now use the actual eigenvalue $\omega_n$ and eigenfunction $V_{n,+}$. Thus, we have
\[
V_{+,n}(\surf)\simeq 2S(-\wkb_+^2)^{-1/2}
(i+A)\exp(i\Phi_{n,+})
\]
As before, we replace $\omega_n$ by the general $\omega$ and let $\omega \to \infty$
\begin{align}
(1/S^2)U_n(\surf)T(\surf)
%&= 2\mu_{\surf}(i+A)\exp(i\Phi_{+})\left((i+A)\exp(i\Phi_{+}) - (-i+A) \exp(-i\Phi_{+})\right) \nonumber \\
&=2\mu_{\surf}(i+A)^2\exp(2i\Phi_{+})\left(1 - \frac{-i+A}{i+A} \exp(-2i\Phi_{+})\right).
\label{eq: U_nT for grazing}
\end{align}

Now we must compute $A$ which can be thought of now as the ``reflection'' coefficient from the interface.
Near $r = \refl$, we can use the asymptotic formula
\[
V_+ \simeq2 S\pi^{1/2} q_0^{-1/6} \omega^{1/6}
(\mathrm{Bi}(-\omega^{2/3}q_0^{1/3}(r-\refl))+ A \mathrm{Ai}(-\omega^{2/3}q_0^{1/3}(r-\refl)))
\]
It will be convenient to denote $S_1 = 2S\pi^{1/2} q_0^{-1/6}\mathrm{Ai}(0) $ and $S_2 =-S 2\pi^{1/2} q_0^{1/6}\mathrm{Ai}'(0) \mu_{+,\refl}$.
Below the interface, we use the ansantz $V_- = T \abs{\wkb_-}^{-1/2}\exp\left(-\omega \int_r^\refl \abs{\wkb_-} \dd r\right)$ which satisfies the inner boundary condition to leading order.

Then the transmission conditions at $r= \refl$ are given by
\begin{align*}
\omega^{1/6}S_1(1+\sqrt{3}A) &= \abs{\wkb_-}^{-1/2}T \\
\omega^{5/6}S_2(1-\sqrt{3}A) &= -\omega\abs{\wkb_-}^{-1/2}T
\end{align*}
which leads to the matrix equation
\[
\bmat{ \omega^{1/6}S_1\sqrt{3} & -\abs{\wkb_-}^{-1/2}\\
-\omega^{5/6}S_2 \sqrt{3} & -\omega \abs{\wkb_-}^{1/2}\mu_{-,\refl}
}\col{A \\ T} = \col{-\omega^{1/6} S_1 \\ -\omega^{5/6}S_2}.
\]
\end{itemize}
The determinant is $d = -\omega^{7/6}S_1 \sqrt{3}\abs{\wkb_-}^{1/2}\mu_{-,\refl} - \omega^{5/6} S_2 \abs{\wkb_-}^{-1/2}$ so the solution is
\[
\col{A \\ T} = \frac{1}{d} \bmat{-\omega\abs{\wkb_-}^{1/2}\mu_{-,\refl} & \abs{\wkb_-}^{-1/2} \\ \omega^{5/6}S_2 \sqrt{3} & \omega^{1/6}S_1 \sqrt{3}} \col{-\omega^{1/6}S_1 \\ - \omega^{5/6} S_2}
\]
So we obtain
\[
A = \frac{\omega^{7/6}\abs{\wkb_-}^{1/2}S_1\mu_{-,\refl} - \omega^{5/6}S_2\abs{\wkb_-}^{-1/2}}
{-\omega^{7/6}S_1 \sqrt{3} \abs{\wkb_-}^{1/2}\mu_{-,\refl} - \omega^{5/6}S_2\abs{\wkb_-}^{-1/2}}
\to -\frac{1}{\sqrt{3}}
\]
as $\omega \to \infty$, which modulo a normalization constant, the reflection coefficient is $1$ while $T = 0$ to leading order, as expected.

We can now give a more explicit asymptotic formula for \eqref{eq: U_nT for grazing}. First, note that $(-i+A)/(i+A)$ has modulus $1$ when using $A = -1/\sqrt{3}$ and angle $\arctan{-\sqrt{3}} = -\pi/3$. Hence, $(-i+A)/(i+A) = e^{-\pi/3}$, which to principal order, the extra phase shift the interface creates for the wave. Thus, we obtain
\[
(1/S^2)U_n(\surf)T(\surf)
 =2\mu_{\surf}(4/3)e^{-i2\pi/3}\exp(2i\Phi_{+})\left(1 - e^{-i\pi/3} \exp(-2i\Phi_{+})\right).
\]
Hence, in the earlier formula for $E$, we instead get
\[
E = \sum_{l_0=0}^\infty \exp(-2\ii l_0 \Phi_+)
\]
where in the definition of $\Phi_+$, $\delta_+ = \pi/6$, which is the adjusted phase shift from the turning ray that was $\pi/2$.

\subsubsection*{Turning point at $r = \CMB$}

 It is possible that for certain turning rays, $R^* = \CMB$, in which case the Neumann boundary condition $\p_r V_-\restriction_{r=\CMB=R^*} = 0$ must be satisfied as well. This condition will be satisfied by using the representation \eqref{eq: Langer} near the grazing point and also introducing $\mathrm{Bi}(x)$ in addition to $\mathrm{Ai}(x)$ above so that near $r = R^*$, $V_-$ is a linear combination of $\omega^{1/6}\mathrm{Ai}(- \omega^{2/3} \tConst^{1/3} (r - \rturn))$ and $\omega^{1/6}\mathrm{Bi}(- \omega^{2/3} \tConst^{1/3} (r - \rturn))$.
 So for $r\backsimeq R^*$, we use the ansantz
 \begin{equation}\label{e: V in glanging regime}
V_-(r)\backsimeq C_12 \pi^{1/2} \tConst^{-1/6} \omega^{1/6}[
\mathrm{Bi}(- \omega^{2/3} \tConst^{1/3} (r - \rturn))
+ C_2
\mathrm{Ai}(- \omega^{2/3} \tConst^{1/3} (r - \rturn))]
 \end{equation}
 Then \[
 \p_r V_-(r)\restriction_{r=\CMB} =
- C_12 \pi^{1/2} \tConst^{1/6} \omega^{5/6}[
\mathrm{Bi}'(0)
+ C_2
\mathrm{Ai}'(0)]
 \]
 Setting $\p_r V(r)\restriction_{r=\CMB} = 0$ and using $-\mathrm{Ai}'(0) = \mathrm{Bi}'(0)/\sqrt{3}$ we get $C_2= \sqrt{3}$.

 Now, $V_+$ is the same as the reflecting case, and the coefficient $A$ is computed identically. In the formula, the coefficient $C$ will change to account for the grazing at the boundary. To see this, for $r$ near $\refl$, we may write $V_-$ in terms of complex exponentials exactly as in the previous case where the turning point was at the interface. We obtain for $r$ near $\refl$
 \begin{align*}
 V_- &\simeq
 C_1(-\wkb_-^2)^{-1/2}\left(
(i+C_2)\exp(i\Phi_-) + (-i+C_2) \exp(-i\Phi_-)
\right)\\
&= C_1(i+C_2)(-\wkb_+^2)^{-1/2}\left(
\exp(i\Phi_-) + \frac{-i+C_2}{i+C_2} \exp(-i\Phi_-)
\right) \\
&=C_1(i+C_2)(-\wkb_+^2)^{-1/2}\left(
\exp(i\Phi_-) + e^{-\pi/3} \exp(-i\Phi_-)
\right).
 \end{align*}
Hence, in the formula for $A$ an $E$ in the previous sections, $C$ gets replaced by $e^{-\pi/3}$ which only affects the KMAH index each time the ray turns. Hence, the same computations as in the reflecting case go through with an adjusted KMAH index.

\subsection{Multiple interfaces}

We now consider an $N$-layered sphere whose wavespeed and density are smooth in each layer $j$ denoted $\Omega_j = \{ d_j < r < d_{j+1}\}$. The index $j$ of the layers increases as $j$ increases. We have $N$ discontinuities $r = d_j$, $j=1,\dots, N$ and are also indexed in increasing radius. Hence $r=d_N$ is the surface and $r=d_1$ is the core-mantle boundary. As before, we consider the reflecting regime where there are no turning points. Analysis for the other regime will extend easily from the analysis we did in the two interface case.

\begin{itemize}
\item
\textbf{Reflection } ($ 0< p < \CMB/c(\CMB)$):
We have the reflection coefficients $R_{j+1,j j+1}$ for a wave that reflects from interface $r=d_j$ from above. The first and last index indicate the wave began and ended in layer $\Omega_{j+1}$ while the middle index indicates which reflector the wave hit. Really, it is a function of $\omega$ and $l$. Then we have $T_{j+1,j}$ and $T_{j,j+1}$ as transmission from layer $j+1$ to $j$ and $j$ to $j+1$ respectively. Likewise, reflection from below interface $r=d_j$ is $R_{j, j, j}$. Corresponding to $C$ before, we will label $A_{j-1}$ corresponding to the total amplitude of outgoing waves at interface $r = d_{j-1}$.

Let $U_j = U \restriction_{\Omega_{j}}$. As before, we set
\[
U_{j+1} = S_{j+1}(h^{(2)}_{j+1}(r) + A_j h^{(1)}_{j+1}(r)) ,
\]
\[
U_j = S_j(h^{(2)}_j(r) + A_{j-1}h^{(1)}_j(r)) .
\]
Then the same calculations as before lead to
\[
A_j = R_{j+1,j,j+1} + T_{j+1,j}A_{j-1}T_{j,j+1}\sum_{p=1}^\infty (R_{j,j,j}A_{j-1})^{p-1}.
\]
We denote
\begin{align} Q = (R_{N,N-1,N}, R_{N,N,N},T_{N+1,N},T_{N,N+1}, \dots,
&R_{j+1,j,j+1}, R_{j,j,j},T_{j+1,j},T_{j,j+1},
\\&\dots,R_{2,1,2}, R_{1,1,1},T_{2,1},T_{1,2})
\end{align}
and $M= (m_1, \dots , m_{4(N-1)})$. We then define $Q_M:=Q^M$, that is, as the product of the amplitudes
\beq \label{e: Q_M term}
Q_M=R^{m_1}_{2,1,2} R^{m_2}_{1,1,1}T^{m_3}_{2,1}T^{m_4}_{1,2}  \cdots
R^{m_{4N-7}}_{N,N-1,N} R^{m_{4N-6}}_{N,N,N}T^{m_{4N-5}}_{N+1,N}T^{m_{4N-4}}_{N,N+1}.
\eeq
Note that $Q_M$ depends on $p$ but not $\omega$ using Lemma \ref{l: R/T independent of omega}.
As before, we define $Q_{M,i}$ according to \eqref{e: defining Q_M,i} with $A_i$ replacing $A$ in those formulas. The radial travel times $\tau_{M,i}$ will be constructed analogously using iteration from the two interface case. However, we do have to distinguish the different regimes to obtain the correct KMAH index.

For the radial travel times, define
\[
\Phi_{j}(\omega,p) = \int_{d_{j-1}}^{d_{j}} \wkb_j(r') \dd r' +\delta_j(p)/(2\omega).
\]
Here, $\delta_j$ depends on $p$, since if it is the reflecting regime for $\Omega_j$ (that is, the ray does not turn in $\Omega_j$, then $\delta_j = 0$. If the ray turns in $\Omega_j$ but does not graze, then $\delta_j = \pi/2$. If it grazes, then $\delta_j = \pi/12$.

For $M=(m_1,m_2,\dots,m_{4(N-4)}) \in \mathbb{Z}_{\geq 0}^{4(N-4)}$ let
\beq \label{e: Phi_m radial travel time}
\Phi_M = \sum_{j=1}^{4(N-4)} 2m_j\Phi_j.
\eeq

In such a case, iterating the calculation of \eqref{eq: A^l with tilde R and tilde T} and \eqref{eq: sum with tilde R and tilde T} to obtain $E$ in \eqref{Debye expansion of E} from the single interface case, where $C$ gets replaced by $A_{j-1}$, we will have
\[
E = \sum_{M \in \mathbb N^{4(N-1)}}
n_M Q_M \exp\left(-2\ii \omega \Phi_M \right)
\]
where $n_M$ is a combinatorial constant counting the number of dynamic analogs as in \cite{HronCriteria},

To simplify the notation, we assume that we are considering $U(r),U(r_0)$ in layer $\Omega_N$ and $r_{N-1} = \refl$; the formulas are analogous for the other layers. As before we compute $D^{-1}U_l(r)U_l(r_0)$ as
\begin{align*}
&E \wkb(\surf)f(r,r_0)/2\exp\left[\ii \left(\omega\int_{\refl}^r \wkb \dd r +
\omega\int_{\refl}^{r_0} \wkb \dd r + \delta_N -2\omega\Phi_N\right)\right]
\\
&\ \ + E \wkb(\surf)f(r,r_0)/2\exp\left[\ii \left(\omega\int_{r_0}^r \wkb \dd r \right)\right]
\\
&\ \ + E \wkb(\surf)f(r,r_0)/2\exp\left[\ii \left(\omega\int_{r}^{r_0} \wkb \dd r \right)\right]
\\
&\ \ +
E A_N\wkb(\surf)f(r,r_0)/2\exp\left[\ii \left(-\omega\int_{\refl}^r \wkb \dd r -
\omega\int_{\refl}^{r_0} \wkb \dd r - \delta_N -2\omega\Phi_N\right)\right],
\end{align*}
As before, we denote the radial travel times as
\begin{eqnarray}\label{tau_M formula N interface}
   \tau_{M,1}(r,r_0;p) &=&
         \int_{r_0}^r \wkb(r';p) \dd r'  + \Phi_M ,
\\
   \tau_{M,2}(r,r_0;p) &=&
        \int_{\refl}^{r_0} \wkb(r';p) \dd r'
      + \int_{\refl}^r \wkb(r';p) \dd r'  -2\Phi_N+ \Phi_M ,
\\
   \tau_{M,3}(r,r_0;p) &=&
        \int_{r_0}^{\surf}
                          \wkb(r';p) \dd r'
      + \int_r^{\surf}
                          \wkb(r';p) \dd r' + \Phi_M ,
\\
   \tau_{M,4}(r,r_0;p) &=&
      - \int_{r_0}^r \wkb(r';p) \dd r' +\Phi_M ,
\\
   \tau_M(p) &=& \Phi_M
\end{eqnarray}
and corresponding amplitudes by
\begin{eqnarray} \label{e: Q_Mi terms}
Q_{M,1} &=& Q_M , \\
\label{e: defining Q_M,i again}
Q_{M,2} &=& A_NQ_M ,\\
Q_{M,3} &=& Q_M ,\\
Q_{M,4} &=& Q_M .
\end{eqnarray}
As before, we would have to expand $A_N$ with a Neumann series, with each term contributing to the phase. However, the main form of the final formula does not change and so we opt not to do this in order to simplify the indexing.

\item
\textbf{Turning/gliding/grazing/total internal reflection} ($ \CMB/c(\CMB)\leq p < \surf /c(\surf)$):
There exists a minimal $l$ such that $d_{l}/c(d_{l})\leq p <d_{l+1}/c(d_{l+1})$.
For the other regimes let $R^*$ be the turning radius of the deepest ray, and it depends on $p$. It is possible that the ray internally reflects at an interface or grazes it, in which case $R^* = r_{d_l}$. Due to Herglotz, all layers above $r_{d_{l+1}}$ are reflecting and we can repeat the analysis above to compute $U_N,U_{N-1},\cdots, U_{l+3}$. Note that $\overline \Omega_{l+1}$ is where the ray turns, grazes, or internally reflects, so the region $r_{d_{j}} \leq r \leq r_{d_{j+2}}$ can be analyzed using the single interface case from before. To compute $U_{l+2},U_{l+1}$, we repeat the calculation of the single interfaces case to determine $A_{l+1}$ and $A_l$, where $A_l$ plays the role of $C$ in the two interface case and $A_{l+1}$ replaces $A$ in that case.

\end{itemize}

In all cases, we can unify the expressions so that to leading order
 we compute $\hat G(r,r_0,\Theta,\omega)$, assuming $r,r_0 > d_{N-1}$ (analogous formulas hold for the other intervals)
\begin{multline}\label{eq: wave propagator form N interface in app}
   \simeq \frac{1}{4\pi} (-)^{(s-1)/2}
       (r r_0 c^{(N)}(r) c^{(N)}(r_0))^{-1}
       (2\pi \rho^{(N)}(r) \rho^{(N)}(r_0) \sin \Theta)^{-1/2}
\hspace*{3.0cm}
\\
   \int (\wkb_N(r;p) \wkb_N(r_0;p))^{-1/2}
   \sum_{M \in \mathbb{N}^{4(N-1)}}n_M
   \\
   \cdot \sum_{i=1}^{4}
           \exp[-\ii \omega (\tau_{M,i}(r,r_0;p) + p \Theta + (s-1) p \pi)]Q_{M,i}
\\
   \exp[\ii (\pi/4) (2 N_{M,i} - 1)] (\omega p)^{-3/2} \dd p .
\end{multline}
It is important to note that $N_{M,i}$ depends on $p$, relating to a different phase shift from the various regimes described above.

\subsubsection{Proof of Proposition \ref{t: gliding ray trace}: Wave trace near a gliding ray}
\label{a: proof of gliding}

Here, we will prove Proposition \ref{t: gliding ray trace} showing the behavior of the wave trace near a gliding ray.

First, let $\gamma$ be a periodic ray with travel time $T$ that contains a gliding \branch{}. We assume that other rays with travel time $T$ have the same number of reflected/transmitted \branch{}s or differ from $\gamma$ only through a rotation. Thus, there is an $\epsilon$ such that there are no periodic rays outside of $[\gamma]$ with travel time in $[T,\epsilon)$. We prove in section \ref{s: gliding as limits} that there is a sequence of nongliding, broken turning rays $\gamma_m$, $m= 1,2,3,\dots$ converging to $\gamma$. Let $T_m$ being the travel time of these rays with ray parameter $p_m$. We would like to understand $\text{Tr}(\p_t G)\restriction_{(T-\epsilon,T+\epsilon)}$.

\begin{proof}[Proof of Proposition \ref{t: gliding ray trace}]
First, let us assume there is only a single interface at $r = \refl$. When $\gamma$ hits the interface at a critical angle, the transmitted \branch{} is tangent to the interface.
 As described in \cite[p.181]{CervenyHeadWaves}, when the angle of incidence is a little less than the critical angle, the ray of the transmitted wave has a turning point in the lower medium and later strikes the interface. It can be reflected from the interface (from below) and strike it again, and so on. Thus, the gliding wave, is a limit of waves which strike the interface from below $m=0,1,2,\dots$ times. These turning waves can be constructed with the standard WKB procedure we do in the turning regime. The limiting rays that strike the interface from below $m$ number of times are $\gamma_m$.
 There will be turning rays with travel times approaching $T$ from below that reflect from below the interface $m$ times.
  Following \eqref{eq: principal term spheric symm}, the principal  coefficient $a_m$ in the trace corresponding to this ray has the form
  \[
  a_m = C_d T_m^\sharp Q_m(p_m) \ii ^{N_m}n_m \abs{p_m^{-2} \p_p^2 \tau_m}^{-1/2}
  \]
  where $C_d$ is independent of $m$, $Q_m$ is the product of the scattering coefficients, and the other quantities are explained there. Each term above remains bounded, but $Q_m$ and $\abs{\p_p^2 \tau_m}^{-1/2}$ have decay properties that we will quantify as $ m \to \infty$.

  Since $\gamma_m$ enters the lower medium, reflects $m$ times, and exits into the upper medium, we have $Q_m(p_m) = Q_m'R_{m,--}^mT_{m,-+}$ for some uniformly bounded $Q_m'$. We showed in Appendix \ref{a: Generalized Debye} that for all non-gliding rays, the leading order contribution is $\int O(\omega^{3/2}) e^{\ii\omega(t-T_m)} \dd \omega$ while for $p=p_G$, the gliding ray, it is $$\int O(\omega^{3/2-\epsilon}) e^{\ii\omega(t-T)}\dd \omega$$
  where $\epsilon >0 $ is unknown, and it is even possible that $\epsilon = \infty$, which is essentially the case in \cite{ColinGliding, ColinClusterpoints} albeit a slightly different setting. Hence, to leading order,
\begin{multline*}
    \text{Tr}(\p_t G)(t)\restriction_J = \sum_m (t-T_m+\ii 0)^{-5/2}C_d T_m^\sharp Q'_m
    \\
    \ii ^{N_m}n_m \abs{p_m^{-1/2} \p_p^2 \tau_m}^{-1/2}R_{m,--}^mT_{m,-+}\abs{p_m^{-2} \p_p^2 \tau_m}^{-1/2} .
\end{multline*}

We must make sure this sum is finite. First, we note in \eqref{e: R_-- and T_-+} that $R_{m,--} \to -1$ as $m \to \infty$. Next,
\[
T_{m,-+} = \frac{2\mu_-(\refl)\wkb_{m,-}(\refl)}{\mu_-(\refl)\wkb_{m,-}(\refl) + \mu_+(\refl)\wkb_{m,+}(\refl)}
\]
Now, we already have $T_{m,+-} \to 0 $ as $m \to \infty$ since $\wkb_{m,-} \to 0$ but we need to know the rate this happens for the infinite sum above.
Let $\Theta_H$ be the epicentral distance the gliding \branch{} travels and $\Theta_{m,-}$ the epicentral distance of a turning segment. We know explicitly
\[
\Theta_{m,-} = 2 \int_{R^*_m}^b \frac{p_m}{(r')^2 \wkb_{m,-}} dr'
\]
where $R^*_m < \refl$ is the turning radius.
Next, we use that near the turning point, $r \approx R^*_m$, we have
\[
\wkb^2_{m,-} \simeq q_0(r-R^*_m) .
\]
Hence,
\begin{multline}
\Theta_{m,-} \simeq \frac{2p_m}{\sqrt{q_0}} \int_{R^*_m}^b \frac{1}{(r')^{-2}\sqrt{r-R^*_m}} dr' \\
\simeq \frac{2p_m}{b^2\sqrt{q_0}} \int_{R^*_m}^b \frac{1}{\sqrt{r-R^*_m} } dr' = \frac{4p_m}{b^2\sqrt{q_0}} \sqrt{b-R^*_m} \simeq  \frac{4p_m}{b^2q_0}\wkb_{m,-}(b) ,
\end{multline}
using that $R^*_m \to \refl$ as $m \to \infty$.
We also have by our construction
\[
m \Theta_{m,-} \approx \Theta_H
\]
so for large $m$, $\wkb_{m,-} = O(1/m)$ and hence $T_{m,+-} = O(1/m)$. Note that this is similar to estimate (6.17) in \cite{CervenyHeadWaves}.
Also, the radial travel $\tau_m$ has the form $\tau_m = 2\tau_m' + 2m\tau_{m,-}(\refl)$ where $\tau_m'$ remains uniformly bounded.
Hence, we obtain $\abs{\p^2_p \tau_m}^{-1/2} = O(1/\sqrt{m})$ (analogous to \cite{Bennett1982} and \cite[Section 6.1]{CervenyHeadWaves}). Thus, the sum converges.

The same argument holds in the case of multiple interfaces. The limiting principal symbol  $a_m$ will still involve a term of the form $T_{m,j,j-1} = O(1/m)$ where $r = d_j$ is the interface containing the gliding segment. In addition, the same argument above gives $\abs{\p^2_p \tau_m}^{-1/2} = O(1/\sqrt{m})$ which is all that is needed for a convergent sum.
\end{proof}

\section{Periodic grazing Ray}\label{a: grazing}
In this appendix, we will provide a more detailed analysis on the contribution of a periodic grazing ray to the trace formula. Our analysis closely follows \cite[Chapter 1]{bennett1982poisson}. We do the analysis for $p$ near the grazing value $\CMB/c(\CMB)$ and then show the minor change necessary for $p$ near the value $\refl/c(\refl)$ corresponding to grazing at the interface. We will show that the leading order (as $\omega \to \infty$) contribution will have the ``classic'' form of \eqref{eq: wave propagator form N interface in app} that can be handled with stationary phase while the lower order terms involve integrals of Airy functions where stationary phase does not apply. This is similar to the wave parametrix near a grazing ray described in \cite{MelroseGliding} involving Airy functions.

We assume $U$ satisfies the inner boundary condition and $U_n$ satisfies both boundary conditions. We will need to compute
\[
D = U_nT\restriction_{r = \surf} - U_nT\restriction_{r = R} = U_n T\restriction_{r = \surf}
\]
We then replace $\omega_n$ by a general $\omega$. Using the asymptotic computation to sum the eigenfunctions computed earlier or using the computation in \cite{ZhaoModeSum}, we have the Green's function representation

\beq \label{eq: hat G with general efunctions}
\hat G(x,x_0,\omega) = \frac{1}{2\pi}\sum_{l=0}^\infty
\frac{l + \tfrac{1}{2}}{l(l+1)} D^{-1} \mathbf{D}_l (\mathbf{D}_l)_0  P_l(\cos \Theta).
\eeq

Let $A_r$ and $B_r$ denote two linearly independent solutions to leading order for the equation \eqref{eq: equation for U_2} via solving \eqref{eq:VSchro} first. We will later pick 
\begin{equation}
\begin{split}
%\hspace*{0.75cm}
   A_r = A_r(\omega, p) &= 2 \pi^{1/2}\mu^{-1/2}r^{-1} \chi^{1/6} (-\wkb^2)^{-1/4}
                   \mathrm{A_+}(\omega^{2/3}\chi^{2/3}(r)) ,\\
   \chi(r) &= -(3/2)  \int_{\rturn}^r (-\wkb^2)^{1/2} \dd r ,
\end{split}
\end{equation}
and similarly for $B_r$ but using the Airy function $\mathrm A_-$, where $\mathrm A_\pm$ are Airy functions described in \cite{Bennett1982,MelroseGliding}.
Following similar notation as in section \ref{a: Generalized Debye} and equation \eqref{e: u_+ and u_- intro}, we write $U_n$ restricted to the first layer $\Omega_+$
\[
U^{(+)}_n = S( A_r + A B_r)
\]
for coefficients $S$ and $A$ that depend on $p$ and $\omega$, and $A$ was computed as \eqref{e: formula for A}. Similar to \eqref{e: u_+ and u_- intro}, for $U_n$ restricted to the second layer we set
\[
U^{(-)}_n = B( A_r + C B_r).
\]
We do not add the $(\pm)$ superscripts for $A_r,B_r$ since it will be clear in context based on which $r$ value we are evaluating.
Note that $A$ is computed to be the same as \eqref{e: formula for A} to satisfy the transmission conditions where $h_+^{(2)} = A_\refl$, $h_+^{(1)} = B_\refl$, and similarly for $h_-^{(1)},h_-^{(2)}$ in the formula.
The Neumann inner boundary condition to leading order is
\[
\p_r U^{(-)}_n\restriction_{r=R} = 0
\]
so
\[
U^{(-)}_n =B (A_r - \frac{A'_R}{B_R'}B_r),
\]
where a specific eigenvalue $\omega_n$ is being used, and for a radial function $D_r$, we use the notation $D'_b = \frac{d}{dr}\restriction_{r=b} D_r.$
Thus, we get 
\begin{equation}\label{e: grazing T equation}
    \frac{1}{\mu}T
    = \p_r U =S( A_r' +A B_r').
\end{equation}
Since $U_n$ is an eigenfunction, then
$\p_r U_n = 0$ at $r = \surf$ gives
\[
A_\surf' + A B_\surf' = 0
\]
when $\omega = \omega_n$. Thus, we can write
\[
U_n(r) = S(A_r - \frac{A_\surf'}{B_\surf'}B_r)
= \frac{S}{B_\surf'}(A_r B_\surf' - A_\surf'B_r)
\]
which implies
\[
U_n(\surf) = \frac{S}{B_\surf'} W(A, B),
\]
where $W(A,B)$ is the Wronskian of $A_r, B_r$ and is independent of $r$. We can now compute using \eqref{e: grazing T equation}
\[
\mu^{-1}D(\omega)
= U_n(1)T(1)
= \frac{S^2A_\surf '}{B_\surf'} W(A, B)\left(
1+ \frac{B_\surf'}{A_\surf'}A \right)
\]

Thus, 
\[
\frac{\mu_\surf U(r) U(r_0)}{D}
= \frac{1}{W(A,B)}
\left(\frac{B_\surf'}{A_\surf'}A_r - B_r \right) \left(A_{r_0}- \frac{A_\surf'}{B_\surf'}B_{r_0} \right)
\sum_k\left(-A\frac{B_\surf'}{A_\surf'} \right)^k
\]

Note that even though $B_{s}' = \frac{d}{dr}\restriction_{r = s}B_r $ and similarly for $A_{s}'$, to leading order as $\omega \to \infty$, we have
\[
\frac{B_s'}{A_s'}
= \frac{\mathrm A'_-(\omega^{2/3}\chi^{2/3}(s))}{\mathrm A'_+(\omega^{2/3}\chi^{2/3}(s))}
\]
Next, following the computation in section \ref{a: Generalized Debye}, the quantity $A^k$ above will consist of a sum of terms of the form
\[
R_{++}^{m_0}(T_\pm T_\mp)^{m_1}R_{--}^{m_2}
\left(\frac{A_{\refl^+}}{B_{\refl^+}} \right)^{m_3}\left(\frac{A_{\refl^-}}{B_{\refl^-}} \right)^{m_4}\left(\frac{A'_\CMB}{B'_\CMB} \right)^{m_5}
\]
where $A_{\refl^\pm}$ and $B_{\refl^\pm}$ comes from restricting $U^{(\pm)}$ to the interface $r=\refl$, and
the last term comes from the quantity $C$ determined by the inner boundary condition. This last term is where stationary phase cannot be applied for $p$ near $\CMB/c(\CMB)$ while the other terms will be ``classical'' after using Airy function asymptotics. It will be convenient to use the multiindex $M= (m_0,m_1,m_2,m_3,m_4,m_5) \in \mathbb Z_{\geq 0}^6$.

When computing the trace $\int_{\CMB}^\surf D^{-1}U(r)U(r) \rho dr$,
we need to compute the quantities
\[
l_{-1} = \int_{\CMB}^\surf A_r^2 \rho r^2 \dd r, \qquad
l_0 = \int_{\CMB}^\surf A_rB_r \rho r^2\dd r, \qquad
l_1 = \int_{\CMB}^\surf B_r^2 \rho r^2\dd r
\]
to leading order as $\omega \to \infty$. If these quantities are a symbol in $\omega$, as well as $B_R'/A_R'$ for $p$ near the grazing ray value $R/c(R)$, then we can just apply stationary phase to $(B_1'/A_1')^k$ using the asymptotic expansion of the Airy function as $\omega \to \infty$ by treating the rest of the integrand as the amplitude in the stationary phase calculation.

Thus, using \eqref{e: G with p integral} and the above computations, we get to leading order as $\omega \to \infty$
\begin{equation}
\int \hat G(x,x,\omega) \ dx    
\simeq \sum_{j=-1}^1\sum_{M \in \mathbb Z_{\geq 0}^6}\sum_i \sum_s
V^{(j)}_{isM}(\omega)
\end{equation}
where
\[
V^{(j)}_{isM}
=\omega^2 \int e^{\ii \pi \omega p s}
a^{(j)}_{s,M}(p, \omega)
\left( \frac{B_1'}{A_1'}\right)^{i+j}
\left( \frac{A_R'}{B_R'}\right)^{m_5}
\ \dd p
\]
and
\[
a^{(j)}_{s,M}(p, \omega)
=  \frac{1}{2\pi W(A,B)} (-)^{(s-1)/2}  p^{1/2}Q_M(p)\left(\frac{A_{\refl^+}}{B_{\refl^+}} \right)^{m_3}\left(\frac{A_{\refl^-}}{B_{\refl^-}} \right)^{m_4} l_j
\]
is a symbol of order two, and $Q_M$ is a product of transmission and reflection coefficients described in appendix \ref{a: Generalized Debye}.
Let us write
\begin{equation}\label{e: Visj definition}
V^{(j)}_{isM}
= \omega^2\int b_{ijsM}(p) \left( \frac{A_R'(p)}{B_R'(p)}\right)^{m_5}
\dd p
=\omega^2 \int  \left(\frac{d}{dp}\int_{-\infty}^p b_{ijsM}(y) \dd y\right)    \left(\frac{A_R'(p)}{B_R'(p)}\right)^{m_5}
\dd p, 
\end{equation}
where
\[
 b_{ijsM}(p) :=  e^{\ii \pi \omega p s}
a^{(j)}_{sM}(p, \omega)
\left( \frac{B_1'}{A_1'}\right)^{i+j}.
\]
We integrate by parts to obtain
\begin{multline}
=  \omega^2\left[\int_{-\infty}^pb_{ijsM}(y) \dd y  \left(\frac{A_R'(p)}{B_R'(p)}\right)^{m_5}   \right]_{p = -\infty}^\infty \\
-
\omega^2\int \dd y \int_{-\infty}^p
b_{ijsM}(y)(m_5) \left(\frac{A_R'(p)}{B_R'(p)}\right)^{m_5-1}
\frac{B_R'\frac{d}{dp}A'_R - A_R'\frac{d}{dp}B_R'}{(B_R')^2} \dd p.
\end{multline}
The first term is
\[
\omega^2\int_{-\infty}^\infty b_{ijsM}(y) \dd y  \left(\frac{A_R'(\infty)}{B_R'(\infty)}\right)^{m_5}  = \omega^2\int_{-\infty}^\infty b_{ijsM}(y) \dd y
\]
since $\mathrm A'_+(\infty)/\mathrm A'_-(\infty) = 1$. This is the main term which has a classic form, where we can apply the method of steepest descent argument used in section \ref{s: proof of trace formula}. We just need to verify that the other term is indeed lower order.

After using the Airy equation, the second term becomes
\[
\omega^2\int \dd y \int_{-\infty}^p
b_{ijs}(y)(m_5) \frac{(A_R'(p))^{m_5-1}}{(B_R'(p))^{m_5+1}} W(A,B) (d_p \chi_R^{2/3})\chi_R^{2/3}\omega^{4/3} \dd p
\]
\beq
=\omega^{10/3}(m_5) W(A,B)\int 
\tilde b_{ijsM}(p, \omega) \frac{(A_R'(p))^{m_5-1}}{(B_R'(p))^{m_5+1}} (d_p \chi_R^{2/3})\chi_R^{2/3} \dd p. 
\eeq
where the subscript $R$ on $\chi_R$ means its evaluated at $r = R$ and
\[
\tilde b_{ijsM}(p, \omega)
= \int_{-\infty}^p
b_{ijsM}(y)  \dd y
\]
Our integrand contains terms of the form
\[
A'_{\pm}(\omega^{2/3} \chi_R^{2/3}(p))
\]
so we use the substitution
\[
q = \chi_R^{2/3}(p) , \qquad
\dd q = d_p\chi_R^{2/3}(p) \dd p 
\]
so $p = p(q)$ is a function of $q$ and we get
\beq
=\omega^{10/3}(m_5) W(A,B)\int 
\tilde b_{ijsM}(q, \omega) \frac{(A_+'(\omega^{2/3}q))^{m_5-1}}{(A_-'(\omega^{2/3}q))^{m_5+1}}q \dd q. 
\eeq

Now we substitute 
\[
w = \omega^{2/3} q
\]
to obtain
\beq
=\omega^{2}(m_5) W(A,B)\int 
\tilde b_{ijs}( \omega^{-2/3} w, \omega) \frac{(A_+'(w))^{m_5-1}}{(A_-'(w))^{m_5+1}} w \dd w. 
\eeq
Near the $p$ value $p_g := R/c(R)$ corresponding to a periodic grazing ray is where stationary phase fails. If $p = p_g$, then $q = 0$. Thus, we will do a Taylor series about $w =0$ and we have
\[
\tilde b_{ijsM}( \omega^{-2/3} w, \omega)
= \tilde b_{ijsM}( 0, \omega)
+ \omega^{-2/3}\tilde c_{ijsM}( \omega^{-2/3} w, \omega)
\]
Applying the \cite[proof of Proposition 9]{bennett1982poisson}, the second term is indeed of order $\omega^{-2/3}$ and lower order than the principal term and can be disregarded. In fact, one can continue the Taylor expansion of the second terms and actually obtain lower order terms in the trace formula but we do not pursue this. Thus, taking the principal term gives us
\[
\simeq \omega^{2}(m_5) W(A,B)\int 
\tilde b_{ijs}( 0, \omega) \frac{(A_+'(w))^{m_5-1}}{(A_-'(w))^{m_5+1}} \dd w. 
\]
The analogous computation in \cite[proof of Proposition 9]{bennett1982poisson}, we have
\[
(i+j)\int^\infty_{-\infty} W(A, B)
\frac{(A_+'(w))^{m_5-1}}{(A_-'(w))^{m_5+1}} w\dd w
= \left(\frac{A_R'(\infty)}{B_R'(\infty)}\right)^{m_5} = 1.
\]
We are then left with
\[
V_{ik}^j \eqsim
\omega^2 \int_{-\infty}^\infty
b_{ijsM}(y) \dd y
- \omega^2 \tilde b_{ijsM}(q= 0, \omega)
= \omega^2 \int_{p_g}^{\infty}
b_{ijsM}(y) \dd y .
\]

For the other case where we consider periodic rays with a \branch{} that grazes the interface, we need to do the above analysis for $p$ near $\refl/c(\refl)$. The above argument applies but the quantities $\left( \frac{A_{\refl^+}}{B_{\refl^+}}\right)^{m_3}$ and 
$\left( \frac{A_R'}{B_R'}\right)^{m_5}$ need to be interchanged in \eqref{e: Visj definition} and the rest of the argument below that.

\bibliography{poisson_summation}
\bibliographystyle{plain}
\end{document}